\documentclass[10pt]{amsart}

\usepackage{amssymb,amsmath,amsthm,amscd}

\advance\oddsidemargin by -1cm
\advance\evensidemargin by -1cm
\newtheorem{theorem}{Theorem}

\newtheorem{proposition}{Proposition}
\newtheorem{corollary}{Corollary}
\newtheorem{lemma}{Lemma}

\theoremstyle{definition}
\newtheorem{remark}{Remark}
\newtheorem{example}{Example}

\newcommand{\bdm}{\begin{displaymath}}
\newcommand{\edm}{\end{displaymath}}
\newcommand{\bq}{\begin{equation}}
\newcommand{\eq}{\end{equation}}
\newcommand{\bqn}{\begin{equation*}}
\newcommand{\eqn}{\end{equation*}}

\newcommand{\rn}{\mathbb{R}^n}

\newcommand{\eps}{\varepsilon}

\newcommand{\phw}{\tilde \Phi^{wk}}

\newcommand{\norm}[1]{\left\| #1 \right\|}
\newcommand{\mklm}[1]{\left\{ #1 \right\}}
\newcommand{\eklm}[1]{\left\langle #1 \right\rangle}

\renewcommand{\d}{\,d}
\newcommand{\N}{{\mathbb N}}
\newcommand{\Z}{{\mathbb Z}}
\newcommand{\C}{{\mathbb C}}
\newcommand{\R}{{\mathbb R}}

\newcommand{\D}{{\mathcal D}}
\newcommand{\E}{{\mathcal E}}
\newcommand{\F}{{\mathcal F}}

\renewcommand{\H}{{\mathcal H}}

\newcommand{\M}{{\mathcal M}}

\newcommand{\1}{{\bf 1}}
\renewcommand{\epsilon}{\varepsilon}
\renewcommand{\phi}{\varphi}
\renewcommand{\rho}{\varrho}

\newcommand{\Cinft}{{\rm C^{\infty}}}
\newcommand{\CT}{{\rm C^{\infty}_c}}

\renewcommand{\L}{{\rm L}}

\renewcommand{\S}{{\mathcal S}}

\newcommand{\g}{{\bf \mathfrak g}}
\renewcommand{\k}{{\bf \mathfrak k}}
\newcommand{\p}{{\bf \mathfrak p}}

\newcommand{\Ad}{\mathrm{Ad}\,}
\newcommand{\ad}{\mathrm{ad}\,}

\newcommand{\id}{\mathrm{id}\,}

\renewcommand{\det}{\mathrm{det}\,}

\newcommand{\vol}{\text{vol}\,}

\newcommand{\Crit}{\mathrm{Crit}}
\newcommand{\mult}{\mathrm{mult}}

\DeclareMathOperator{\supp}{supp}

\DeclareMathOperator{\tr}{tr}
\DeclareMathOperator{\gd}{\partial}

\newcommand{\e}[1]{\,{\mathrm e}^{#1}\,}
\newcommand{\dbar}{{\,\raisebox{-.1ex}{\={}}\!\!\!\!d}}

\begin{document}

\author{Pablo Ramacher}
\title[Singular equivariant asymptotics and Weyl's law]{Singular equivariant asymptotics and Weyl's law. \\ On the distribution of  eigenvalues of an \\ invariant elliptic operator} 
\address{Pablo Ramacher, Philipps-Universit\"at Marburg, Fachbereich Mathematik und Informatik, Hans-Meer\-wein-Str., 35032 Marburg, Germany}
\email{ramacher@mathematik.uni-marburg.de}
\date{August 11, 2011}

\begin{abstract}
We study the spectrum of an invariant, elliptic, classical pseudodifferential  operator on a closed $G$-manifold $M$,  where $G$ is a compact, connected Lie group acting effectively and  isometrically on $M$. Using resolution of singularities, we  determine the asymptotic distribution of  eigenvalues along the isotypic components, and  relate it with the reduction of  the corresponding Hamiltonian flow, proving that the reduced spectral counting function satisfies Weyl's law, together with an estimate for the remainder.
\end{abstract}

\maketitle

\setcounter{tocdepth}{1}
\tableofcontents

\section{Introduction}

The asymptotic distribution of eigenvalues of an elliptic operator has been object of mathematical research  for a long time. It was first studied by Weyl \cite{weyl} for certain second order differential operators  in Euclidean space using variational techniques, followed by work of Carleman \cite{carleman},  Minakshishundaram and  Pleijel \cite{minakshisundaram-pleijel}, G{\aa}rding \cite{garding}, and Avacumovi\v{c} \cite{avacumovic}.
Later, H\"ormander \cite{hoermander68} and Duistermaat-Guillemin \cite{duistermaat-guillemin75} extended these results to elliptic pseudodifferential operators on compact manifolds within  the theory of Fourier integral operators.  In this paper, we shall consider this problem in the case that additional symmetries are present. 
 
Let $M$ be a compact, connected, $n$-dimensional Riemannian manifold without boundary, $dM$ its volume density, and 
\bqn
P_0: \Cinft(M) \longrightarrow \L^2(M)
\eqn
 an elliptic, classical  pseudodifferential operator of order $m$ on $M$, regarded as an operator in the Hilbert space $\L^2(M)$ of square integrable functions on $M$ with respect to $dM$, its domain being the space $\Cinft(M)$ of smooth functions on $M$. Assume that $P_0$ is positive and symmetric, which implies that  $P_0$ has a  unique self-adjoint extension  $P$. Due to the compactness of $M$, the spectrum of $P$ is discrete. Consider now in addition  a compact, connected Lie group $G$, acting effectively and isometrically on $M$, and assume that $P$ commutes with the regular representation of $G$ in $\L^2(M)$. In this situation, each eigenspace of $P$ becomes a unitary $G$-module, and it is a natural question to ask about the distribution of the spectrum of $P$ along the isotypic components of $\L^2(M)$ in  the decomposition
\bqn
\L^2(M)\simeq \bigoplus_{\chi \in \hat G} \L^2(M)(\chi), 
\eqn
and the way it is related to the reduction of the corresponding Hamiltonian flow. It  is described by  the reduced spectral counting function $
N_\chi(\lambda)=d_\chi \sum _{t \leq \lambda} \mult_\chi(t)$,
where $\mult_\chi(t)$ denotes the multiplicity of the unitary irreducible representation $\pi_\chi$ corresponding to the character $\chi \in \hat G$ in the eigenspace $E_t$ of $P$ belonging to the  eigenvalue $t$. Let $T^\ast M$ be the  cotangent bundle of $M$, $p(x,\xi)$ the principal symbol of $P_0$, and $S^\ast M=\{(x,\xi) \in T^\ast M: p(x,\xi) = 1\}$. In his classical paper  \cite{hoermander68}, H\"ormander showed that the spectral counting function
$
N(\lambda)=\sum_{t\leq \lambda} \dim E_t
$
satisfies Weyl's law
\bq
\label{eq:weyl}
N(\lambda)=\frac {\vol S^\ast M}{n(2\pi)^n} \lambda^{n/m} +O(\lambda^{{(n-1)}{/m}}), \qquad \lambda \to +\infty,
\eq
 and it has been a long-standing open question whether  a similar description for $N_\chi(\lambda)$ can be achieved.   While in the general case of effective group actions  the leading term  was obtained  via heat kernel methods by Donnelly \cite{donnelly78} and Br\"uning--Heintze \cite{bruening-heintze79}, estimates for the remainder are  not accessible via this approach. On the other hand, the derivation of remainder estimates within the framework of Fourier integral operators meets with serious difficulties when singular orbits are present, and  until recently  could only  be obtained for finite group actions, or actions with orbits of the same dimension as in the work of  Donnelly \cite{donnelly78}, Br\"uning--Heintze \cite{bruening-heintze79}, Br\"uning \cite{bruening83}, Helffer--Robert \cite{helffer-robert84, helffer-robert86}, Guillemin--Uribe \cite{guillemin-uribe90}, and El-Houakmi--Helffer \cite{helffer-elhouakmi91}. It was only in  Ramacher \cite{ramacher08} and Cassanas--Ramacher \cite{cassanas-ramacher09} that first partial results towards more general group actions were obtained  within the setting of approximate spectral  projections using resolution of singularities. The goal of  this paper is to   generalize this  approach, and 
give an asymptotic description of $N_\chi(\lambda)$ analogous to   \eqref{eq:weyl}  for general effective group actions  within the theory of Duistermaat, Guillemin, and H\"ormander. 
  
In order to  explain the difficulties in a more detailed way, denote by $Q=(P)^{1/m}$ the $m$-th root of $P$ given by the spectral theorem, which is a classical pseudodifferential operator of order $1$ with principal symbol $q(x,\xi)=p(x,\xi)^{1/m}$. If $0<\lambda_1 \leq \lambda_2 \leq \dots$ are  the eigenvalues of $P$ repeated according to their multiplicity, the eigenvalues of $Q$ are $\mu_j =(\lambda_j)^{1/m}$. Let $\{dE^Q_\mu\}$ be the spectral resolution of $Q$. The starting point of the method of Fourier integral operators is the Fourier transform of the spectral measure
 \bqn 
 U(t)= \int e^{-it\mu} dE^Q_\mu = e^{-itQ}, \qquad t \in \R,
 \eqn
 which constitutes a one-parameter group of unitary operators in $\L^2(M)$. Although $U(t)$ itself is not trace-class, it has a distribution trace given by  the tempered distribution
\bqn 
\tr U(\cdot ): \S(\R)  \ni \rho \longmapsto \int \sum _{j=1}^\infty e^{-it \mu_j}  \rho(t)dt =\sum _{j=1}^\infty \hat \rho(\mu_j) < \infty,
\eqn
which is the Fourier transform of the spectral distribution
\bqn 
\sigma(\mu) = \sum_{j=1}^\infty \delta(\mu-\mu_j).
\eqn
An asymptotic description of the spectrum of $P$ is then attained by studying  the singularities of the distribution kernel of  $U(t)$ and of  $\tr U(\cdot )$  for small $|t|$.  To be more precise, let  $\Omega_{1/2}$ denote the bundle of half-densities over $M$, and  $U_{1/2}$  the operator which assigns to $u_0 \in \Cinft(M,\Omega_{1/2})$ the solution $u\in \Cinft(\R\times M , \Omega_{1/2})$ of the Cauchy problem 
 \bqn
 \big (i^{-1} \gd_t +Q_{1/2}\big )u=0, \qquad u(0,x)=u_0(x),
 \eqn
 where $Q_{1/2}u = \d M^{1/2} \, Q(u \d M^{-1/2})$. 
Then $U_{1/2}:  \Cinft(M , \Omega_{1/2})\rightarrow \Cinft(\R\times M , \Omega_{1/2})$ can be characterized  globally as a Fourier integral operator with kernel  ${\mathcal{U}} \in I^{-1/4}(\R \times M, M; C')$ and canonical relation
\begin{align*} 
C=&\big \{((t,\tau), (x,\xi), (y,\eta)): (x,\xi), (y,\eta) \in T^\ast M \setminus 0, (t,\tau) \in T^\ast \R\setminus 0,  \\  & \tau + q(x,\xi)=0, \, (x,\xi) = \Phi^t(y,\eta)\big \},
\end{align*}
where $\Phi^t$  is the flow in $T^\ast M\setminus 0$ of the Hamiltonian vector field associated to $q$, and $C'=\mklm{((t,\tau), (x,\xi), (y,-\eta)):((t,\tau), (x,\xi), (y,\eta)) \in C} $ \cite{duistermaat-guillemin75}. This implies that  $\hat \sigma$ is a Fourier integral operator as well, and the study of its singularity at $t=0$   leads to the main result
\bqn 
\hat \sigma \big ( \check \rho e^{i(\cdot) \mu}\big )= \sum_{j=1}^\infty \hat \rho (\mu- \mu_j) \sim (2\pi)^{1-n}  \sum_{k=0}^\infty c_k \mu^{n-1-k}, \qquad \mu \to +\infty,
\eqn
for  suitable  $\rho \in \S(\R)$,  where  $\check \rho(t)=\rho(-t)$, with in principle known coefficients $c_k$. For $\mu \to -\infty$, the above expression is rapidly decreasing. From this, \eqref{eq:weyl} follows using a Tauberian theorem.
To obtain a similar description of $N_\chi(\lambda)$,  one needs an asymptotic expansion of  the sum $\sum_{j=1}^\infty m^Q_\chi(\mu_j) \hat \rho (\mu-\mu_j)$ for suitable $\rho \in S(\R)$, where  $m^Q_\chi(\mu_j)=d_\chi {\mult^Q_\chi(\mu_j)}/{\dim E^Q_{\mu_j}}$, ${\mult^Q_\chi(\mu_j)}$ being the multiplicity of the  irreducible representation $\pi_\chi$ in the eigenspace $E^Q_{\mu_j}$ of $Q$ belonging to the eigenvalue $\mu_j$. In this way, we are led to study the singularities of the distribution trace of $P_\chi \circ U(t)$,  where $P_\chi$ denotes the projector onto the $\chi$-isotypic component $\L^2(M)(\chi)$. This trace   is the Fourier transform of 
\bqn 
\sigma_\chi(\mu)=\sum _{j=1}^\infty m^Q_\chi(\mu_j)\, \delta(\mu- \mu_{j}),
\eqn
and it turns out that, when regarding $\hat \sigma_\chi$ as a distribution density on $\R$ of order $1/2$,   $\hat \sigma_\chi= d_\chi \pi_\ast \, \bar \chi \, \Gamma^\ast \, {\mathcal{U}}$, 
where $\pi:\R \times G \times M \rightarrow \R$ is the projection $(t,g,x) \mapsto t$, and $ \Gamma: \R\times G \times M \rightarrow \R \times M \times M$ the mapping $(t,g,x) \mapsto (t,x,gx)$. Both the pushforward $\pi_\ast$ and the pullback $\Gamma^\ast$ 
can be characterized as  Fourier integral operators, but in general,   neither their composition  $\pi_\ast \,  \bar \chi \, \Gamma^\ast$ nor $\hat \sigma_\chi$ have smooth wavefront sets. Indeed, as pointed out in \cite{donnelly78},
\begin{align*}
\mathrm{WF}( \hat \sigma_\chi)=& \{ (t,\tau): \text{ there exist } x,\eta, g \text{ such that } (x,\eta) \in \Omega,  \\ &(x, -g^\ast \eta) =\Phi^t(gx,\eta), 
  \quad \tau +  q(x, -g^\ast \eta) =0\},
\end{align*} 
where $\Omega= \mathbb{J}^{-1}(0)$ denotes the zero level of the canonical symplectic momentum map
 $\mathbb{J}:T^\ast M\rightarrow \g^\ast$. If the underlying group action is not free, $\mathbb{J}$ is no longer a submersion, so that $\Omega$ is not a smooth manifold.  Therefore $\hat \sigma_\chi$  fails to be  a Fourier integral operator in general, so that, a priori,  it is not clear how to describe its singularities by the method of Duistermaat, Guillemin, and H\"ormander.  Ultimately, the difficulties arise from the necessity  to understand the asymptotic behavior of oscillatory integrals of the form
\begin{align*}
I(\mu)&= \int _{T^\ast Y}  \int_{G} e^{i\mu  \Phi(x , \xi,g) }   a( g  x,  x , \xi,g)  \d g \d(T^\ast Y)( x , \xi),  \qquad \mu \to +\infty,  
\end{align*}
via the stationary phase theorem, where $(\kappa,Y)$ are local coordinates on $M$,  $ \d(T^\ast Y)( x , \xi)$ is the canonical volume  density   on $T^\ast Y$, and $dg$ the volume density on $G$ with respect to some left invariant Riemannian metric on $G$, while $a \in \CT(Y \times T^\ast  Y\times G)$ is an amplitude which might also depend on $\mu$, and $\Phi(x, \xi, g) =\eklm{\kappa(x) - \kappa (g x), \xi }$.  For this, it would be necessary that  the critical set of the phase function $\Phi(x,\xi,g)$
 \begin{align*}
 \Crit(\Phi)&=\mklm{( x  ,\xi, g) \in (\Omega  \cap T^\ast Y)\times G:  \,  g \cdot (x,\xi )=(x, \xi)}
\end{align*}
 were a smooth manifold, which, nevertheless, is only true for free group actions. In  the case of general effective actions  the  stationary phase theorem can therefore not immediately be applied to the study of  integrals of the type $I(\mu)$, compare \cite{bruening83}. 
 
  In this paper,  we shall show how to  overcome this obstacle by partially resolving the singularities of  
$
\mathcal{C}= \mklm{(x,\xi,g) \in \Omega \times G: g\cdot (x,\xi) = (x,\xi)}  
 $,
and applying  the stationary phase principle in a suitable resolution space. This will be achieved by constructing a resolution
of the set 
\bqn 
\mathcal{N}=\mklm{(x,g) \in \M: gx =x}, \qquad \M= M \times G,
\eqn
  which is equivalent to a monomialization of its ideal sheaf $I_\mathcal{N}\subset \E_\M$, where $\E_\M$ denotes the structure sheaf of $\M$.  To be more precise, put $X=T^\ast M \times G$, and let $I_\mathcal{C}\subset \E_X$ be  the ideal sheaf of $\mathcal{C}$.
Consider further the local ideal $I_\Phi=(\Phi)$ generated by the phase function $\Phi$, together with its vanishing set $V_\Phi$. The derivative of $I_\Phi$ is given by $D(I_\Phi) = I_{\mathcal{C}|T^\ast Y\times G}$, 
and $ \Crit(\Phi)\subset V_\Phi$.
The main idea is  to construct a  resolution of $\mathcal{N}$, yielding a partial resolution $\mathcal{Z}:\tilde X \rightarrow X$ of $\mathcal{C}$, and a partial
monomialization of $I_\Phi$ according to 
\bqn 
\mathcal{Z}^\ast (I_\Phi) \cdot \E_{\tilde x, \tilde X} = \prod_j\sigma_j^{l_j}   \cdot\mathcal{Z}^{-1}_\ast(I_\Phi) \cdot \E_{\tilde x, \tilde X}, \qquad \tilde x \in \tilde X,
\eqn
 in such a way that $D(\mathcal{Z}^{-1}_\ast(I_\Phi))$ is a resolved ideal sheaf. Here    $  \mathcal{Z}^\ast (I_\Phi)$ denotes the inverse image ideal sheaf, $\mathcal{Z}^{-1}_\ast(I_\Phi)$  the weak transform of $I_\Phi$,while the  $\sigma_j$ are local coordinate functions, and $l_j$ are natural numbers.  
 As a consequence, the phase function factorizes locally according to $\Phi \circ \mathcal{Z} \equiv \prod \sigma_j^{l_j} \cdot  \tilde \Phi^ {wk}$,
and we show  that  the weak transforms $ \tilde \Phi^ {wk}$ have clean critical sets in the sense of Bott \cite{bott56}. An asymptotic description of the integrals $I(\mu)$ can then be obtained by pulling  them back to the resolution space $\tilde X$, and applying  the stationary phase theorem to the weak transforms $\tilde \Phi^{wk}$  with the variables  $\sigma_j$ as parameters. The desingularization of $\mathcal{N}$  will rely on the stratification of $M$ into orbit types, and consist  of a series of monoidal transformations over $\M$ where the centers are successively chosen as isotropy bundles over unions of maximally singular orbits. 

The main result of the present paper is formulated in Theorem \ref{thm:main}. It states that  the reduced spectral counting function satisfies Weyl's law
\bqn 
N_\chi(\lambda)= \frac{d_\chi [{\pi_\chi}_{|H}:1]}{(n-\kappa)(2\pi)^{n-\kappa}}   \mathrm{vol} \, [(\Omega \cap S^\ast M)/G]  \,  {\lambda} ^{\frac{n-\kappa}m }   + O\big (\lambda^{{(n-\kappa-1)}/m} (\log \lambda)^{\Lambda} \big ),\qquad \lambda \to +\infty,
\eqn
provided that $n-\kappa\geq 1$, where   $\kappa$ is the dimension of a $G$-orbit of principal type, $d_\chi$  the dimension of the irreducible representation $\pi_\chi$,  $ [{\pi_\chi}_{|H}:\1] $   the multiplicity of the trivial representation in the restriction of $\pi_\chi$ to a principal isotropy group $H$, and $\Lambda$ a natural number which is bounded by the number of orbit types of the $G$-action on $M$. The paper itself is structured as follows. Section \ref{sec:FIO} describes the theory of Duistermaat, Guillemin and H\"ormander of spectral asymptotics in the equivariant setting, and 
 explains how the problem of determining $N_\chi(\lambda)$ reduces to the study of integrals of the  type $I(\mu)$ as $\mu \to +\infty$. Section \ref{sec:CGMM} contains some general remarks on compact group actions and the momentum map, followed by the computation of the critical set of the phase function $\Phi$. Singular asymptotics are discussed in Section  \ref{sec:SPRS}, after a  brief account on the  stationary phase principle and resolution of singularities. 
 In Section \ref{sec:DP}, the desingularization process is carried out, giving way  in Sections \ref{sec:MT1} and \ref{sec:MT2} to the phase analysis of the weak transforms $\tilde \Phi^{wk}$.  Asymptotics for integrals of the type $I(\mu)$ are then obtained in Section \ref{sec:INT}, while the proof of the main result is given in Section \ref{sec:MR}.

Singular equivariant asymptotics with reminder estimates were previously  obtained   by Br\"uning-Heintze \cite{bruening-heintze79} and Duistermaat-Kolk-Varadarajan \cite{DKV1} for   the spectrum of   a discrete,  uniform subgroup $\Gamma$ of a connected, semisimple Lie group $G$ with maximal compact subgroup $K$. In the first case, a reminder estimate for the Gelfand-Gangolli-Wallach formula is given,  which describes the distribution of eigenvalues of the Casimir operator along the isotypic components of $\L^2(\Gamma\setminus G)$. For torsion-free $\Gamma$, this corresponds to  the distribution of eigenvalues of the Bochner-Laplace operator on the spaces $\L^2(\Gamma\setminus G /\,  K, E^\chi)$, where $E^\chi$ denotes the vector bundle on $\Gamma\setminus G /\,  K$ induced by an  arbitrary $\chi\in \hat K$. 
 In the second case, and under the assumption that $\Gamma$ has no torsion, asymptotics for the spectral counting function of  the Laplace-Beltrami operator $\Delta$ on $\L^2(\Gamma\setminus G /\,  K)\simeq \L^2(\Gamma\setminus G)^K$ are derived. This  amounts to an asymptotic description of $N_\chi(\lambda)$ for $\Delta$ on $\L^2(\Gamma \setminus G)$ in case that $\chi$ corresponds to  the trivial representation, and  Theorem \ref{thm:main}  generalizes this result to arbitrary $\chi \in \hat K$, and subgroups $\Gamma$ with torsion, as well as arbitrary invariant, elliptic, classical pseudodifferential operators. This is explained in Section \ref{sec:G}.
 
\medskip

{\bf{Acknowledgments.}} The author wishes to express his gratitude to  Mikhail Shubin for introducing him to this subject, and to his former collaborator Roch Cassanas. He also would like to thank Richard Melrose, Werner M\"uller and Mich\`{e}le Vergne for valuable  conversations. This research was completed while the author was a member of the Mathematical Institute of G\"{o}ttingen University, and financed by the grant RA 1370/2-1 of the German Research Foundation (DFG).

\section{Fourier integral operators and equivariant asymptotics}
\label{sec:FIO}

\subsection*{Generalities}

 Let $M$ be a compact, connected, $n$-dimensional Riemannian manifold,  and $G$ a compact, connected Lie group of dimension $d$, acting effectively and isometrically on $M$. Denote the canonical volume density on $M$ by $dM$   \cite{sternberg}, page 112, and choose a left invariant Riemannian metric on $G$ with volume density $dg$.
  Let  $P_0$ be an elliptic,  classical pseudodifferential operator of order $m$ on $M$, regarded as an operator in $\L^2(M)$ with domain $\Cinft(M)$, and assume that $P_0$ is positive and symmetric. Then $P_0$ has a  unique self-adjoint extension $P$ with the $m$-th Sobolev space $H^m(M)$ as domain. Moreover, the spectrum of $P$ is  discrete. Assume now that $P$ commutes with the regular representation of $G$ in $\L^2(M)$ given by 
 \bqn
 T(g) \phi(x) = \phi(g^{-1} x), \qquad g \in G.
 \eqn
 Then every eigenspace of $P$ becomes a unitary $G$-module, and it is natural to  ask about the distribution of the spectrum of $P$ along the isotypic components of $\L^2(M)$, which is described by  the reduced spectral counting function $N_\chi(\lambda)$ introduced in the previous section.  We shall study this problem within the theory of Fourier integral operators developed by H\"ormander, Duistermaat and Guillemin \cite{hoermander68, duistermaat-guillemin75}, and consider for this  the $m$-th root $Q=(P)^{1/m}$
 of $P$ given by the spectral theorem. By Seeley, $Q$ is a classical pseudodifferential operator of order $1$ with principal symbol $q(x,\xi)=p(x,\xi)^{1/m}$ and  domain  $H^1(M)$. If $0<\lambda_1 \leq \lambda_2 \leq \dots$ are  the eigenvalues of $P$ repeated according to their multiplicity, the eigenvalues of $Q$ are $\mu_j =(\lambda_j)^{1/m}$. Denote by $\{dE^Q_\mu\}$ the spectral resolution of $Q$. The starting point of the method developed by H\"ormander, and which goes back to work of Avacumovi\v{c} and Lewitan, is the Fourier transform of the spectral measure
 \bqn 
 U(t)= \int e^{-it\mu} dE^Q_\mu = e^{-itQ}, \qquad t \in \R,
 \eqn
 which constitutes a one-parameter group of unitary operators in $\L^2(M)$. Now, if  $\mklm{e_j}$ denotes an orthonormal basis of eigenfunctions in $\L^2(M)$ of $Q$ corresponding to the eigenvalues $\mklm{\mu_j}$, then
  \bq
  \label{eq:vale}
U(t) u =\sum _{j=1}^\infty e^{-it \mu_j}  (u,e_j)_{\L^2} \, e_j,
\eq
where $u \in \Cinft(M)$, and   $u \in H^s(M)$, $ s\in \Z$, 
 the sum converging in the $\Cinft$-, and $H^s$-topology, respectively, see  \cite{shubin}, page 151. Thus, the distribution kernel of the operator $U(t):\Cinft(M) \rightarrow \Cinft(M)\subset \D'(M)$ can be written as 
\bqn 
U(t,x,y)=\sum _{j=1}^\infty e^{-it \mu_j}  e_j(x) \, \overline{e_j(y)} \, \in \D'(M\times M). 
\eqn
Although $U(t)$ itself is not trace-class, it has a distribution trace given by  the tempered distribution
\bqn 
\tr U(\cdot): \S(\R)  \ni \rho \longmapsto \int \sum _{j=1}^\infty e^{-it \mu_j}  \rho(t)dt =\sum _{j=1}^\infty \hat \rho(\mu_j) < \infty.
\eqn
Indeed, for $N_0 \in \N$, $P^{-N_0}$ is a classical pseudodifferential operator of order $-N_0m$. If $N_0m>n$, its kernel is continuous, and $P^{-N_0}$ is Hilbert-Schmidt, so that $\sum_{j=1}^\infty \lambda_{j}^{-2N_0}<\infty$.
Moreover,  for  $\rho \in \S(\R)$ the infinite sum $\sum _{j=1}^\infty \hat \rho(\mu_j)  e_j(x) \, \overline{e_j(y)}$ converges in $\Cinft(M\times M)$, see \cite{grigis-sjoestrand}, page 133. Because the Fourier transform is an isomorphism in $\S(\R)$, we conclude that $\tr U(t)= \sum _{j=1}^\infty e^{-it \mu_j}$ is the Fourier transform of the spectral distribution
\bqn 
\sigma(\mu) = \sum_{j=1}^\infty \delta(\mu-\mu_j),
\eqn
proving at the same time that $\sigma$ is tempered. An asymptotic description of the spectrum of $P$ is then attained by studying  the singularities of  $U(t,x,y)$ and $\tr U(\cdot)$  for small $|t|$. For this, H\"ormander locally approximated the operator $U(t)$  by  Fourier integral operators, which solve the Cauchy problem approximately. More precisely, let $U_{1/2}$ be the operator which assigns to $u_0 \in \Cinft(M,\Omega_{1/2})$ the solution $u \in \Cinft(\R\times M , \Omega_{1/2})$ of the hyperbolic Cauchy problem
 \bqn
 \big (i^{-1} \gd_t +Q_{1/2}\big )u=0, \qquad u(0,x)=u_0(x),
 \eqn
 where $\Omega_{1/2}$ denotes the bundle of half-densities over $M$, and $Q_{1/2} u = \d M^{1/2} Q( u \d M^{-1/2})$.  It can then be shown \cite{duistermaat-guillemin75}, Theorem 1.1,  that   $U_{1/2}:  \Cinft(M , \Omega_{1/2})\rightarrow \Cinft(\R\times M , \Omega_{1/2})$ can be characterized  globally as a Fourier integral operator with kernel $\mathcal{U} \in I^{-1/4}(\R \times M, M, C')$ and  canonical relation
\begin{align}
\label{eq:canrel}
\begin{split} 
C=&\big \{((t,\tau), (x,\xi), (y,\eta)): (x,\xi), (y,\eta) \in T^\ast M \setminus 0, (t,\tau) \in T^\ast \R\setminus 0,  \\  & \tau + q(x,\xi)=0, \, (x,\xi) = \Phi^t(y,\eta)\big \},
\end{split}
\end{align}
where $\Phi^t$  is the flow in $T^\ast M\setminus 0$ of the Hamiltonian vector field associated to $q$. This implies that the Fourier transform $U_{1/2}(t):  \Cinft(M , \Omega_{1/2})\rightarrow \Cinft(M , \Omega_{1/2})$  of the spectral measure of $Q_{1/2}$  is a Fourier integral operator of order $0$ defined by the canonical transformation $\Phi^t$, and  that $\hat \sigma$ can be characterized as a Fourier integral operator too, see \cite{duistermaat-guillemin75}, pp. 66. Moreover,
\bqn 
\mathrm{sing} \supp U_{1/2} =\mklm{(t,x,y) \in \R \times M \times M: (x,\xi) =\Phi^t(y,\eta) \text{ for suitable } \xi \in T^\ast_xM\setminus 0, \, \eta \in T^\ast_yM\setminus 0},
\eqn 
and similarly, $WF(\hat \sigma) \subset \mklm{(t,\tau): \tau <0 \text{ and } (x,\xi)=\Phi^t(x,\xi) \text{ for some } (x,\xi)}$, 
so that $\hat \sigma$ is smooth on the complement of the set of periodic orbits. The study of the singularity of $\hat \sigma =\tr U$ at $t=0$ then leads to the main result of H\"ormander
\bqn 
\hat \sigma \big ( \check \rho e^{i(\cdot) \mu}\big )= \sum_{j=1}^\infty \hat \rho (\mu- \mu_j) \sim (2\pi)^{1-n}  \sum_{k=0}^\infty c_k \mu^{n-1-k}, \qquad \mu \to +\infty,
\eqn
for suitable $\rho \in \S(\R)$, $\check \rho(t) =\rho(-t)$, and with in principle known coefficients $c_k$, while for   $\mu \to -\infty$  the expression  is rapidly decreasing. From this,  Weyl's classical law \eqref{eq:weyl} follows by a  Tauberian theorem. 

Let us now come back to our initial question. To obtain a description of $N_\chi(\lambda)$,  and to understand the way it is related to the reduction of the corresponding Hamiltonian flow \cite{guillemin-uribe90}, we would like to find an asymptotic expansion of 
$\sum_{j=1}^\infty m^Q_\chi(\mu_j) \hat \rho (\mu-\mu_j)$
 for suitable $\rho \in S(\R)$, where  $m^Q_\chi(\mu_j)=d_\chi {\mult^Q_\chi(\mu_j)}/{\dim E^Q_{\mu_j}}$. This  amounts to study the singularities of 
$
\sum _{j=1}^\infty m^Q_\chi(\mu_j)\,  e^{-it\mu_{j}} \in \S'(\R).
$
It corresponds to the distribution trace of $P_\chi \circ U(t)$,  $P_\chi$ being the projector onto the $\chi$-isotypic component $\L^2(M)(\chi)$, and is the Fourier transform of 
\bqn 
\sigma_\chi(\mu)=\sum _{j=1}^\infty m^Q_\chi(\mu_j)\, \delta(\mu- \mu_{j}).
\eqn
In what follows, denote by $\pi:\R \times G \times M \rightarrow \R$ the projection $(t,g,x) \mapsto t$, and by $ \Gamma: \R\times G \times M \rightarrow \R \times M \times M$ the mapping $(t,g,x) \mapsto (t,x,gx)$. The global theory of Fourier integral operators \cite{donnelly78}, Lemma 7.1,  implies that  the transposed of the pullback, or 
pushforward $\pi_\ast:\D'(\R\times G \times M)\rightarrow \D'(\R)$, can be characterized as a Fourier integral operator  of class $\mathrm{I}^{-n/4-d/4}(\R, \R \times G \times M, C_1)$ with canonical relation 
\bqn 
C_1=\mklm{\big (t,\tau);(t,\tau),(g,0),(x,0)\big )}.
\eqn
It amounts to integration over $M\times G$.
Similarly, the pullback $\Gamma^\ast:\Cinft (\R \times M \times M)\rightarrow \Cinft(\R\times G \times M)$ constitutes a Fourier integral operator of class $\mathrm{I}^{n/4-d/4}(\R\times G\times M, \R \times M \times M, C_2)$ with canonical relation 
\bqn 
C_2=\mklm{\big (t,\tau),(g,x^\ast \xi_2),(x,\xi_1+g^\ast \xi_2);(t,\tau),(x,\xi_1),(gx,\xi_2)\big )},
\eqn
where the map $x:G\rightarrow M$ is given by $g \mapsto gx$, and the map $g:M \rightarrow M$ by $x \mapsto gx$. 
A computation then shows that if we regard $\hat \sigma_\chi$ as a distribution density on $\R$ of order $1/2$, and $\pi_\ast$ and $\Gamma^\ast$ as maps between half densitites,
\bqn 
\hat \sigma_\chi= \sum _{j=1} ^\infty m^Q_\chi(\mu_j) \, e ^{-i(\cdot) \mu_j} =d_\chi \pi_\ast \, \bar \chi \, \Gamma^\ast \, \mathcal{U},
\eqn
compare \cite{duistermaat-guillemin75}, page 66, and \cite{donnelly78}, Section 7. Now, although  $\pi_\ast $, $\bar \chi \, \Gamma^\ast$, and $U_{1/2}$ are Fourier integral operators, their composition is not necessarily a Fourier integral operator. Indeed, the composition of the canonical relations of $\pi_\ast$ and $\Gamma^\ast$ reads
\begin{align*} 
C_1 \circ C_2 &= \mklm{ \big ((t,\tau);(t,\tau), (x,\xi_1),(gx, \xi_2) \big ): x^\ast \xi_2=0, \,  \xi_1+g^\ast \xi_2=0}. 
\end{align*}
But $x^\ast \xi_2=0$ means that $\xi_2 \in \mathrm{Ann} \, T_x (G\cdot x)$. As will be explained in the next section, this is equivalent to $(x,\xi_2) \in \Omega= \mathbb{J}^{-1}(0)$, where
$
\mathbb{J}:T^\ast M\rightarrow \g^\ast
$
 is the canonical symplectic momentum map, and we obtain
\begin{align*}
C_1 \circ C_2  \circ C=& \{ (t,\tau): \text{ there exist } x,\eta, g \text{ such that } (x,\eta) \in \Omega,  \\ &(x, -g^\ast \eta) =\Phi^t(gx,\eta), 
  \quad \tau +  q(x,-g^\ast \eta) =0\}.
\end{align*}
The singularities of $\hat \sigma_\chi$ are therefore determined by the restriction of $\Phi^t$ to  $\Omega$. Since for general effective group actions, the zero level $\Omega$ is not smooth, neither $C_1 \circ C_2$ nor $C_1 \circ C_2  \circ C$ are smooth submanifolds in this case.   Consequently,   neither $\pi_\ast \, \bar \chi \, \Gamma^\ast$ nor  $\hat \sigma_\chi$ are Fourier integral operators in general. This faces us with serious difficulties when trying  to study the singularities of $\hat \sigma_\chi$ within the theory of H\"ormander, Duistermaat, and Guillemin.

\subsection*{A trace formula}

In what follows, we would like to understand the main singularity of $\hat \sigma_\chi$ at $t=0$ in greater detail. To this end, we shall first express 
\bqn 
\hat \sigma_\chi (\check \rho e^{i(\cdot) \mu})=\sum_{j=1}^\infty m^Q_\chi(\mu_j) \hat \rho (\mu-\mu_j), \qquad \rho \in \S(\R),
\eqn 
as the $\L^2$-trace of a certain operator. Observe that  $\sum_{j=1}^\infty \hat \rho(\mu_j) e_j(x) \overline{e_j(y)} \in \Cinft(M\times M)$
is the Schwartz kernel of the bounded operator
\bqn 
\int_{-\infty}^{+\infty} \rho(t) U(t) dt: \L^2(M) \longrightarrow \L^2(M),
\eqn
which is defined as a Bochner integral. It is  of $\L^2$-trace class, since its kernel is square integrable over $M \times M$. Therefore
\bq
\label{op}
P_\chi \circ \int_{-\infty}^{+\infty} \rho(t) U(t) dt = \int_{-\infty}^{+\infty} \rho(t) P_\chi \circ U(t) dt
\eq 
must be of trace class, too, where 
\bqn 
P_\chi= d_\chi \int_G \overline{\chi(g)} T(g) \d g
\eqn 
denotes the projector onto the isotpyic component $\L^2(M)(\chi)$, and $d_\chi$ the dimension of the irreducible representation corresponding to the character $\chi\in \hat G$. We  assert that   kernel of the operator \eqref{op} is given by $\sum_{j=1}^\infty \hat \rho(\mu_j) P_\chi e_j(x) \overline{e_j(y)} \in \Cinft(M\times M)$. Indeed, by choosing the eigenfunctions $\mklm{e_j}$ according to the decomposition of the eigenspaces of $Q$ into isotypic components, we can assume that $P_\chi e_j= 0$ if $e_j \notin \L^2(M)(\chi)$, and $P_\chi e_j =e_j$ otherwise. By Sobolev's inequality we have 
\bqn 
\norm{P_\chi e_j}_{C^k} \leq \norm{e_j}_{C^k} \leq c' \norm{ e_j}_{H^{k+n+1}} \leq c'' \norm{ Q^{k+n+1} e_j}_{L^{2}} \leq c'' \mu_j^{k+n+1},
\eqn
showing that $\sum_j \hat \rho(\mu_j) P_\chi e_j(x) \overline{e_j(y)}$ converges in $ \Cinft(M\times M)$, and with \eqref{eq:vale}  one computes
\begin{align*}
 \int_{-\infty}^{+\infty} \rho(t) P_\chi \circ U(t) dt \, u(x)&=  \int_{-\infty}^{+\infty} \rho(t) \sum_{j=1}^\infty e^{-it\mu_j} (u,e_j)_{\L^2} P_\chi e_j(x) dt \\
 &= \int_M \,  u(y) \sum_{j=1}^\infty   \int_{-\infty}^{+\infty}  \rho(t)  e^{-it\mu_j} P_\chi e_j(x) \overline{e_j(y)} \d t \d M(y), \qquad  u \in \L^2(M),
\end{align*}
everything being absolutely convergent. As a consequence,
\bqn 
\tr \int_{-\infty}^{+\infty} \rho(t) P_\chi \circ U(t) dt =\sum _{j=1}^\infty \hat \rho (\mu_j) (P_\chi e_j, e_j)_{L^2} = \sum _{j=1}^\infty \hat \rho (\mu_j) m^Q_\chi(\mu_j) = \hat \sigma_\chi(\rho), 
\eqn
and we obtain the following $\L^2$-trace formula, which was already derived in \cite{bruening83}.

\begin{lemma}
Let $\rho \in \S(\R)$. Then 
\bq
\hat \sigma_\chi (\rho e^{i(\cdot) \mu})= \tr \int_{-\infty}^{+\infty} \rho(t)  e^{it\mu} P_\chi \circ e^{-itQ} dt =\tr  d_\chi \int_{-\infty}^{+\infty}  \int _G\rho(t)  e^{it\mu}\overline{\chi(g)} T(g)  \circ e^{-itQ} \d g \d t.
\eq
\end{lemma}
\qed

Let us now recall that   $U_{1/2}:  \Cinft(M , \Omega_{1/2})\rightarrow \Cinft(\R\times M , \Omega_{1/2})$ can be characterized  globally as a Fourier integral operator of class $I^{-1/4}(\R \times M, M, C')$ with canonical relation given by \eqref{eq:canrel}. This means that for each coordinate patch $(\kappa, Y)$, and sufficiently small $t \in (-\delta, \delta)$, the kernel of $U_{1/2}(t)$ can be described locally as an oscillatory integral of the form
\bqn 
\tilde U (t,\tilde x, \tilde y)= \int_{\R^n} e^{i( \psi(t,\tilde x, \eta) - \eklm{\tilde y, \eta } ) } a ( t, \tilde x, \eta) \dbar \eta
\eqn 
on any compactum in $Y \times Y$, where  $\tilde x, \tilde y \in \tilde Y = \kappa(Y)\subset \R^n$, and $a \in \mathrm{S}^0_{phg}$ is a classical symbol with $a(0, \tilde x, \eta)=1$, while $\psi(t,\tilde x, \eta) - \eklm{\tilde y , \eta}$ is the defining phase function of $C$ in the sense that 
\bqn 
C'= \mklm{(t,\gd \psi/\gd t),(  \tilde x, \gd \psi/ \gd \tilde x),(   \gd \psi/ \gd \eta, -\eta)},
\eqn 
see \cite{hoermanderIV}, page 254. Here we employed the notation $\dbar \eta= (2\pi)^{-n} \d \eta$, $d\eta$ being Lebesgue measure in $\rn$. 
Since  $\tau + q(x,\xi)=0$ on $C$, and $(\tilde x, \gd \psi/ \gd \tilde x)= (\gd \psi/\gd \eta, \eta)$ for $t=0$, we deduce $d_{\tilde x, \eta} \psi ( 0, \tilde x, \eta) = d_{\tilde x, \eta} \eklm{\tilde x, \eta}$, so that $\psi$ is the solution of  the Hamilton-Jacobi problem
\bqn 
\frac {\gd \psi} { \gd t } + q \Big (x, \frac{\gd \psi}{\gd \tilde x}\Big )=0 , \qquad  \psi( 0, \tilde x, \eta) = \eklm{\tilde x, \eta},
\eqn  
$\psi$ being homogeneous of degree $1$. To construct an approximation of $U(t): \L^2(M) \rightarrow \L^2(M)$, let $\mklm{(\kappa_\gamma, Y_\gamma)}$ be an atlas  for $M$,  $\mklm{f_\gamma}$ a corresponding partion of unity, and  
\bqn 
\hat v (\eta) = \int_{\R^n} e^{-i\langle \tilde y, \eta \rangle} v(\tilde y) \, d\tilde y, \qquad v \in \CT(\tilde Y_\gamma),
\eqn
  the Fourier transform of $v$. Write $ (\kappa_\gamma^{-1})^\ast  \d M= \beta_\gamma \d \tilde y$, and 
denote by  $\tilde U_\gamma(t)$ the operator
\bqn 
[\tilde U_\gamma(t)v] (\tilde x)= \int _{\R^n} e ^{i\psi_\gamma(t,\tilde  x, \eta)} a_\gamma (t, \tilde x, \eta) \widehat{v\beta_\gamma}(\eta) \dbar \eta, 
\eqn 
$a_\gamma$ and $\psi_\gamma$ being as described above, and set $\bar U_\gamma(t) u = [ \tilde U_\gamma(t) (u \circ \kappa_\gamma^{-1})] \circ \kappa_\gamma$, $ u \in \CT( Y_\gamma)$,  so that we obtain the diagram 
\begin{displaymath}
\begin{CD} 
\CT(Y_\gamma)       @>{\bar U_\gamma(t)}>>   \Cinft(Y_\gamma)               \\
@A {\kappa_\gamma^\ast}AA @AA {\kappa_\gamma^\ast}A\\
 \CT(\tilde Y_\gamma)  @> {\tilde U_\gamma(t)}>>  \Cinft(\tilde Y_\gamma)       
\end{CD}
\end{displaymath} 
 Consider further test functions $\bar f_\gamma \in \CT( Y_\gamma)$ satisfying $\bar f_\gamma \equiv 1$ on $\supp f_\gamma$, and define
\bqn 
\bar U(t) = \sum _\gamma F_\gamma \, \bar U_\gamma(t) \, \bar F_\gamma, 
\eqn
where  $F_\gamma$, $\bar F_\gamma$ denote the multiplication operators corresponding to  $f_\gamma$ and $\bar f_\gamma$, respectively. Then  the result of H\"ormander implies that 
\bq
\label{eq:R(t)}
R(t) = U(t) -\bar U(t) \, \text{is an operator with smooth kernel,}
\eq 
compare \cite{grigis-sjoestrand}, page 134.  Next, one computes  for $u \in \Cinft(M)$
\begin{align*}
F_\gamma \bar U_\gamma(t) \bar F_\gamma u (x) &= f_\gamma(x) [ \tilde U_\gamma(t) ( \bar f_\gamma u \circ \kappa^{-1}_\gamma)] \circ \kappa_\gamma(x)\\ 
& = f_\gamma(x) \int_{\R^n} e^{i\psi_\gamma ( t, \kappa_\gamma(x),\eta)} a_\gamma(t, \kappa_\gamma(x), \eta) [\widehat{\beta_\gamma (\bar f_\gamma u \circ \kappa^{-1}_\gamma)}] (\eta) \dbar \eta \\
&=\int_{\tilde Y_\gamma} \int_{\R^n} f_\gamma(x) e^{i[\psi_\gamma(t,\kappa_\gamma(x), \eta)- \eklm{ \tilde y,\eta}]} a_\gamma(t, \kappa_\gamma(x), \eta)  (\bar f_\gamma u) (\kappa_\gamma^{-1} ( \tilde y )) \beta_\gamma(\tilde y) d\tilde y \, \dbar \eta \\
&= \int _{Y_\gamma} \Big [ f_\gamma(x) \int _{\R^n}   e^{i[\psi_\gamma(t,\kappa_\gamma(x), \eta)- \eklm{ \kappa_\gamma(y),\eta}]} a_\gamma(t, \kappa_\gamma(x), \eta) \dbar \eta \bar f _\gamma(y) \Big ] u(y) dM(y),
\end{align*}
where the last two expressions are oscillatory integrals with suitable regularizations.
With \eqref{eq:R(t)} and the previous lemma we therefore obtain for $\hat \sigma_\chi (\rho e^{i(\cdot) \mu}) $ the expression 
\begin{gather*}
 \frac{d_\chi}{(2\pi)^n} \sum _\gamma  \int_{-\infty}^{+\infty}  \int _G  \int_{T^\ast Y_\gamma} \rho(t)  e^{it\mu}\overline{\chi(g)} f_\gamma(g^{-1} x) e^{i[\psi_\gamma(t,\kappa_\gamma(g^{-1} x), \eta)- \eklm{ \kappa_\gamma(x),\eta}]} a_\gamma(t, \kappa_\gamma(g^{-1}x), \eta) \\  \bar f _\gamma(x)  d(T^\ast Y_\gamma)(x, \eta) \d g \d t + O(|\mu|^{-\infty}),
 \end{gather*}
where  $ d(T^\ast Y_\gamma)(x, \eta)$ denotes the canonical volume density on $T^\ast Y_\gamma$, and $\rho \in \CT(-\delta, \delta)$.  After  the substitution $x'=g^{-1}x$ we get the following 
\begin{corollary}
For $\rho \in \CT(-\delta, \delta)$ one has the equality
\begin{gather*}
\hat \sigma_\chi (\rho e^{i(\cdot) \mu}) 
=  \frac{d_\chi}{(2\pi)^n} \sum _\gamma  \int_{-\delta}^{+\delta}  \int _G  \int_{T^\ast  Y_\gamma } e^{i\big [\psi_\gamma(t,\kappa_\gamma(x), \eta)- \eklm{ \kappa_\gamma( g x),\eta}+t \mu \big ]}  \rho(t)  \overline{\chi(g)} f_\gamma(x)  \\  a_\gamma(t, \kappa_\gamma(x) , \eta)   \bar f _\gamma (g x)  J_\gamma(g, x)  d(T^\ast Y_\gamma)( x, \eta) \d g \d t + O(|\mu|^{-\infty}),
\end{gather*}
where $ J_\gamma(g, x)$ is a Jacobian. 
\end{corollary}
\qed

\subsection*{The singularity of $\hat \sigma_\chi$ at t=0.} So far we have expressed  $\hat \sigma_\chi$ as an oscillatory integral. In order to study it by means of the stationary phase theorem, let us remark that since $\psi_\gamma$ is homogeneous in $\eta$ of degree $1$, Taylor expansion for small $t$ gives
\bqn 
\psi_\gamma(t,\tilde x, \eta) =\psi_\gamma(0,\tilde x, \eta) +t \frac{\gd \psi_\gamma}{\gd t} (0, \tilde x , \eta) + O(t^2)|\eta|= \eklm{\tilde x, \eta} -t q_\gamma(\tilde x, \eta) +O(t^2)|\eta|, 
\eqn 
where we wrote $q_\gamma(\tilde x, \eta)=q(\kappa_\gamma^{-1}(\tilde x),\eta)$. In other words, there exists a smooth function $\zeta_\gamma$ which is homogeneous in $\eta$ of degree $1$ satisfying
\begin{align}
\begin{split}\label{eq:zeta}
\psi_\gamma(t, \tilde x, \eta) &= \eklm {\tilde x, \eta} -t \zeta_\gamma(t, \tilde x, \eta), \\
\zeta_\gamma(0, \tilde x, \eta) &= q_\gamma(\tilde x, \eta), \qquad -2 \gd_t \zeta_\gamma(0, \tilde x, \eta) = \eklm{ \gd_\eta q_\gamma(\tilde x, \eta), \gd_{\tilde x} q_\gamma (\tilde x, \eta)}.
\end{split}
\end{align}
Let us now define
\bqn 
\F(\tau, \tilde x, \eta)= \int_{-\infty}^{+\infty} e^{it \tau} \rho(t) a_\gamma(t, \tilde x, \eta)e^{iO(t^2) | \eta| }  dt.
\eqn 
Clearly, $\F(\tau, \tilde x, \eta)$  is rapidly decaying as a function in $\tau$. More precisely, since $a_\gamma \in \mathrm{S}^0_{phg}$, 
\bq
\label{eq:reg}
|\F(\tau, \tilde x, \eta)| \leq C_N (1+\tau^2)^{-N}, \qquad \forall N \in \N, \, \tilde x \in \tilde Y_\gamma, \, \eta \in \R^n,
\eq
for some constant $C_N>0$ which depends only on $N$. Next note that $ q_\gamma(\tilde x, \omega) \geq \text{const} > 0$ for all $\tilde x$ and $\omega \in S^{n-1}=\mklm{\eta \in \R^n: \norm{\eta}=1}$. There must therefore exist a constant $C>0$ such that 
\bqn 
 C |\eta| \geq q_\gamma( \tilde x, \eta) \geq \frac 1 C |\eta| \qquad \forall \tilde x \in \tilde Y_\gamma, \, \eta \in \R^n,
\eqn 
which implies that for fixed $\mu$, $\F(\mu-q_\gamma(\tilde x, \eta), \tilde x, \eta)$ is rapidly decaying in $\eta$. This yields a regularization of the oscillatory integral in the previous corollary, and we obtain
\begin{gather*}
\hat \sigma_\chi (\rho e^{i(\cdot) \mu}) 
=  \frac{d_\chi}{(2\pi)^n} \sum _\gamma    \int _G  \int_{T^\ast  Y_\gamma } e^{i\eklm{\kappa_\gamma( x)-\kappa_\gamma( g x),\eta}}   \overline{\chi(g)} f_\gamma( x)   \F(\mu -q( x, \eta), \kappa_\gamma( x), \eta)  \\ \bar f _\gamma (g x)  J_\gamma(g,  x)  d(T^\ast Y_\gamma)( x, \eta) \d g  + O(|\mu|^{-\infty}).
\end{gather*}
 But even more is true. If we replace $\mu$ by $-\nu$, then $(\mu-q_\gamma(\tilde x, \eta))^2 \geq 2 \nu q_\gamma(\tilde x, \eta)\geq 2 \nu |\eta|/C$. From \eqref{eq:reg} we therefore infer  that $\hat \sigma_\chi( \rho e^{i(\cdot) \mu})$ is rapidly decreasing as $\mu \to -\infty$, reflecting the positivity of the spectrum.  Assume now that $|1- q_\gamma (\tilde x, \eta/\mu)| \geq \text{const} >0$. Then
\begin{align*}
|\F(\mu-q_\gamma( \tilde x, \eta), \tilde x, \eta)| &\leq C_{N+M} \frac 1 {|\mu|^{N}} \frac 1 { |1 - q_\gamma( \tilde x, \eta /\mu)|^N} \frac 1 {| \mu - q_\gamma( \tilde x, \eta)|^M}\\
& \leq C_{N+M} \frac 1 {|\mu|^{N}}  \frac 1 {| \mu - q_\gamma( \tilde x, \eta)|^M}
\end{align*}
for arbitrary $N, M \in \N$. Let therefore $0 \leq \alpha \in \CT(1/2, 3/2)$ be such that $\alpha \equiv 1$ in a neighborhood of $1$, so that 
\bqn 
1 - \alpha(q_\gamma( \tilde x , \eta/\mu)) \not= 0 \quad \Longrightarrow \quad |1-q_\gamma( \tilde x, \eta/\mu)| \geq  \text{const} >0.
\eqn  
Substituting $\eta=\mu \eta'$,  we can  rewrite  $\hat \sigma_\chi (\rho e^{i(\cdot) \mu})$ as 
\begin{gather*}
\hat \sigma_\chi (\rho e^{i(\cdot) \mu}) 
= \frac {|\mu|^n  d_\chi}{(2\pi)^n} \sum _\gamma  \int_{-\delta}^{+\delta}  \int _G  \int_{T^\ast Y_\gamma } e^{i\mu \big [\psi_\gamma(t,\kappa_\gamma( x),\eta)- \eklm{ \kappa_\gamma(g x),\eta}+t \big ]}  \rho(t)  \overline{\chi(g)} f_\gamma(x)   \\ a_\gamma(t, \kappa_\gamma( x) , \mu \eta)  \bar f _\gamma (g  x)  J_\gamma(g, x)  \alpha( q( x, \eta)) d(T^\ast  Y_\gamma)(x, \eta) \d g \d t + O(|\mu|^{-\infty}),
\end{gather*}
where all integrals are absolutely convergent. 
Now, since $\zeta_\gamma(0, \tilde x , \omega)= q_\gamma(\tilde x, \omega)$, there exists a constant $C>0$ such that for sufficiently small $t \in (-\delta, \delta)$
\bqn 
 C \geq \zeta_\gamma(t,  \tilde x, \omega) \geq \frac 1 C \qquad \forall \tilde x \in \tilde Y_\gamma, \, \omega \in K,
\eqn 
$K$ being a compactum. By introducing the coordinates $\eta = R\omega$, $R>0$, $\zeta_\gamma(t,\kappa_\gamma( x), \omega)=1$, one finally arrives at the following
\begin{proposition}
\label{prop:0}
Let $\delta>0$ be sufficiently small, and  $\rho \in \CT(-\delta, \delta)$. Then, as $\mu \to +\infty$, 
\begin{gather*}
\hat \sigma_\chi (\rho e^{i(\cdot) \mu}) 
= \frac {\mu^n  d_\chi}{(2\pi)^n} \sum _\gamma  \int_{\R}\int _{\R} e^{i\mu[t-Rt]}  \int _G  \int_{S^\ast_t Y_\gamma } e^{i{R \mu} \eklm{\kappa_\gamma( x) - \kappa_\gamma( g x),\omega}}  \rho(t)  \overline{\chi(g)} f_\gamma(x) \\   a_\gamma(t, \kappa_\gamma( x) , \mu R \omega)  \bar f _\gamma (g  x)  J_\gamma(g,  x)  \alpha(R q( x, \omega)) {d(S^\ast_t Y_\gamma)(x, \omega) \d g}    R^{n-1} \d R\d t, 
\end{gather*}
up to terms of order  $O(\mu^{-\infty})$, where $S^\ast_t Y_\gamma =\mklm{(x,\omega) \in T^\ast Y_\gamma: \zeta_\gamma(t, \kappa_\gamma(x), \omega)=1}$.  Here ${d(S^\ast_t Y_\gamma)(x, \omega)}$ denotes the quotient of the volume density on $T^\ast Y_\gamma$ by Lebesgue measure in $\R$ with respect to $\zeta_\gamma(t, \tilde x, \omega)$. On the other hand,  $\hat \sigma_\chi (\rho e^{i(\cdot) \mu}) $ is rapidly decaying as  $\mu \to -\infty$.
\end{proposition}\qed

Assume now that $\mu \geq 1$. To study the limit of $\hat \sigma_\chi (\rho e^{i(\cdot) \mu}) $  as $\mu \to +\infty$, we shall apply the stationary phase principle to the $Rt$-integral first, and then to the integral over $G\times S^\ast_t Y_\gamma$. For the later phase analysis, it will be convenient to replace the integration over $G\times S^\ast_t Y_\gamma$ by an integration over $G \times T^\ast Y_\gamma$. Let us us therefore note that since $\alpha \in \CT(1/2, 3/2)$,  
 \bqn 
1/2 \leq  R q( x, \omega) \leq 3/2 \qquad \forall x \in Y_\gamma, \omega \in (S^\ast_t Y_
\gamma)_x.
 \eqn
For sufficiently small $ \delta$ we can therefore assume that the $R$-integration is over a compact intervall in $\R^+$. 
Let now $\sigma \in \CT(\R)$ be a non-negative function with $\int \sigma(s) ds =1$, and define $\Delta_{\epsilon, r}(s)=\epsilon^{-1} \sigma((s-r)/\epsilon)$, $r \in \R$. Then 
\bqn 
\Delta_{\epsilon,r} \longrightarrow \delta_r \qquad \text{ as } \epsilon \to 0
 \eqn
 with respect to the weak topology in $\E'(\R)$. Using this approximation of the $\delta$-distribution and the theorem of Lebesgue on bounded convergence we obtain for $\hat \sigma_\chi (\rho e^{i(\cdot) \mu}) $ the expression 
 \begin{gather*}
 \frac {\mu^n  d_\chi}{(2\pi)^n} \sum _\gamma  \int_{\R}\int _{\R} e^{i\mu[t-Rt]}  \int _G  \int_{S^\ast_t Y_\gamma }  \lim_{\epsilon \to 0} \int_{\R^+} e^{i {\mu} \eklm{\kappa_\gamma( x) - \kappa_\gamma( g x),s \omega}}  \rho(t)  \overline{\chi(g)} f_\gamma(x) \\   a_\gamma(t, \kappa_\gamma( x) , \mu s \omega)  \bar f _\gamma (g  x)  J_\gamma(g,  x)  \alpha( q( x, s \omega)) 
 \Delta_{\epsilon,R} (s)  s^{n-1}  \, ds \,  {d(S^\ast_t Y_\gamma)(x, \omega) \d g}     \d R\d t \\
= \frac {\mu^n  d_\chi}{(2\pi)^n} \, \lim_{\epsilon \to 0}\,  \sum _\gamma  \int_{\R}\int _{\R} e^{i\mu[t-Rt]}  \int _G  \int_{T^\ast Y_\gamma }    e^{i {\mu} \eklm{\kappa_\gamma( x) - \kappa_\gamma( g x),\eta}}  \rho(t)  \overline{\chi(g)} f_\gamma(x) \\   a_\gamma(t, \kappa_\gamma( x) , \mu \eta)  \bar f _\gamma (g  x)  J_\gamma(g,  x)  \alpha( q( x, \eta))  
  \Delta_{\epsilon,R}(\zeta_\gamma(t,\kappa_\gamma( x), \eta) )    {d(T^\ast Y_\gamma)(x, \eta) \d g}     \d R\d t, 
\end{gather*}
since $ \int \Delta_{\epsilon, r}(s) \, ds =1$, and all integrals are over compact sets. Let us now  apply the stationary phase theorem to the $Rt$-integral for each fixed $\epsilon$. We then arrive at  the following
\begin{theorem}
\label{thm:Rt}
Let $\rho \in \CT(-\delta, \delta)$, and $\mu \geq 1$. For sufficiently small $\delta$ one has the asymptotic expansion
\begin{gather*}
\hat \sigma_\chi (\rho e^{i(\cdot) \mu})
=\frac {\mu^{n-1}  d_\chi \rho(0)}{(2\pi)^{n-1}} \lim_{\epsilon \to 0} \sum _\gamma   \int _G  \int_{T^\ast Y_\gamma }    e^{i {\mu} \eklm{\kappa_\gamma( x) - \kappa_\gamma( g x),\eta}}   \overline{\chi(g)} f_\gamma(x)      \bar f _\gamma (g  x)  J_\gamma(g,  x)   \\ \Delta_{\epsilon,1}(q( x, \eta))  \,    {d(T^\ast Y_\gamma)(x, \eta) \d g}  + O(\mu^{n-2}),
\end{gather*}
where 
\begin{gather*}
 O(\mu^{n-2})= C \mu^{n-2} \sum _{|\beta| \leq 5} \sup_{R,t} \Big | \gd _{R,t}^\beta    \int _G  \int_{T^\ast Y_\gamma }    e^{i {\mu} \eklm{\kappa_\gamma( x) - \kappa_\gamma( g x),\eta}}  \rho(t)  \overline{\chi(g)} f_\gamma(x) \\   a_\gamma(t, \kappa_\gamma( x) , \mu \eta)  \bar f _\gamma (g  x)  J_\gamma(g,  x)  \alpha( q( x, \eta))  
  \Delta_{\epsilon,R}(\zeta_\gamma(t,\kappa_\gamma( x), \eta))  \,  {d(T^\ast Y_\gamma)(x, \eta) \d g}  \Big |. 
\end{gather*}
For $\mu \to -\infty$,  the expression $\hat \sigma_\chi (\rho e^{i(\cdot) \mu}) $ is rapidly decaying.
\end{theorem}
\begin{proof} Since $(t, R)= (0,1)$ is the only critical point of $t-Rt$, the assertion follows from the classical stationary phase theorem  \cite{grigis-sjoestrand}, Proposition 2.3.
\end{proof}

We have thus partially unfolded the singularity of $\hat \sigma_\chi$ at $t=0$. Theorem \ref{thm:Rt} shows that  its structure is more involved than in the non-equivariant setting, or in the case of finite group actions, compare \cite{duistermaat-guillemin75}, pp. 46, and \cite{bruening83}, pp 92. To obtain a complete description, we are therefore left with the task of examining the asymptotic behavior of integrals of the form 
\begin{align}
\label{int}
I(\mu)
&= \int _{T^\ast Y}  \int_{G} e^{i\mu  \Phi(x , \xi,g) }   a( g  x,  x , \xi,g)  \d g \d(T^\ast Y)(x,\xi),  \qquad \mu \to +\infty,  
\end{align}
via the generalized stationary phase theorem, where $(\kappa,Y)$ are local coordinates on $M$,  and $dg$ is the  volume density of a left invariant metric on $G$, while $a \in \CT(Y \times T^\ast  Y\times G)$ is an amplitude which might also depend on $\mu$, and
\bq
\label{eq:phase}
\Phi(x, \xi, g) =\eklm{\kappa(x) - \kappa (g x), \xi }.
\eq 
This will occupy us for the rest  of this paper.

\section{Compact group actions and the momentum  map}
\label{sec:CGMM}

\subsection*{Compact group actions} We commence this section by briefly recalling some basic facts about compact group actions that will be needed later. For a detailed exposition,  we refer the reader  to \cite{bredon}. Let $G$ be a compact Lie group acting locally smoothly on some $n$-dimensional $\Cinft$-manifold $M$, and assume that the orbit space $M/G$ is connected. Denote the stabilizer, or isotropy group, of a point $x\in M$ by
$$G_x=\{ g\in G\, :\, g\cdot x=x \}.$$
The orbit of  $x \in M$  under the action of $G$ will be denoted by $G\cdot x$ or, alternatively, by  $\mathcal{O}_x$, and is homeomorphic to $G/G_x$. The equivalence class of an orbit $\mathcal{O}_x$ under equivariant homeomorphisms is called its orbit type, and the conjugacy class $(G_x)$ of $G_x$ in $G$ its isotropy type. Now, if  $K_1$ and $K_2$ are closed subgroups of $G$, a partial ordering of orbit and isotropy types is   given by
\bqn 
\mathrm{type}\, (G/K_1) \leq  \mathrm{type}\,(G/K_2) \Longleftrightarrow (K_2) \leq (K_1) \Longleftrightarrow
 K_2 \mbox{ is conjugated to a subgroup of } K_1.
\eqn
One of the main results in the theory of compact group actions is the following
\begin{theorem}[Principal orbit theorem]\label{thm:orbit}
There exists a maximum orbit type $G/H$ for $G$ on $M$. The union $M{(H)}$ of orbits of type $G/H$ is open and dense, and its image in $M/G$ is connected.
\end{theorem}
\begin{proof}
See \cite{bredon}, Theorem IV.3.1.
\end{proof}
Orbits of type $G/H$ are called of principal type, and the corresponding isotropy groups are called principal. A principal isotropy group  has the property that it is  conjugated to a subgroup of each stabilizer of $M$. Let $K\subset G$ be a closed subgroup containing $H$. An orbit of type $G/K$ is called singular, if $\dim K/H>0$, and exceptional, if $K/H$ is finite and non-trivial, in which case $\dim G/K= \dim G/H$, but $\mathrm{type}\, (G/K) \not= \mathrm{type}\,(G/H)$.
The following result says that there is a stratification of $G$-spaces into orbit types.
\begin{theorem}\label{thm:strat}
Let $G$ and $M$ be as above, $K$ a subgroup of $G$, and denote the set of points on orbits of type $G/K$ by $M{(K)}$. Then $M{(K)}$ is a topological manifold, which is locally closed. Furthermore, $\overline{M{(K)}}$ consists of orbits of type less than or equal to type $G/K$. The orbit map $M{(K)} \rightarrow M{(K)}/G$ is a fiber bundle projection with fiber $G/K$ and structure group $N(K)/K$.
\end{theorem}
\begin{proof}
See \cite{bredon}, Theorem IV.3.3.
\end{proof}
Let now $M_\tau$ denote the union of non-principal orbits of dimension at most $\tau$.
\begin{proposition}
\label{prop:dim}
If $\kappa$ is the dimension of a principal orbit, then $\dim M/G =n-\kappa$, and  $M_\tau$  is a closed set of dimension at most $n-\kappa+\tau-1$.
\end{proposition}
\begin{proof}
See \cite{bredon}, Theorem IV.3.8.
\end{proof}
Here the dimension of $M_\tau$ is understood in the sense of general dimension theory.
In what follows, we shall write $\mathrm{Sing}\,  M= M - M{(H)}=M_\kappa$. Clearly,
\bqn
\mathrm{Sing} \, M= M_0 \cup (M_1-M_0) \cup (M_2-M_1) \cup  \dots \cup(M_\kappa -M_{\kappa -1}),
\eqn
where $M_i-M_{i-1}$ is precisely the union of non-principal orbits of dimension $i$, and $M_{-1}=\emptyset$, by definition. Note that 
\bqn
M_i-M_{i-1}=\bigcup_{j} M{(H^i_j)}, \qquad H^i_j \subset G, \, \dim G/H^i_j =i, 
\eqn
is a disjoint union of topological manifolds of possibly different dimensions. Now, a crucial feature of smooth compact group actions is the existence of invariant tubular neighborhoods. 
\begin{theorem}[Invariant tubular neighborhood theorem]
\label{thm:tub}
 Assume that  $G$ acts smoothly on $M$, and let $A$ be a closed $G$-invariant submanifold of $M$.  Then $A$ has an invariant tubular neighborhood, that is, there exists a smooth $G$-vector bundle $\xi:E\rightarrow A$ on $A$ together with an equivariant diffeomorphism $\psi:E \rightarrow M$ onto an open neighborhood $W$ of $A$ such that the restriction of $\psi$ to the the zero section of $\xi$ is the inclusion of $A$ in $M$.
\end{theorem}
\begin{proof}
See \cite{bredon}, Theorem VI.2.2.
\end{proof}
Furthermore, by taking a  $G$-invariant metric on $M$, $W$ can be identified via the exponential map with a neighborhood of the zero section in the normal bundle $\nu(A)$ of $A$. From now on, let $M$ be a closed, connected Riemannian manifold, and $G$ a connected compact Lie group acting on $M$ by isometries. Relying on the stratification of $M$ into orbit types, one can construct a  $G$-invariant covering of $M$ as follows, compare  \cite{kawakubo}, Theorem 4.20. Let $(H_1), \dots, (H_L)$ denote the isotropy types of $M$, and  arrange them in such a way that 
\bqn
(H_i) \geq (H_j)  \quad \Rightarrow \quad i \leq j.
\eqn
By Theorem \ref{thm:strat},  $M$ has a stratification into orbit types according to  $M=M(H_1) \cup \dots \cup M(H_L)$, and the principal orbit theorem implies that  the set $M(H_L)$ is open and dense in $M$, while $M(H_1)$ is a closed, $G$-invariant submanifold. Denote by $\nu_1$ the normal $G$-vector bundle of $M(H_1)$, and by $f_1: \nu_1 \rightarrow M$ a $G$-invariant tubular neighbourhood of $M(H_1)$ in $M$. Take a $G$-invariant metric on $\nu_1$, and put
\bqn
{D}_t(\nu_1)=\mklm {v \in \nu_1: \norm{v} \leq t }, \qquad t >0.
\eqn
We then define the compact, $G$-invariant submanifold with boundary
\bqn
M_2=M - f_1(\stackrel{\circ}{D}_{1/2}(\nu_1)), 
\eqn
on which the isotropy type $(H_1)$ no longer occurs, and endow it with a $G$-invariant Riemannian metric with product form in a $G$-invariant collar neighborhood of $\gd M_2$ in $M_2$. Consider now the union $M_2(H_2)$ of orbits in $M_2$ of type $G/H_2$, a compact $G$-invariant submanifold of $M_2$ with boundary, and let $f_2:\nu_2 \rightarrow M_2$ be a $G$-invariant tubular neighbourhood  of $M_2(H_2)$ in $M_2$, which exists due to the particular form of the metric on $M_2$. Taking a $G$-invariant metric on $\nu_2$, we define
\bqn
M_3=M_2 - f_2(\stackrel{\circ}{D}_{1/2}(\nu_2)), 
\eqn
which constitutes a compact $G$-invariant submanifold with corners and isotropy types $(H_3), \dots (H_L)$. Continuing this way, one finally obtains the decomposition 
\bqn  
M= f_1({D}_{1/2}(\nu_1)) \cup \dots  \cup f_L({D}_{1/2}(\nu_L)),
\eqn
where we identified  $f_L({D}_{1/2}(\nu_L))$ with $M_L$, which leads to the covering 
\bq
\label{eq:covering}
M= f_1(\stackrel{\circ}{D}_{1}(\nu_1)) \cup \dots \cup f_L(\stackrel{\circ}{D}_{1}(\nu_L)),\qquad  f_L(\stackrel{\circ}{D}_{1}(\nu_L))=\stackrel{\circ} M_L.
\eq
We introduce  now the set 
 \bq
\label{eq:calN}
\mathcal{N}=\mklm{(x,g) \in \M: gx =x},  \qquad  \M= M \times G,
\eq
which will play an important role later.  If all isotropy groups of the $G$-action on $M$ have the same dimension, that is, if there are no singular orbits, $\mathcal{N}$ is a smooth manifold. Otherwise, $\mathcal{N}$ is singular, as can be seen from Theorem \ref{thm:tub}. Clearly, $\mathcal{N}= \bigcup_{k} \mathrm{Iso} \, M(H_k)$, where $\mathrm{Iso} \, M(H_k)\rightarrow  M(H_k)$ denotes the isotropy bundle on $M(H_k)$, and by Proposition \ref{prop:dim} we have 
\bqn 
\dim \mathrm{Iso} \, M(H_k)= \dim M(H_k) + \dim H_k \leq n - \kappa + \tau -1 +\dim G - \tau = \dim  \mathrm{Iso} \, M(H_L)-1, 
\eqn
where $ 1 \leq k \leq L-1$, and $\tau = \dim G/H_k$.  The regular part $\mathrm{Reg} \, \mathcal{N}$ is given by the union over all total spaces $\mathrm{Iso} \, M(H_k)$ with non-singular isotropy type $(H_k)$, and  is in general not dense in $\mathcal{N}$. 

\subsection*{The momentum map} We shall now discuss the canonical symplectic momentum map of a closed, connected Riemannian manifold  $M$ on which a  connected, compact Lie group $G$ acts by isometries, and the way it is  related to our problem. Consider the cotangent bundle  $\pi:T^\ast M\rightarrow M$, as well as the tangent bundle $\tau: T(T^\ast M)\rightarrow T^\ast M$, and define on $T^\ast M$ the Liouville form 
\bqn 
\Theta(\mathfrak{X})=\tau(\mathfrak{X})[\pi_\ast(\mathfrak{X})], \qquad \mathfrak{X} \in T(T^\ast M).
\eqn
We regard $T^\ast M$ as a  symplectic manifold with symplectic form 
\bqn 
\omega= d\Theta,
\eqn
and define for any  element $X$ in the Lie algebra $\g$ of $G$ the function
\bqn
J_X: T^\ast M \longrightarrow \R, \quad \eta \mapsto \Theta(\widetilde{X})(\eta),
\eqn
where $\widetilde X$ denotes the fundamental vector field on $T^\ast M$, respectively $M$,  generated by $X$. Note that  $ \Theta(\widetilde{X})(\eta)=\eta(\widetilde X_{\pi(\eta)})$. Indeed, put $\gamma(s)=\e{-s X} \cdot \eta$, $s \in ( -\epsilon, \epsilon)$ for some $\epsilon >0$, so that  $\gamma(0)=\eta$, $\dot {\gamma}(0) =\widetilde X_\eta$. Since $\pi(\e{-sX} \cdot \eta)=\e{-sX} \cdot \pi(\eta)$, one computes
\bqn 
\pi_\ast(\widetilde X_\eta)=\frac d {ds} \pi \circ \gamma(s)_{|s=0} = \frac d {ds} \e{-sX} \cdot \pi(\eta)_{|s=0} = \widetilde X_{\pi(\eta)}.
\eqn 
Therefore
$$ \Theta(\tilde {X})(\eta)=\tau(\widetilde X_\eta)[\pi_\ast(\widetilde X_\eta)]=\eta(\widetilde X_{\pi(\eta)}),$$
as asserted. The function $J_X$ is linear in $X$, and due to the invariance of the Liouville form
\bqn
\mathcal{L}_{\widetilde X} \Theta = dJ_X+ \iota_{\widetilde X} \omega =0, \qquad \forall X \in \g,
\eqn
where $\mathcal{L}$ denotes the Lie derivative. This means that  $G$ acts on $T^\ast M$ in a Hamiltonian way. The corresponding symplectic momentum  map is  then given by 
\bqn
\mathbb{J}:T^\ast M\to \g^\ast,  \quad \mathbb{J}(\eta)(X)=J_X(\eta).
\eqn
As explained in the previous section, we are interested in the asymptotic behavior of integrals of the form \eqref{int}, and would like to  study them by means of the generalized stationary phase theorem, for which  we have to compute the critical set of the phase function $\Phi(x, \xi,g)$. Let $(\kappa,Y)$ be local coordinates on $M$ as in \eqref{int}, and write  $\kappa(x)=(\tilde x_1,\dots, \tilde x_n)$,  $\eta=\sum \xi_i (d\tilde x_i)_x \in T_x^\ast Y$. One  computes then for any $X \in \g$
\begin{align*}
\frac d{dt} \Phi( x, \xi, \e{tX})_{|t=0}&=\frac d{dt} \eklm{ \kappa (\e{-tX}  x),\xi}_{|t=0}= \sum \xi_i \widetilde X_{x}(\tilde x_i)  = \sum \xi_i (d\tilde x_i)_{x}(\widetilde X_{x})\\ 
&=\eta (\widetilde X_{x})=\Theta(\widetilde X)(\eta) = \mathbb{J}(\eta)(X).
\end{align*}
Therefore $\Phi$ represents the global analogue of the momentum map; furthermore, their critical sets are essentially the same. Indeed, one has
\bq
\label{eq:vanx}
\gd _{\tilde x} \Phi ( \kappa ^{-1} ( \tilde x), \xi, g) = [ \1 - \, ^T ( \kappa \circ g \circ \kappa ^{-1} )_{\ast,  \tilde x} ]  \xi= (\1 - g^\ast_{\tilde x}) \cdot \xi, 
\eq
so that $\gd _x \Phi ( x, \xi, g )=0$ amounts precisely to the condition $g^\ast  \xi=\xi$. Since  $\gd_\xi \Phi(x,\xi,g)=0$ if, and only if $gx=x$, 
one obtains  
 \begin{align*}
 \begin{split}
 \Crit(\Phi)&=\mklm{(x,\xi,g) \in T^\ast Y \times G: (\Phi_{\ast})_{(x,\xi,g)}=0} 
= \mklm{( x  ,\xi, g) \in (\Omega  \cap T^\ast Y)\times G:  \,  g\cdot (x,\xi)=(x,\xi)},
\end{split}
\end{align*}
  where $
  \Omega=\mathbb{J}^{-1}(0)
$
is the zero level of the momentum  map. Note that 
\bq
\label{eq:Ann}
\eta \in \Omega  \cap T^\ast _xM \quad \Longleftrightarrow \quad \eta  \in \mathrm{Ann}(T_x (G\cdot x)),
\eq
where $\mathrm{Ann} \, (V_x) \subset T_x^\ast M$ denotes the annihilator of a vector subspace $V_x \subset T_xM$. Now,  the major difficulty in applying the generalized stationary phase theorem in our setting  stems from the fact that, due to the  orbit structure of the underlying group action,  the zero level $\Omega$ of the momentum  map, and, consequently, the considered critical set $\Crit(\Phi)$, are in general singular varieties. In fact,  if the $G$-action on $T^\ast M$ is not free, the considered momentum  map is no longer a submersion, so that $\Omega$ and the symplectic quotient $\Omega/G$ are no longer smooth. Nevertheless, it can be shown that these spaces have Whitney stratifications into smooth submanifolds, see  \cite{lerman-sjamaar}, and \cite{ortega-ratiu}, Theorems 8.3.1 and 8.3.2, which correspond to the stratifications of $T^\ast M$, and $M$ by orbit types  \cite{duistermaat-kolk}. In particular, if $(H_L)$ denotes the principal isotropy type of the $G$-action in $M$, 
$\Omega$ has a principal stratum given by
\bq
\label{eq:x}
\mathrm{Reg} \, \Omega =\mklm{ \eta \in \Omega: G_\eta \sim H_L},
\eq
  where $G_\eta$ denotes the isotropy group of $\eta$. To see this, let $\eta \in \Omega\cap T^\ast _xM$, and  $G_x \sim H_L$. In view of \eqref{eq:Ann} one computes  for $g \in G_x$, and  $\mathfrak{X} =\mathfrak{X}_T + \mathfrak{X}_N \in T_xM =T_x(G\cdot x) \oplus N_x (G \cdot x)$
 \bqn 
g \cdot \eta(\mathfrak{X})  =  \eta \big ( (L_{g^{-1}})_{\ast,x} ( \mathfrak{X}_N))= \eta(\mathfrak{X}),
\eqn 
since $G_x$ acts trivially on $N_x(G\cdot x)$, see  \cite{bredon}, pages 308 and 181. But $G_\eta \subset G_{\pi(\eta)}$ for arbitrary $\eta\in T^\ast M$,  so that we conclude  
\bq
\label{eq:y}
\eta \in \Omega\cap T^\ast _xM, \quad G_{x} \sim H_L \quad \Rightarrow  \quad G_\eta =G_{x}.
\eq
Since  the stratum $\mathrm{Reg} \,\Omega$ is  open and dense in  $\Omega$,  equality \eqref{eq:x} follows. Note that $\mathrm{Reg }\,  \Omega$ is a smooth submanifold in $T^\ast M $ of codimension equal to the dimension $\kappa$ of a principal $G$-orbit in $M$. 
It is therefore clear that the smooth part of $\Crit(\Phi)$ corresponds to 
\bq
\label{eq:z}
\mathrm{Reg}\,  \Crit(\Phi)=\mklm{(x, \xi,g)  \in (\mathrm{Reg}\, \Omega \cap T^\ast Y) \times G: \, g \in G_{(x,\xi)}},
\eq
and constitutes a submanifold of codimension $2\kappa$.

\section{The generalized stationary phase theorem and resolution of singularities}
\label{sec:SPRS}

\subsection*{The principle of the  stationary phase} 

Since  the critical set  of the phase function \eqref{eq:phase} is not necessarily smooth, the stationary phase method can not immediately be applied to derive asymptotics for the integral \eqref{int}. We shall therefore first partially resolve the singularities of $\mathrm{Crit}(\Phi)$,  and then apply the stationary phase principle  in a suitable resolution space. To explain our approach, let us begin by recalling 

\begin{theorem}[Generalized stationary phase theorem for manifolds]
\label{thm:SP}
Let $M$ be a  $n$-dimensional Riemannian manifold with volume density $dM$,  $\psi \in \Cinft(M)$ a real valued phase function,   and set
\bq
\label{eq:SPT}
\mathcal{I}({\mu})=\int_M e^{i\psi(m)/\mu} a(m) \, dM(m), \qquad \mu >0,
\eq
where $a(m)\in \CT(M)$. Let
$$\mathcal C=\mklm{m \in M: \psi_\ast:T_mM \rightarrow T_{\psi(m)}\R \text{ is zero}}$$
 be the  critical set of the phase function $\psi$, and assume that it is clean in the sense of Bott \cite{bott56}, meaning that 
\begin{enumerate}
\item[(I)] $\mathcal{C}$ is a smooth submanifold of $M$  of dimension $p$ in a neighborhood of the support of $a$;
\item[(II)] at each point $m \in \mathcal{C}$, the Hessian $\psi''(m)$  of $\psi$ is transversally non-degenerate, i.e. non-degenerate on $T _mM/ T_mC \simeq N_m\mathcal{C}$, where $N_m\mathcal{C}$ denotes the normal space to $\mathcal{C}$ at $m$.
\end{enumerate}
\noindent
Then, for all $N \in \N$, there exists a constant $C_{N,\psi}>0$ such that
\bqn
|\mathcal{I}(\mu) - e^{i\psi_0/\mu}(2\pi \mu)^{\frac {n-p}{2}}\sum_{j=0} ^{N-1} \mu^j Q_j (\psi;a)| \leq C_{N,\psi} \mu^N \vol (\supp a \cap \mathcal{C}) \sup _{l\leq 2N} \norm{D^l a }_{\infty,M},
\eqn
where $D^l$ is a differential operator on $M$ of order $l$, and $\psi_0$ is the constant value of $\psi$ on $\mathcal{C}$. Furthermore, for each $j$ there exists a constant $\tilde C_{j,\psi}>0$ such that 
\bqn
|Q_j(\psi;a)|\leq \tilde C_{j,\psi}  \vol (\supp a \cap \mathcal{C}) \sup _{l\leq 2j} \norm{D^l a }_{\infty,\mathcal{C}},
\eqn
and, in particular,
\bqn
Q_0(\psi;a)= \int _{\mathcal{C}} \frac {a(m)}{|\det \psi''(m)_{|N_m\mathcal{C}}|^{1/2}} d\sigma_{\mathcal{C}}(m) e^{ i \frac\pi 4 \sigma_{\psi''}},
\eqn
where $d\sigma_{\mathcal{C}}$ is the induced volume density on $\mathcal{C}$, and 
 $\sigma_{\psi''}$  the constant value of the signature of the transversal Hessian $\psi''(m)_{|N_m\mathcal{C}}$ on $\mathcal{C}$.
\end{theorem}
\begin{proof}
See for instance  \cite{hoermanderI}, Theorem 7.7.5, together with \cite{combescure-ralston-robert}, Theorem 3.3, as well as \cite{varadarajan97}, Theorem 2.12.
\end{proof}
\begin{remark}
\label{rmk:A}
An examination of the proof of the foregoing theorem shows that the constants $C_{N,\psi}$ are essentially bounded from above by 
\bqn 
\sup_{m \in \mathcal{C} \cap \supp a} \norm {\Big ( \psi''(m)_{|N_m\mathcal{C}}\Big ) ^{-1}}.
\eqn
Indeed, let $\alpha:(x,y) \rightarrow m \in \mathcal{O}\subset M$ be local normal coordinates such that $\alpha(x,y) \in \mathcal{C}$ if, and only if, $y=0$. The transversal Hessian $\mathrm{Hess} \, \psi(m)_{|N_m\mathcal{C}}$ is given in these coordinates by the matrix
\bqn 
\Big (\gd _{y_k} \gd _{y_l} (\psi \circ \alpha)(x,0) \Big )_{k,l} \
\eqn
where $m =\alpha (x,0)$, compare \eqref{eq:Hess}. If  the transversal Hessian of $\psi$ is non-degenerate at the point $m=\alpha(x,0)$, then $y=0$ is a non-degenerate critical point of the function  $y \mapsto (\psi \circ \alpha)(x,y)$, and therefore an isolated critical point by the lemma of Morse. As a consequence,  
\bq
\label{eq:est}
\frac{|y|}{|\gd_y (\psi \circ \alpha)(x,y)|} \leq 2 \norm {\Big (\gd_{y_k} \gd _{y_l} (\psi\circ \alpha)(x,0)\Big ) _{k,l}^{-1}}
\eq
for $y$ close to zero. The assertion now follows by applying  \cite{hoermanderI}, Theorem 7.7.5,  to the integral
\bqn 
\int_{\alpha^{-1}(\mathcal{O})} e^{i (\psi \circ \alpha) (x,y)/\mu} (a \circ \alpha)(x,y) \d y \d x
\eqn
in the variable $y$ with $x$ as a parameter, since in our situation the constant $C$ occuring in \cite{hoermanderI}, equation (7.7.12), is precisely bounded by \eqref{eq:est}, if we assume as we may that $a$ is supported near $\mathcal{C}$. A similar observation holds with respect to the constants $\tilde C_{j,\psi}$.
\end{remark}
Conditions (I) and (II) in Theorem \ref{thm:SP} are essential. Actually,  the existence of singularities might alter  the asymptotics, as can be seen from the following
\begin{example}
Let $M=\R^2$, $\psi(x,y)=(xy)^2$, and consider the asymptotic behavior of the integral $\mathcal{I}(\mu)= \int \int e^{i\psi(x,y)/\mu} a(x,y) \, dx \, dy$ as $\mu \to 0^+$, where $a(x,y) \in \CT( \R^2)$ is a compactly supported amplitude, and $dx \, dy$ denotes Lebesgue measure in $\R^2$. The critical set of $\psi$ is given by the singular variety $\Crit (\psi) = \mklm{ xy=0}$, and a computation shows that 
\bqn 
\mathcal{I}(\mu)= \frac{ e^{i\pi/4}}{\sqrt 2} a(0,0 ) (2\pi\mu)^{1/2} \log (\mu^{-1})+ O(\mu^{1/2}).
\eqn
\end{example}
In general, one faces serious difficulties in describing the asymptotic behavior of integrals of the form  \eqref{eq:SPT} if the critical set $\mathcal{C}$ is not smooth and in what follows, we shall indicate how to circumvent this obstacle by  using resolution of singularities.

 \subsection*{Resolution of singularities}   Let $M$ be a smooth variety over a field of characteristic zero, $\mathcal{O}_M$ the structure sheaf of rings of $M$, and $I \subset \mathcal{O}_M$ an ideal sheaf. The aim of resolution of singularities is  to construct a birational morphism $\Pi: \tilde M \rightarrow M$ such that $\tilde M$ is smooth, and the inverse image ideal sheaf $\Pi^\ast( I)\subset \mathcal{O}_{\tilde M}$, which is the ideal sheaf generated by the pullbacks of local sections of $I$, is locally principal. This is called the {principalization} of $I$, and implies resolution of singularities. That is, for every quasi-projective variety $X$, there is a smooth variety $\tilde X$, and a birational and projective morphism $\pi:\tilde X \rightarrow X$. Vice versa, resolution of singularities implies principalization. If $\Pi^\ast (I)$ is monomial, that is, if for every $x \in \tilde M$ there are local coordinates $z_i$ and natural numbers $c_i$ such that 
 \bqn 
 \Pi^\ast (I) \cdot  \mathcal{O}_{x,\tilde M} = \prod_i z_i^{c_i} \cdot \mathcal{O}_{x,\tilde M},
 \eqn
 one obtains strong resolution of singularities, which means that, in addition to the properties stated above, $\pi$ is an isomorphism over the smooth locus of $X$, and $\pi^{-1} (\mathrm{Sing}\,  X)$ a divisor with simple normal crossings. By the work of Hironaka \cite{hironaka}, resolutions are known to exist, and we refer the reader to  \cite{kollar} for a detailed exposition.
Let next  $D(I)$ be the derivative of $I$, which is the ideal sheaf that is generated by all derivatives of elements of $I$. Let further $Z \subset M$ be a smooth subvariety, and $\pi: B_Z M \rightarrow M$ the corresponding monoidal transformation with center $Z$ and exceptional divisor $F \subset B_ZM$. Assume that $(I,m)$ is a  marked ideal sheaf with $m \leq \mathrm{ord}_Z I$. The total transform $\pi^\ast (I)$ vanishes along $F$ with multiplicity $\mathrm{ord}_Z I$, and by removing the ideal sheaf $\mathcal{O}_{B_ZM}(m  \cdot F)$ from $\pi^\ast (I)$  we obtain the {birational, or weak transform} $\pi_\ast^{-1} (I,m)=(\mathcal{O}_{B_Z(M)}(mF) \cdot \pi^\ast (I),m) $ of $(I,m)$. Take now local coordinates $(x_1,\dots, x_n)$ on $M$ such that $Z=(x_1= \dots =x_r=0)$. As a consequence,
\bqn
y_1=\frac {x_1}{x_r}, \dots, y_{r-1}=\frac{x_{r-1}}{x_r}, y_r=x_r, \dots,  y_n=x_n
\eqn
define local coordinates on $B_Z M$, and for $(f,m) \in (I,m)$ one puts
\bqn
\pi_\ast^{-1} (f(x_1,\dots,x_n),m)= (y_r^{-m} f (y_1y_r, \dots y_{r-1}y_r, y_r, \dots, y_n),m).
\eqn
By computing the first derivatives of $\pi_\ast^{-1} (f(x_1,\dots,x_n),m)$, one then sees that for any composition  $\Pi:\tilde M \rightarrow M$ of blowing-ups of order greater or equal than $m$, 
\bq
\label{eq:derideal}
\Pi^{-1}_\ast ( D(I,m)) \subset D(\Pi^{-1}_\ast(I,m)),
\eq
see  \cite{kollar}, Sections 3.5 and 3.7. 

Consider now an oscillatory integral of the form \eqref{eq:SPT}, and its asymptotic behavior as $\mu \to +0$, in case that the critical set $\mathcal{C}$ of the phase function $\psi$ is not clean. The essential idea behind our approach  to singular  asymptotics via resolution of singularities is  to obtain a  partial monomialization 
\bqn 
\Pi^\ast( I_\psi)  \cdot  \mathcal{O}_{x,\tilde M}= {z}^{c_1}_{1}\cdots z^{c_k}_{k}  \, \Pi^{-1}_\ast (I_\psi) \cdot  \mathcal{O}_{x,\tilde M}
\eqn
 of the ideal sheaf $I_\psi=(\psi)$ generated by the phase function $\psi$ via a suitable resolution  $\Pi:\tilde M \rightarrow M$  in such a way that  the corresponding  weak transforms $\tilde \psi^ {wk}= \Pi_\ast^{-1}(\psi)$ have clean critical sets in the sense of Bott \cite{bott56}. Here   $z_1, \dots, z_k$ are local variables near each $x \in \tilde M$ and  $c_i$ are natural numbers. This  enables  one  to apply the stationary phase theorem in the resolution space $\tilde M$ to the weak transforms $ \tilde \psi^ {wk}$  with the variables  $z_1, \dots,z_k$ as parameters. Note that by Hironaka's theorem,  $I_\psi$  can always be monomialized.  But in general, this monomialization would not be explicit enough  to allow an application of the stationary phase theorem.  
 
 In the situation of the previous sections, consider the set $\mathcal{N}$ defined in \eqref{eq:calN}. To derive asymptotics for the integral \eqref{int}, we  shall construct  a strong resolution  of $\mathcal{N}$,  from which we shall deduce a partial desingularization   $\mathcal{Z}:\tilde X \rightarrow X=T^\ast M \times G$ of the set
  \bq
  \label{eq:crit}
\mathcal{C}= \mklm{(x,\xi,g) \in \Omega \times G: g\cdot (x,\xi) = (x,\xi)},
  \eq
and a partial 
monomialization of  the local ideal $I_\Phi=(\Phi)$ generated by the phase function \eqref{eq:phase} 
\bqn 
\mathcal{Z}^\ast (I_\Phi) \cdot \E_{\tilde x, \tilde X} = \prod_j\sigma_j^{l_j}   \cdot\mathcal{Z}^{-1}_\ast(I_\Phi) \cdot \E_{\tilde x, \tilde X},
\eqn
 where  $\sigma_j$ are local coordinate functions near each $\tilde x \in \tilde X$, and $l_j$ natural numbers. As a consequence, the phase function factorizes locally according to $\Phi \circ \mathcal{Z} \equiv \prod \sigma_j^{l_j} \cdot  \tilde \Phi^ {wk}$,
and we  show  that  the weak transforms $ \tilde \Phi^ {wk}$ have clean critical sets. 
Asymptotics for the integrals $I(\mu)$ are  then obtained by pulling  them back to the resolution space $\tilde X$, and applying  the stationary phase theorem to the $\tilde \Phi^{wk}$  with the variables  $\sigma_j$ as parameters.

A general description of the asymptotic behavior of oscillatory integrals  with singular critical sets was given in  \cite{bernshtein71}, and later also in \cite{duistermaat74, malgrange74, AGV88},  using Hironaka's theorem on resolution of singularities. It implies that  integrals of the form \eqref{eq:SPT} always  have local expansions of the form 
\bqn 
\sum_{\alpha}\sum_{k=0}^{n-1} c_{\alpha k}(a) \mu^\alpha (\log \mu^{-1})^k, \qquad \mu \to +0,
\eqn
where the coefficient $\alpha$ runs through a finite set of  arithmetic progressions of rational numbers, and the $c_{\alpha k}$ are distributions on $M$ with support in $\mathcal{C}$. The ocurring coefficients $\alpha$ and $k$ are determined by the so-called numerical data of the resolution, and their computation is in general a difficult task, unless one constructs an explicit resolution. 
 Resolution of singularities was first employed  in \cite{bernshtein-gelfand69, atiyah70} to give a new proof of the H\"ormander-Lojasiewicz theorem on the division of distributions and hence to the existence of temperate fundamental solutions for constant coefficient differential operators. Since many problems in analysis originate in  the singularities of some critical variety, it seems likely that an application of resolution of singularities may be relevant in further areas of this field. 

Partial desingularizations  of the zero level set $\Omega$ of the moment map and the symplectic quotient $\Omega/G$ have been obtained  e.g. in \cite{meinrenken-sjamaar} for compact symplectic manifolds with a Hamiltonian compact Lie group action by performing blowing-ups along minimal symplectic suborbifolds containing the strata of maximal depth  in $\Omega$. Recently, resolutions of group actions were also considered in \cite{albin-melrose} to study the equivariant cohomology of compact $G$-manifolds.

\section{The desingularization process}
\label{sec:DP}

We shall now proceed to  resolve the singularities of  \eqref{eq:calN}. For this,   we will have to set up an iterative desingularization process along the strata of the underlying $G$-action, where each step in our iteration will consist of a decomposition, a monoidal transformation, and a reduction. The centers of the monoidal transformations are successively chosen as isotropy bundles over unions of maximally singular orbits. For simplicity, we shall assume that at each iteration step the union of maximally singular orbits is connected. Otherwise each of the connected components, which might even have different dimensions,  has to be treated separately.

\subsection*{First decomposition} Let $M$ be a closed, connected Riemannian manifold, and $G$ a connected compact Lie group  acting on $M$ by isometries. As in the previous section, let $(H_1), \dots, (H_L)$ be the isotropy types of the $G$-action on $M$, $1 \leq k \leq L-1$, and $f_k:\nu_k\rightarrow M_k$  an invariant tubular neighborhood of $M_k(H_k)$ in 
\bdm
M_k=M-\bigcup_{i=1}^{k-1} f_i(\stackrel{\circ}{D}_{1/2}(\nu_i)),
\edm 
a manifold with corners on which $G$ acts with the isotropy types $(H_k), (H_{k+1}), \dots, (H_L)$. Here 
\bqn
f_k(p^{(k)},v^{(k)})=(\exp_{p^{(k)}} \circ \gamma ^{(k)})( v^{(k)}), \qquad  p^{(k)} \in M_k(H_k), \, v^{(k)} \in (\nu_k)_{p^{(k)}},
\eqn
 is an equivariant diffeomorphism, while 
\bqn
\gamma^{(k)}(v^{(k)})=\frac {F_k( p^{(k)})}{(1+\norm{v^{(k)}})^{1/2}} v^{(k)}, 
\eqn
where $F_k:M_k(H_k)\rightarrow \R$ is a smooth, $G$-invariant, positive function, see \cite{bredon}, pp. 306. Let $S_k=\mklm{v \in \nu_k: \norm{v}=1}\rightarrow M_k(H_k)$ be the sphere bundle over $M_k(H_k)$, and put  $W_k=f_k(\stackrel{\circ}{D_1}(\nu_k))$, $W_L = \stackrel{\circ}{M}_L$, so that 
\bqn  
M= W_1 \cup \dots \cup W_L,
\eqn
 see \eqref{eq:covering}.  Endow $G$ with the Riemannian structure
\bdm
d(X_g, Y_g)= -\tr \ad ( dL_{g^{-1}}(X_g)) \ad (dL_{g^{-1}}(Y_g)), \qquad X_g, Y_g \in T_g G,
\edm
where $L_g:h \to g h $, $h \in G$, and introduce for each $p^{(k)}\in M_k(H_k)$  the decomposition
\bqn
T_eG\simeq  \g=\g_{p^{(k)}}\oplus \g_{p^{(k)}}^\perp, 
\eqn
where $\g_{p^{(k)}}\simeq  T_e G_{p^{(k)}}$ denotes the Lie algebra of the  stabilizer $G_{p^{(k)}}$ of $p^{(k)}$, and $\g_{p^{(k)}}^\perp$ its orthogonal complement with respect to the above Riemannian structure. Note that $T_h G_{p^{(k)}} \simeq  dL_h(\g_{p^{(k)}})$, and if $A \in \g_{p^{(k)}}^\perp$, $d(dL_h(X), dL_h(A)) = -\tr \ad( X)\ad ( A)  =0$ for all $ X \in \g_{p^{(k)}}$. Therefore, the mapping
\bqn
\g_{p^{(k)}}^\perp \ni A \mapsto dL_h(A)= \frac {d}{dt} \big (h\e{tA} \big )_{|t=0} \in N_h G_{p^{(k)}}
\eqn
establishes an isomorphism $\g_{p^{(k)}}^\perp  \simeq N_h G_{p^{(k)}}$. In fact,  $\Ad(G_{p^{(k)}}) \g_{p^{(k)}}^\perp \subset \g_{p^{(k)}}^\perp$, so that $G/G_{p^{(k)}}$ constitutes a reductive homogeneous space, while the distribution $G \ni g \mapsto T_g^{{hor}} G=dL_g(\g_{p^{(k)}}^\perp)$ defines a connection on the principal fiber bundle  $G \to G/G_{p^{(k)}}$ for all $p \in M_k(H_k)$.
Consider next  the isotropy  bundle over $M_k(H_k)$
\bqn
\mathrm{Iso} \,M_k(H_k) \rightarrow M_k(H_k),
\eqn
as well as the canonical projection
\bqn 
\pi_k: W_k \rightarrow M_k(H_k), \qquad f_k(p^{(k)},v^{(k)}) \mapsto p^{(k)}, \qquad p^{(k)} \in M_k(H_k), \, v^{(k)} \in (\nu_k)_{p^{(k)}}.
\eqn
 Since $g\in G$ is an isometry, the theorem of Whitehead implies  \begin{align*}
f_k(p^{(k)},v^{(k)}) =g \cdot f_k(p^{(k)},v^{(k)})=(\exp_{gp^{(k)}} \circ \gamma_k)(g_{\ast,p^{(k)}}(v^{(k)})) \quad \Leftrightarrow \quad p^{(k)}=gp^{(k)}, \, v^{(k)}=g_{\ast,p}(v^{(k)}),
\end{align*}
so that one concludes 
\bq
\label{eq:Ncal}
\mathcal{N} \subset  \mathrm{Iso} \, W_L \cup \bigcup_{k=1}^{L-1} \pi^\ast _k \, \mathrm{Iso} \,  M_k(H_k),
\eq
where $\mathrm{Iso} \, W_L \rightarrow W_L$ is the isotropy bundle over $W_L$, and 
\bqn
\pi_k^\ast \, \mathrm{Iso}\,  M_k(H_k)=\mklm {(f_k(p^{(k)},v^{(k)}),h^{(k)})\in W_k \times G: h^{(k)} \in G_{p^{(k)}}}
\eqn
denotes the induced bundle.
Consider now an integral   $I(\mu)$ of the form  \eqref{int}. Introduce a partition of unity $ \{ \chi_k\}_{k=1, \dots, L}$ subordinated to the covering $M=W_1 \cup \dots \cup W_L$, and define
\bqn 
I_k(\mu)
= \int _{T^\ast Y}  \int_{G} e^{i\mu  \Phi(x , \xi,g) }    \chi_k (x) a (g  x,  x , \xi,g)  \d g \d(T^\ast Y)(  x,  \xi).  
\eqn
As will be explained in Lemma \ref{lemma:Reg},  the critical set of the phase function $\Phi$ is clean on the support of $\chi_L a$, so that one can directly apply the stationary phase principle to obtain asymptotics for  $I_L(\mu)$. We shall therefore turn to the case when $1 \leq k \leq  L-1$, and $W_k \cap Y \not= \emptyset$.  
 Let $\{v_1^{(k)},\dots  ,v_{c^{(k)}}^{(k)}\}$ be an orthonormal frame in $\nu_k$,  $(p_1^{(k)},\dots, p^{(k)}_{n-c^{(k)}})$ be local coordinates on $M_k(H_k)$, and write 
\bq
\label{eq:gamma}
\gamma^{(k)} ( v^{(k)})(p^{(k)},\theta^{(k)})= \sum _{i=1}^{c^{(k)}} \theta_i^{(k)} v_i^{(k)}(p^{(k)}) \in \, \gamma^{(k)}((\nu_k)_{p^{(k)}}).
\eq
By choosing $Y$ small enough, we can   assume that the coordinates in the chart $(\kappa, f_k(\nu_k) \cap Y)$ are given by  $\kappa(\exp_{p^{(k)}} \gamma^{(k)}( v^{(k)}))=(\tilde x_1, \dots, \tilde x_n)= (p_1^{(k)},\dots,p_{n-c^{(k)}}^{(k)}, \theta_1^{(k)},\dots,\theta_{c^{(k)}}^{(k)})$.  
By the  considerations leading to \eqref{eq:Ncal}, 
$$\Crit_k(\Phi) \subset \pi^\ast \mathrm{Iso} \, M_k(H_k) \times \rn_\xi,$$
where 
\begin{align*}
  \Crit_k(\Phi) &= \mklm{(x,\xi,g) \in ( \Omega \cap T^\ast (W_k \cap Y)) \times G: g \cdot (x,\xi) =(x,\xi)}.
\end{align*}
 Let therefore $U_k$ be  a tubular neighborhood of $\pi^\ast \mathrm{Iso} \, M_k(H_k)$ in $W_k \times G$, and  
$$\Pi_k: U_k \rightarrow  \pi_k^\ast \,  \mathrm{Iso} \, M_k(H_k)$$
the canonical projection which is obtained by considering geodesic normal coordinates around $\pi_k^\ast \, \mathrm{Iso} M_k(H_k)$ and   by identifying  $U_k$ with a neighborhood of the zero section in  the normal bundle $N\, \pi_k^\ast \,\mathrm{Iso} \, M_k(H_k)$. The non-stationary phase theorem \cite{hoermanderI}, Theorem 7.7.1,  then yields
\begin{align}
\label{intj}
I_k(\mu)
&= \int _{U_k}\int_ { \rn}   e^{i \mu \Phi(x , \xi,g)} \chi_k(x)  b (g  x,  x , \xi,g)  \d \xi  \d g \d M ( x)+ O(\mu^{-\infty}),
\end{align}
where $b$ is equal to the amplitude $a$ multiplied by a smooth cut-off-function with compact support in $U_k$. 
Note that the fiber of $N \, \pi^\ast \mathrm{Iso}\,  M_k(H_k)$ at a point $(f_k(p^{(k)},v^{(k)}),h^{(k)})$ may be identified with the fiber of the normal bundle to $G_{p^{(k)}}$ at the point $h^{(k)}$.  Let  now $A_1(p^{(k)}), \dots, A_{d^{(k)}}(p^{(k)})$ be an orthonormal basis of $\g_{p^{(k)}}^\perp$, and $B_1(p^{(k)}),\dots, B_{e^{(k)}}(p^{(k)})$ an orthonormal basis of $\g_{p^{(k)}}$, and introduce  canonical coordinates of the second kind
\bqn
(\alpha_1, \dots, \alpha_{d^{(k)}}, \beta_1, \dots, \beta_{e^{(k)}}) \mapsto \e{\sum_i \alpha_i A_i(p^{(k)})} \e{\sum_i \beta_i B_i(p^{(k)})}g 
\eqn
in a neighborhood of a point $g \in G$,  see \cite{helgason78}, page 146, which in turn give rise to  coordinates
\bqn
(\alpha_1, \dots, \alpha_{d^{(k)}}) \mapsto \big ( f_k(p^{(k)},v^{(k)}), \e{\sum_i \alpha_i A_i(p^{(k)})} h^{(k)}\big )
\eqn
in $\Pi_k^{-1}(f_k(p^{(k)},v^{(k)}),h^{(k)})$. Integrating along the fibers of the normal bundle to  $\pi_k^\ast \, \mathrm{Iso} M_k(H_k)$, compare \cite{donnelly78}, page 30,   we   obtain  for $I_k(\mu)$ the expression
\begin{align}\begin{split}\label{eq:19}
 I_k(\mu)&=\int_{\pi_k^\ast \, \mathrm{Iso} M_k(H_k)} \left [\int_{\Pi_k^{-1}(f_k(p^{(k)},v^{(k)}),h^{(k)})\times \rn
} e^{i\mu \Phi}\chi_k b \,  {\mathcal{J}}_k \, \d \xi \, dA^{(k)}   \right ]  dh^{(k)} \,  dv^{(k)}  dp^{(k)} \\
 &=\int_{ M_k(H_k) }\left [\int_{  \pi_k^{-1}(p^{(k)})\times G_{p^{(k)}} \times \stackrel{\circ}{D}_ \iota(\g^\perp_{p^{(k)}}) \times\rn
} e^{i\mu \Phi} \chi_k b \,  {\mathcal{J}}_k \, \d \xi \, dA^{(k)} \, dh^{(k)} \, dv^{(k)}  \right ]      dp^{(k)},
\end{split}
\end{align}
up to a term of order $O(\mu^{-\infty})$, where $dp^{(k)}, dv^{(k)}, dh^{(k)}, dA^{(k)}$ are suitable volume densities on the sets $M_k(H_k),\,(\nu_k)_{p^{(k)}}, \,G_{p^{(k)}}, \, \g_{p^{(k)}}^\perp\simeq  N_{h^{(k)}} G_{p^{(k)}} $, respectively, and 
\bq
\label{eq:20}
(p^{(k)}, v^{(k)},  A^{(k)},h^{(k)})\mapsto (f_k(p^{(k)}, v^{(k)}) ,\e{A^{(k)}} h^{(k)} )=(x,g)
\eq
are coordinates on $ U_k $ such that $ \d g \d M( x)\equiv {\mathcal{J}}_k \,  dA^{(k)} \, dh^{(k)} \,  dv^{(k)} \, dp^{(k)}$, $ {\mathcal{J}}_k$ being a Jacobian. Here $\stackrel{\circ}{D}_\iota (\g^\perp_{p^{(k)}})$ denotes the interior of a ball of suitable radius $\iota >0$ around the origin in $\g^\perp_{p^{(k)}}$.

\subsection*{First monoidal transformation} We shall now sucessively resolve the singularities of \eqref{eq:calN}. To begin with, note that  \eqref{eq:Ncal} implies
\bqn 
\mathcal{N}  = \mathrm{Iso} \, W_L  \cup \bigcup_{k=1}^{L-1} \mathcal{N} \cap U_k,
\eqn
and we put $\mathcal{N}_L=\mathrm{Iso } \, W_L$, $\mathcal{N}_k=\mathcal{N} \cap U_k$. While   $\mathcal{N}_L$ is  a smooth submanifold, 
$\mathcal{N}_k$ is in general singular. In particular, if $\dim H_k\not=\dim H_L$,
 $\mathcal{N}_k$ has a  singular locus  given by $\mathrm{Iso} \, M_k(H_k)$. 
We shall  therefore perform for each $k \in \mklm{1,\dots,L-1}$  a monoidal transformation 
\bqn 
\zeta_k: B_{Z_k}( U_k) \longrightarrow U_k 
\eqn
 with center $Z_k= \mathrm{Iso} \, M_k(H_k)\subset \mathcal{N}_k$. By piecing these transformations together, we obtain the monoidal transformation
 \bqn 
 \zeta^{(1)}: B_{Z^{(1)}} ( \M) \longrightarrow \M, \qquad Z^{(1)} = \bigcup _{k=1}^{L-1} Z_k \qquad \mbox{\emph{(disjoint union)}}.
 \eqn
 To get a local description, let $k$ be fixed, and write  $ A^{(k)}(p^{(k)},\alpha^{(k)})=\sum  \alpha_i^{(k)} A_i^{(k)}(p^{(k)})\in \g^\perp_{p^{(k)}}$,  $ B^{(k)}(p^{(k)},\beta^{(k)})=\sum  \beta_i^{(k)} B_i^{(k)}(p^{(k)})\in \g_{p^{(k)}}$.
 With respect to these coordinates and the ones introduced in \eqref{eq:gamma}  and \eqref{eq:20} we have $Z_k\simeq\mklm{T^{(k)}=(\theta^{(k)},\alpha^{(k)})=0}$, so that 
\begin{gather*}
B_{Z_k}( U_k)=\mklm{ (x,g,[t]) \in U_k \times \mathbb{RP}^{c^{(k)} + d^{(k)}-1}:  T^{(k)}_i t_j = T^{(k)}_j t_i  },\\
\zeta_k: (x,g,[t]) \longmapsto (x,g).
\end{gather*}
If $t_\rho \not= 0$,
\bqn 
(x,g,[t]) \mapsto \Big ( p^{(k)}, h^{(k)}, \frac{t_1}{t_\rho}, \dots \, , \hat{.} \, , \dots \, , \frac{t_{c^{(k)}+ d^{(k)}}}{t_\rho}, T^{(k)}_\rho\Big )
\eqn
define local coordinates on $B_{Z_k}(U_k)$. Consequently, setting $V_\rho=\mklm{[t] \in  \mathbb{RP}^{c^{(k)} + d^{(k)}-1}: t_\rho\not=0}$,  we can cover $B_{Z_k}( U_k)$ with charts $\mklm{(\phi^\rho_k, \mathcal{O}^\rho_k)}$, where $\mathcal{O}^\rho_k=B_{Z_k}( U_k)\cap (U_k \times V_\rho)$, such that   $\zeta_k$ is realized in each of the $\theta^{(k)}$-charts $\mklm{\mathcal{O}^\rho_k}_{1\leq \rho\leq c^{(k)}}$ as
\begin{align}
\label{eq:21}\begin{split}
\zeta^\rho_k&=\zeta_k \circ (\phi^\rho_k)^{-1}: ( p^{(k)},\tau_k, \tilde v^{(k)},  A^{(k)}, h^{(k)}) \stackrel{\, ' \zeta^\rho_k} {\mapsto} ( p^{(k)},\tau_k \tilde v^{(k)}, \tau_k A^{(k)}, h^{(k)})   \\ 
& \mapsto (\exp_{p^{(k)}}   \tau_k \tilde v^{(k)}, \e{\tau_k A^{(k)}} h^{(k)})=(x,g),
\end{split}
\end{align}
where $
 \tilde v^{(k)}(p^{(k)},\theta^{(k)})
  \in  \gamma^{(k)} \big (  ( S_k^+)_{p^{(k)}} \big )$,
and 
$ S_k^+=\mklm{v \in \nu_k: v = \sum s_i v_i^{(k)}, s_\rho>0,  \norm{v}=1}$, while $\tau_k \in (-1,1)$. Note that for each $1 \leq \rho\leq c^{(k)}$ we have $W_k \simeq  S_k^+ \times (-1,1)$ up to a set of measure zero. A similar description of $\zeta_k$ is given in the $\alpha^{(k)}$-charts.  As a consequence, we obtain a partial monomialization of the inverse image ideal sheaf $( \zeta^{(1)})^\ast (I_\mathcal{N})$
 \bqn 
 ( \zeta^{(1)})^\ast (I_\mathcal{N}) \cdot \E_{(x,g,[t]),B_{Z^{(1)}} ( \M)}=  \tau_k \cdot ( \zeta^{(1)})^{-1}_\ast (I_\mathcal{N}) \cdot  \E_{(x,g,[t]),B_{Z^{(1)}} ( \M)}
  \eqn 
in a neighborhood of any point $(x,g,[t])\in B_{Z^{(1)}} ( \M)$.  To see this, note that  $I_\mathcal{N}$ is generated locally by the functions  $\tilde x_q( x)-\tilde x_q(g \cdot x)$, $1\leq q \leq n$. We have   $ g \cdot \exp_{p^{(k)}}   \tau_k \tilde v^{(k)}= \exp_{g\cdot p^{(k)}} [g_{\ast,p^{(k)}} ( \tau_k  \tilde v^{(k)})]$,  where $g_{\ast,p^{(k)}} (  \tilde v^{(k)}) \in \gamma^{(k)} (( \nu_k)_{gp^{(k)}} )$, $\nu_k$ being a $G$-vector bundle. Now, Taylor expansion at $\tau_k=0$ gives for  $ y \in Y \cap f_k(\nu_k)$
\bqn
\tilde x_q(\e{\tau_k A^{(k)}} \cdot y)=\tilde x_q(y)- \tau_k\widetilde A^{(k)}_{y} ( \tilde x_q)+ O(|\tau_k^2 A^{(k)}|),
\eqn
where $\tau_k \in (-1,1), A^{(k)} \in \stackrel \circ D_\iota (\g_{p^{(k)}}^\perp)$, and $\iota>0$ is assumed to be sufficiently small. Furthermore, 
 $\widetilde A^{(k)}_{y} ( \tilde x_q)= d\tilde x_q(\widetilde A^{(k)}_{y})$. Consequently,
\begin{align}
\label{eq:kappafact}
\begin{split}
&\kappa\big (\exp_{p^{(k)}} \tau_k \tilde v^{(k)}\big ) - \kappa \big (\e{\tau_k A^{(k)}} h^{(k)} \cdot \exp_{p^{(k)}} \tau_k \tilde v^{(k)}\big ) \\ = \tau_k \Big (\tilde A^{(k)}_{p^{(k)}}(\tilde x_1)&, \dots,   \theta_1^{(k)}(\tilde v^{(k)}) - \theta_1^{(k)}\big  ( ( h^{(k)} )_{\ast,p^{(k)}} \tilde v^{(k)}, \dots  \big ) + O(|\tau_k^2 A^{(k)}|).
\end{split}
\end{align}
Since similar considerations hold  in the $\alpha^{(k)}$-charts $\mklm{\mathcal{O}^\rho_k}_{c^{(k)}+1 \leq \rho \leq c^{(k)}+d^{(k)}}$,  the assertion follows. In the same way, the phase function  \eqref{eq:phase}  factorizes  according to 
\bq
\label{eq:22}
 \Phi \circ (\id_\xi \otimes \zeta_k^\rho) =  \,^{(k)} \tilde \Phi^{tot}=\tau_k \cdot  \,  ^{(k)} \phw, 
 \eq
$ \,^{(k)} \tilde \Phi^{tot}$ and $  \,  ^{(k)} \phw $ being the \emph{total} and \emph{weak transform} of the phase function $\Phi$, respectively.\footnote{ Note that the weak transform is defined only locally, while  the total transform has a global meaning. To keep the notation as simple as possible, we restrained ourselves from making the chart dependence of $\tau_k$ and $ \,  ^{(k)} \phw$ manifest. } In the  $\theta^{(k)}$-charts this explicitly reads  
\begin{align}
\begin{split}
\label{eq:F1}
&\Phi(x, \xi, g) =\eklm{\kappa\big (\exp_{p^{(k)}} \tau_k \tilde v^{(k)}\big ) - \kappa \big (\e{\tau_k A^{(k)}} h^{(k)} \cdot \exp_{p^{(k)}} \tau_k \tilde v^{(k)}\big ), \xi }\\
  =\tau_k  &\left [  \sum_{q=1}^{n-c^{(k)}}  \xi_q \, dp^{(k)}_q(\widetilde A^{(k)}_{p^{(k)}})  + \sum_{r=1}^{c^{(k)}} \Big [  \theta_r^{(k)}(\tilde v^{(k)})-  \theta_r^{(k)}\big  ( (h^{(k)} )_{\ast,p^{(k)}} \tilde v^{(k)} \big )  \Big ]\xi_{n-c^{(k)}+r} +O(|\tau_kA^{(k)}|)\right ].
\end{split}
\end{align}
Since $\zeta_k$ is a real analytic, surjective map, we can lift the integral $I_k(\mu)$ to the resolution space $B_{Z_k}(U_k)$, and introducing a partition $\mklm{u^\rho_k}$ of unity subordinated to the covering $\mklm{\mathcal{O}^\rho_k}$  yields with \eqref{intj} the equality
\bqn 
I_k(\mu)=\sum_{\rho=1} ^{c^{(k)}}  I_k^\rho(\mu)+\sum_{\rho=c^{(k)}+1} ^{d^{(k)}}  \tilde I_k^\rho(\mu)
\eqn 
up to terms of order $O(\mu^{-\infty})$, 
where the integrals  $  I_k^\rho (\mu)$ and $  \tilde I_k^\rho (\mu)$ are given by the expressions
\begin{gather*}
\int_{B_{Z_k}(U_k)\times \rn} u_k^\rho ( \id_\xi \otimes \zeta_k)^\ast (  e^{i\mu \Phi } \chi_k b \d g \d M(x) \d \xi).
\end{gather*}
As we shall see, the weak transforms $\, ^{(k)} \phw$ have no critical points in the $\alpha^{(k)}$-charts, which will imply that the integrals $  \tilde I_k^\rho(\mu)$ contribute to $I(\mu)$ only with lower order terms. In what follows, we shall therefore restrict ourselves to the the examination of the integrals $I_k^\rho(\mu)$. Setting $a_k^\rho = (u_\rho \circ (\phi_k^\rho)^{-1})  \cdot [ (b \chi_k) \circ ( \id_\xi \otimes \zeta_k^\rho)]$
 we obtain with \eqref{eq:19} and \eqref{eq:21}  
\begin{gather*}
I_k^\rho(\mu)= \int_{ M_k(H_k) \times (-1,1) }\Big [\int_{ \gamma^{(k)}((S_k)_{p^{(k)}}) \times  G_{p^{(k)}} \times  \stackrel{\circ}{D}_ \iota(\g_{p^{(k)}}^\perp) \times\rn
} e^{i\mu{\tau_k}\,  ^{(k)} \phw}a_k^\rho \, \bar  {\mathcal{J}}_k^\rho \\   \d \xi \, dA^{(k)} \, dh^{(k)} \, d\tilde v^{(k)}   \Big ]  \d \tau_k \,     \d p^{(k)},
\end{gather*}
 where   $d\tilde v^{(k)}$ is a suitable volume density on $\gamma^{(k)}( (S_k)_{p^{(k)}})$ such that the pulled back density reads $(\zeta_k^\rho)^\ast(\d g \d M( x)) = \bar  {\mathcal{J}}_k^\rho \, d A^{(k)} \, dh^{(k)} \,  d\tilde v^{(k)} \, d\tau_k \, dp^{(k)}$. Furthermore, by compairing \eqref{eq:20} and \eqref{eq:21}  one sees  that  
\bqn 
\bar  {\mathcal{J}}_k^\rho  = |\tau_k|^{c^{(k)}+d^{(k)}-1} \,  {\mathcal{J}}_k\circ  \,  ' \zeta_k^\rho . 
\eqn

\subsection*{First reduction} Let $k$ be fixed, and assume that there exists a $x \in W_k$ with isotropy group $G_x \sim H_j$, and let  $p^{(k)} \in M_k(H_k), v^{(k)} \in (\nu_k)_{p^{(k)}}$ be such that $x=f_k(p^{(k)},v^{(k)})$. Since we can assume that $x$ lies in a slice at $p^{(k)}$ around the $G$-orbit of $p^{(k)}$, we have $G_x \subset G_{p^{(k)}}$, see \cite{kawakubo}, pp. 184, and  \cite{bredon}, page 86. Hence $H_j$ must be conjugate to a subgroup of $H_k\sim G_{p^{(k)}}$. Now, $G$ acts on $M_k$ with the isotropy types $(H_k),(H_{k+1}), \dots, (H_L)$. The isotropy types occuring in $W_k$ are therefore those for which the corresponding isotropy groups  $H_k,H_{k+1}, \dots, H_L$ are conjugate to a subgroup of $H_k$, and we shall denote them by
$
(H_k) = (H_{l_1}), (H_{l_2}), \dots, (H_L).
$
 By the invariant tubular neighborhood theorem, one has the isomorphism
\bqn
W_k/G \simeq (\nu_k)_{p^{(k)}}/ G_{p^{(k)}}
 \eqn
for every $p^{(k)}\in M_k(H_k)$. Furthermore,  $ (\nu_k)_{p^{(k)}}$ is an orthogonal $ G_{p^{(k)}}$-space; therefore $G_{p^{(k)}}$ acts on  $(S_k)_{p^{(k)}}$ with   isotropy types $(H_{l_2}), \dots, (H_L)$, cp.  \cite{donnelly78}, pp. 34, and   $G$ must act on $S_k$ with  isotropy types $(H_{l_2}), \dots, (H_L)$ as well. If all isotropy groups  $H_{l_2}, \dots, H_L$ have the same dimensions, the singularities of  $\mathcal{N}_{k}$   have been resolved. Indeed, note that $\zeta_k^{-1}( \mathcal{N}_k )$ is contained in the union of the $\theta^{(k)}$-charts  $\mklm{\mathcal{O}^\rho_k}_{1\leq \rho\leq c^{(k)}}$ since, in the notation of \eqref{eq:20}, $
\e{A^{(k)}} h^{(k)} \in G_{f_k(p^{(k)}, v^{(k)})} \subset G_{p^{(k)}}$ 
necessarily implies $A^{(k)}=0$. Let therefore $1 \leq \rho \leq c^{(k)}$, and consider  the set $\zeta_k^{-1}( \mathcal{N}_{k}) \cap \mathcal{O}_{k}^\rho$, which  is given by all points $(x,g,[t])$ with coordinates $( p^{(k)}, \tau_k, \tilde v^{(k)}, A^{(k)}, h^{(k)})$ satisfying
 \bqn 
 \e{\tau_kA^{(k)}} h^{(k)}\in G_{\exp_{p^{(k)}}\tau_k \tilde v^{(k)}}\subset G_{p^{(k)}}.
 \eqn
 If $\tau_k\not=0$, this implies $A^{(k)}=0$ and $h^{(k)} \in G_{\tilde v^{(k)}}$. Therefore
\begin{align*}
\zeta_k^{-1}( \mathcal{N}_{k}) \cap \mathcal{O}_{k}^\rho=\mklm{ A^{(k)}=0, \, h^{(k)} \in G_{\tilde v^{(k)}}, \tau_k\not=0 }\cup \mklm{\tau_k=0}.
\end{align*}
Assume now that  all isotropy groups  $H_{l_2}, \dots, H_L$ have the same dimension. 
 If  $H_k$ has the  same dimension, too, $\mathcal{N}_k$ is already a manifold. Otherwise, the invariant tubular neighborhood theorem implies that $\zeta_k^{-1}( \mathrm{Reg }\, \mathcal{N}_{k}) \cap \mathcal{O}_{k}^\rho=\mklm{ A^{(k)}=0, \, h^{(k)} \in G_{\tilde v^{(k)}}, \tau_k \not=0 }$, where $\mathrm{Reg}\,  \mathcal{N}_k= \mathrm{Reg} \, \mathcal{N} \cap U_k$ denotes the regular part of $\mathcal{N}_k$. The closure of this set is  a smooth manifold, and taking the union over all $1 \leq \rho \leq c^{(k)}$  yields a smooth manifold $ \tilde {\mathcal{ N}_k}\subset  B_{Z_k}( U_k) $ which intersects $\zeta^{-1}_k(\mathrm{Sing} \, \mathcal{N}_k)$ normally.   After performing an additional monomial transformation with center $\tilde {\mathcal{ N}_k}\cap \zeta^{-1}_k(\mathrm{Sing} \, \mathcal{N}_k)$, we obtain a strong resolution for $\mathcal{ N}_k$. Furthermore, if $G$ acts on $S_k$ only with isotropy type $(H_L)$,  we shall  see in Sections  \ref{sec:MT1} and \ref{sec:MT2} that in each of the $\theta^{(k)}$-charts the critical sets of the weak transforms $^{(k)} \phw$ are clean, so that one can apply the stationary phase theorem  in order to compute each of the $I_k^\rho(\mu)$. But in general, $G$ will act on $S_k$ with singular orbit types, so that  neither $\mathcal{N}_k$ is   resolved, nor do the weak transforms $^{(k)} \tilde \Phi^{wk}$ have clean critical sets, and we are forced to continue with the iteration.

\subsection*{Second Decomposition}

In what follows, let  $1\leq k \leq L -2$, and $p^{(k)}\in M_k(H_k)$ be fixed. Since $\gamma^{(k)}: \nu_k \rightarrow \nu_k$ is an equivariant diffeomorphism onto its image, $\gamma^{(k)}((S_k)_{p^{(k)}})$ is a compact $G_{p^{(k)}}$-manifold, and we consider the covering 
\bqn
\gamma^{(k)}((S_k)_{p^{(k)}} )=W_{kl_2} \cup \dots \cup W_{kL}, \qquad W_{kl_j}= f_{kl_j}(\stackrel \circ D_1(\nu_{kl_j})), \quad W_{kL}= \mathrm{Int} (\gamma^{(k)}((S_k)_{p^{(k)}})_L),
\eqn
where $f_{kl_j}:\nu_{kl_j} \rightarrow\gamma^{(k)} ((S_k)_{p^{(k)}})_{l_j}$ is an invariant tubular neighborhood of $ \gamma^{(k)}((S_k)_{p^{(k)}})_{l_j}(H_{l_j})$ in 
\bqn 
\gamma^{(k)}((S_k)_{p^{(k)}})_{l_j}= \gamma^{(k)}((S_k)_{p^{(k)}}) - \bigcup_{r=2}^{j-1} f_{kl_r}(\stackrel \circ D_{1/2}(\nu_{kl_r})), \qquad j \geq 2,
\eqn
and $f_{kl_j}(p^{(l_j)},v^{(l_j)})=(\exp_{p^{(l_j)}} \circ \gamma^{(l_j)})(v^{(l_j)})$, $p^{(l_j)} \in\gamma^{(k)} ((S_k)_{p^{(k)}})_{l_j} (H_{l_j})$, $ v ^{(l_j)} \in ( \nu_{kl_j}) _{p^{(l_j)}}$, $\gamma^{(l_j)}: \nu_{kl_j} \rightarrow \nu_{kl_j}$ being an equivariant diffeomorphism onto its image given by 
\bqn
\gamma^{(l_j)}(v^{(l_j)})=\frac {F_{l_j}( p^{(l_j)})}{(1+\norm{v^{(l_j)}})^{1/2}} v^{(l_j)}, 
\eqn
where $F_{l_j}:((S_k)_{p^{(k)}})_{l_j} (H_{l_j})\rightarrow \R$ is a smooth, $G_{p^{(k)}}$-invariant, positive function.
Let now $\{\chi_{kl_j}\}$ denote a partition of  unity subordinated to the covering $\mklm{W_{kl_j}}$, which extends to a partition of unity on $\gamma^{(k)}(S_k)$ as a consequence of the  invariant tubular neighborhood theorem, by which in particular $\gamma^{(k)}(S_k) /G \simeq \gamma^{(k)}((S_k)_{p^{(k)}})/G_{p^{(k)}}$ for all $p^{(k)}$. We then  define
\begin{align}
\label{eq:Ikij}
\begin{split}
I_{kl_j}^\rho(\mu) =& \int_{ M_k(H_k) \times (-1,1)  }\Big [\int_{\gamma^{(k)}( (S_k)_{p^{(k)}}) \times  G_{p^{(k)}} \times  \stackrel{\circ}{D}_ \iota(\g_{p^{(k)}}^\perp) \times\rn
} e^{i\mu{\tau_k} \,  ^{(k)} \phw} a_k^\rho \\ & \chi_{kl_j} \bar {\mathcal{J}}_k^\rho \d \xi \, dA^{(k)} \, dh^{(k)} \, d\tilde v^{(k)}   \Big ]  \d \tau_k \,     \d p^{(k)},
\end{split}
\end{align}
so that $I_k^\rho(\mu)= I_{kl_2}^\rho(\mu) + \dots + I_{kL}^\rho (\mu)$. 
Since $G_{p^{(k)}}$ acts on $W_{kL}$ only with type $(H_L)$, the iteration process for $I_{kL}^\rho (\mu)$ ends here. For the remaining integrals $I_{kl_j}^\rho(\mu)$ with $k < l_j < L$ and non-zero integrand, let us denote by
\bqn
\mathrm{Iso} \,\gamma^{(k)}((S_k)_{p^{(k)}})_{l_j}(H_{l_j}) \rightarrow \gamma^{(k)} ((S_k)_{p^{(k)}})_{l_j}(H_{l_j})
\eqn
the isotropy bundle over $\gamma^{(k)}((S_k)_{p^{(k)}})_{l_j}(H_{l_j})$, and by $\pi_{kl_j}: W_{kl_j} \rightarrow \gamma^{(k)}((S_k)_{p^{(k)}})_{l_j}(H_{l_j})$ the canonical projection.   We then assert that in each $\theta^{(k)}$-chart  $\mklm{\mathcal{O}^\rho_k}_{1\leq \rho\leq c^{(k)}}$ 
\begin{align}
\begin{split}
\label{eq:A}
\Crit_{kl_j}( \,^{(k)} \phw) & \\ \subset \Big \{ (p^{(k)}, \tau_k,\tilde v^{(k)},\xi, h^{(k)}, A^{(k)}): (\tilde v^{(k)}, h^{(k)}) \in& \pi_{kl_j}^\ast \mathrm{Iso} \,\gamma^{(k)}((S_k)_{p^{(k)}})_{l_j}(H_{l_j}), \quad A^{(k)} =0 \Big \},
\end{split}
\end{align}
where 
$$\Crit_{kl_j}( \,^{(k)} \phw)=\mklm{ (p^{(k)}, \tau_k,\tilde v^{(k)},\xi, h^{(k)}, A^{(k)}): \, ^{(k)} \phw_{\ast} = 0,\, \tilde v^{(k)} \in W_{kl_j}}.$$
Indeed,  from \eqref{eq:22} it is clear that for $\tau_k\not=0$ the condition $\gd_\xi  \,^{(k)} \phw=0$ is equivalent to 
\bqn
 \e{\tau_k\sum \alpha_i^{(k)} A_i^{(k)}(p^{(k)})} h^{(k)}\in G_{\exp_{p^{(k)}}\tau_k \tilde v^{(k)}}\subset G_{p^{(k)}},
\eqn
which implies $\alpha^{(k)}=0$, and consequently $h^{(k)} \in G_{\tilde v^{(k)}}$. But if $\tilde v^{(k)}=f_{kl_j}(p^{(l_j)}, v)$, where $p^{(l_j)} \in\gamma^{(k)}((S_k)_{p^{(k)}})_{l_j}(H_{l_j})$, $v \in \stackrel \circ D_1(\nu_{kl_j})_{p^{(l_j)}}$, then  $h^{(k)} \in G_{p^{(l_j)}}$.
On the other hand, assume that $\tau_k=0$. By \eqref{eq:F1},  the vanishing of the $\xi$-derivatives of $\,^{(k)} \phw$ is equivalent to
\bqn
 \big (\widetilde A^{(k)}_{ p^{(k)}}(p^{(k)}_1), \dots, \widetilde A^{(k)}_{ p^{(k)}}(p^{(k)}_{n-c^{(k)}}) \big )=0, \qquad  (\1 - h^{(k)} )_{\ast,p^{(k)}} \,\tilde v^{(k)}=0,
\eqn
which again implies $\alpha^{(k)}=0$, as well as $h^{(k)} \in G_{\tilde v^{(k)}}$. But if $\tilde v^{(k)}=f_{kl_j}(p^{(l_j)}, v)$ as above, we again conclude $h^{(k)} \in G_{p^{(l_j)}}$, and \eqref{eq:A} follows, since
\begin{align*}
 \pi_{kl_j}^\ast \mathrm{Iso} \,\gamma^{(k)}((S_k)_{p^{(k)}})_{l_j}(H_{l_j})=&\{(w, g) \in W_{kl_j} \times G_{p^{(k)}}: w=f_{kl_j}(p^{(l_j)}, v), \\ & p^{(l_j)} \in\gamma^{(k)}((S_k)_{p^{(k)}})_{l_j}(H_{l_j}), \, v \in \stackrel \circ D_1(\nu_{kl_j})_{p^{(l_j)}}, \,g \in G_{p^{(l_j)}}\}.
\end{align*}
The same reasoning also shows that the weak transforms $^{(k)} \tilde \Phi^{wk}$ can have no critical points in the $\alpha^{(k)}$-charts  $\mklm{\mathcal{O}^\rho_k}_{c^{(k)}+1 \leq \rho\leq c^{(k)}+d^{(k)}}$ .
Let now  $U_{kl_j}$ denote a tubular neighborhood of the set $\pi_{kl_j}^\ast \mathrm{Iso} \,\gamma^{(k)}((S_k)_{p^{(k)}})_{l_j}(H_{l_j})$ in $W_{kl_j} \times G_{p^{(k)}}$, and let $b_k^\rho$ be equal to the product of the amplitude $a_k^\rho$ with some smooth cut-off-function with compact support in $U_{kl_j}$ that depends smoothly on $p^{(k)}$. The non-stationary phase theorem then implies that, up to terms of lower order, we can replace $a_k^\rho$ by $b_k^\rho$  in \eqref{eq:Ikij}, compare Section \ref{sec:MR}.  For given  $p^{(l_j)} \in\gamma^{(k)}((S_k)_{p^{(k)}})_{l_j}(H_{l_j})$,  consider next the decomposition
\bqn
\g = \g_{p^{(k)}} \oplus \g_{p^{(k)}}^\perp =(\g_{p^{(l_j)}}\oplus \g_{p^{(l_j)}}^\perp) \oplus \g_{p^{(k)}}^\perp.
\eqn
Let further $h^{(l_j)} \in G_{p^{(l_j)}}$, and $A_1^{(l_j)}, \dots ,A_{d^{(l_j)}}^{(l_j)}$ be an orthonormal frame in $ \g_{p^{(l_j)}}^\perp$,  as well as $B_1^{(l_j)}, \dots ,B_{e^{(l_j)}}^{(l_j)}$ be an orthonormal frame in $ \g_{p^{(l_j)}}$, and $v_1^{(kl_j)}, \dots, v_{c^{(kl_j)}}^{(kl_j)}$ an orthonormal frame in $(\nu_{kl_j}) _{p^{(l_j)}}$. Integrating along the fibers in a neighborhood of $\pi_{kl_j}^\ast \mathrm{Iso} \,\gamma^{(k)}((S_k)_{p^{(k)}})_{l_j}(H_{l_j})$ then yields  for $ I_{kl_j}^\rho(\mu)$ the expression
\begin{align*}
I_{kl_j}^\rho(\mu) &=  \int_{M_k(H_k)\times (-1,1) } \Big [ \int_{\gamma^{(k)}((S_k)_{p^{(k)}})_{l_j}(H_{l_j})} \Big [ \int_{ \pi^{-1}_{kl_j}(p^{(l_j)}) \times G_{p^{(l_j)}}\times  \stackrel{\circ}{D}_ \iota(\g_{p^{(l_j)}}^\perp) \times \stackrel{\circ}{D}_ \iota(\g_{p^{(k)}}^\perp) \times  \rn} e^{i{\mu \tau_k} \, ^{(k)} \phw } \\ &b_k^\rho  \, \chi_{kl_j} \,  {\mathcal{J}}_{kl_j}^\rho \, \d \xi \, d A^{(k)} \, d A^{(l_j)} \, dh^{(l_j)} \,  dv^{(l_j)}     \big ] dp^{(l_j)}  \Big ]  d\tau_k \,  dp^{(k)}
\end{align*}
up to lower order terms, where ${\mathcal{J}}_{kl_j}^\rho$ is a Jacobian, 
and 
\bqn
(p^{(l_j)}, v^{(l_j)}, A^{(l_j)},h^{(l_j)})\mapsto (f_{kl_j}(p^{(l_j)}, v^{(l_j)}),\e{A^{(l_j)}}h^{(l_j)})=(\tilde v^{(k)},h^{(k)})
\eqn
are coordinates on $U_{kl_j}$, while $dp^{(l_j)}$, $dA^{(l_j)}, dh^{(l_j)}$, and $ dv^{(l_j)} $ are suitable volume densities in the spaces   $\gamma^{(k)}((S_k)_{p^{(k)}})_{l_j}(H_{l_j})$,  $\g_{p^{(l_j)}}^\perp$, $G_{p^{(l_j)}}$, and $\stackrel \circ D_1(\nu_{kl_j})_{p^{(l_j)}}$, respectively, such that we have the equality $ \bar  {\mathcal{J}}_k^\rho \d h^{(k)} \d \tilde v^{(k)}\equiv {\mathcal{J}}_{kl_j}^\rho \,  dA^{(l_j)} \, dh^{(l_j)} \,  dv^{(l_j)}\, dp^{(l_j)} $.

\subsection*{Second monoidal transformation}

Put $\tilde M^{(1)}=B_{Z^{(1)}}(\M)$, and consider   the  monoidal transformation 
\bqn 
\zeta^{(2)}: B_{Z^{(2)}} (\tilde M^{(1)}) \longrightarrow \tilde M^{(1)}, \qquad Z^{(2)}= \bigcup_{k<l<L,  \, (H_l) \leq (H_k)} Z_{kl} \qquad \mbox{\emph{(disjoint union)}},
\eqn
where
\bqn
Z_{kl}\simeq    \bigcup _{p^{(k)} \in M_k(H_k)} (-1,1)\times \mathrm{Iso} \,  \gamma^{(k)} ( (S_k)_{p^{(k)}})_l (H_l), \qquad k< l< L, \quad  (H_l) \leq (H_k),
\eqn
are  the possible maximal singular loci of $(\zeta^{(1)})^{-1}(\mathcal{N})$. To obtain a local description of $\zeta^{(2)}$, let us write   $ A^{(l)}(p^{(k)},p^{(l)},\alpha^{(l)})=\sum  \alpha_i^{(l)} A_i^{(l)}(p^{(k)}, p^{(l)})\in \g_{p^{(l)}}^\perp$,   $ B^{(l)}(p^{(k)},p^{(l)},\beta^{(l)})=\sum  \beta_i^{(l)} B_i^{(l)}(p^{(k)}, p^{(l)})\in \g_{p^{(l)}}$, as well as
\bqn
 \gamma^{(l)}(v^{(l)})(p^{(k)}, p^{(l)},\theta^{(l)})= \sum _{i=1}^{c^{(l)}} \theta_i^{(l)} v_i^{(kl)}(p^{(k)},p^{(l)})\in \gamma^{(l)}((\nu_{kl})_{p^{(l)}}).
\eqn
One has $Z_{kl}\simeq\mklm{\alpha^{(k)}=0, \, \alpha^{(l)}=0, \, \theta^{(l)}=0}$, which in particular shows that each $Z_{kl}$ is a manifold. If we now cover $B_{Z^{(2)}} (\tilde M^{(1)})$ with the standard charts, a computation shows that $(\zeta^{(1)} \circ \zeta^{(2)})^{-1}( \mathcal{N})$ is contained in the $(\theta^{(k)},\theta^{(l)})$-charts. For our purposes, it will therefore  suffice to examine $\zeta^{(2)}$ in each of these charts in which it reads
\begin{align}
\begin{split}
\label{eq:25}
\zeta_{kl}^{\rho\sigma}: (p^{(k)},\tau_k, p^{(l)}, \tau_l, \tilde v^{(l)}, A^{(l)},  h^{(l)},  A^{(k)})  \stackrel{\, '\zeta_{kl}^{\rho\sigma}}{\mapsto}  (p^{(k)},\tau_k, p^{(l)}, \tau_l  \tilde v^{(l)},\tau_l  A^{(l)},  h^{(l)}, \tau_l A^{(k)}) 
\\ \mapsto (p^{(k)},\tau_k, \exp_{p^{(l)}} \tau_l \tilde v^{(l)},  \e{\tau_l  A^{(l)}} h^{(l)}, \tau_l A^{(k)})\equiv(p^{(k)},\tau_k, \tilde v^{(k)},h^{(k)}, A^{(k)}),
\end{split}
\end{align}
where $\tau_l \in (-1,1)$, and 
\bqn
\tilde v^{(l)}(p^{(k)},p^{(l)},\theta^{(l)}) \in  \gamma^{(l)} ((S_{kl}^+)_{p^{(l)}}). 
\eqn
Here $S_{kl}$ stands for the 
the sphere subbundle in $\nu_{kl}$,  and  $S_{kl}^+=\mklm { v \in S_{kl}:  v=\sum s_i v_i^{(kl)}, \, v_\sigma>0 }$
for some $\sigma$. Note that $Z_{kl}$ has normal crossings with the exceptional divisor $E_k=\zeta_k^{-1}(Z_k) \simeq \mklm{\tau_k=0}$, and that for each $p^{(k)}\in M_k(H_k)$ we have 
$W_{kl} \simeq S_{kl}^+ \times (-1,1)$, up to a set of measure zero. Now, Taylor expansion at $\tau_l=0$ gives
\begin{gather*}
^T \theta^{(k)}\big ( \exp _{p^{(l)}} \tau_l  \tilde v ^{(l)} \big )- ^T \theta^{(k)}\big ( (  \e{\tau_l A^{(l)}} h^{(l)})_{\ast, p^{(k)}} \exp _{p^{(l)}} \tau_l  \tilde v ^{(l)} \big ) \\ 
= \tau_l \frac \gd {\gd \tau_l} \Big [ \, ^T \theta^{(k)}\big ( \exp _{p^{(l)}} \tau_l  \tilde v ^{(l)} \big )- \, ^T\theta^{(k)}\big ( (  \e{\tau_l A^{(l)}} h^{(l)})_{\ast, p^{(k)}} \exp _{p^{(l)}} \tau_l  \tilde v ^{(l)} \big )\Big ]_{|\tau_l=0} \\ 
+ O(|\tau_l ^2 \, A^{(l)}|)+O\big (\big |\tau_l^2  \big [\theta^{(l)} ( \tilde v ^{(l)}) - \theta^{(l)} \big ((h^{(l)})_{\ast, p^{(k)}} \tilde v ^{(l)}\big )\big ]\big | \big )\\
= \tau_l \left ( \frac {\gd \theta ^{(k)}(1, p^{(l)}, 0)}{\gd ( \tau_k, p^{(l)}, \theta^{(l)})}\right ) \, ^T \Big (0,  dp^{(l)}_1( \tilde A^{(l)} _{p^{(l)}}), \dots, dp^{(l)}_{c^{(k)} -c^{(l)} -1}( \tilde A^{(l)} _{p^{(l)}}),  \theta^{(l)}_1\big  (   \tilde v^{(l)} \big )- \theta^{(l)}_1\big  ( ( h^{(l)})_{\ast, p^{(k)}}\tilde v^{(l)} \big ),  \\ \dots, \theta^{(l)}_{c^{(l)}} \big  (   \tilde v^{(l)} \big )-  \theta^{(l)}_{c^{(l)}} \big  (  ( h^{(l)})_{\ast, p^{(k)}} \tilde v^{(l)} \big ) \Big )+  O(|\tau_l^2 A^{(l)}|)+O\big (\big |\tau_l^2  \big [\theta^{(l)} ( \tilde v ^{(l)}) - \theta^{(l)} \big ((h^{(l)})_{\ast, p^{(k)}} \tilde v ^{(l)}\big )\big ]\big | \big ), 
\end{gather*}
where $\{p^{(l)}_r\}$ are local coordinates on $ \gamma^{(k)} ((S_k)_{p^{(k)}})_{l}(H_{l})$, 
\bqn 
\left ( \frac {\gd \theta ^{(k)}}{\gd ( \tau_k, p^{(l)}, \theta^{(l)})}\right ) (\tau_k, p^{(l)}, \theta^{(l)})
\eqn
denotes the Jacobian of the coordinate change $ \theta^{(k)}=  \theta^{(k)}( \tau_k \exp_{p^{(l)}} \gamma^{(l)} ( v^{(l)}))$, and 
all vectors are considered as row vectors, the transposed being a column vector. Since similar considerations hold in the other charts, we obtain with  \eqref{eq:kappafact} and \eqref{eq:25} a  partial monomialization of $( \zeta^{(1)}\circ  \zeta^{(2)})^\ast (I_\mathcal{N})$ according to
 \bqn 
 ( \zeta^{(1)} \circ  \zeta^{(2)})^\ast (I_\mathcal{N}) \cdot \E_{\tilde m, B_{Z^{(2)}} (\tilde M^{(1)})}=  \tau_k \tau_l \cdot ( \zeta^{(1)} \circ  \zeta^{(2)})^{-1}_\ast (I_\mathcal{N}) \cdot  \E_{\tilde m, B_{Z^{(2)}} (\tilde M^{(1)})}
  \eqn 
  in a neighborhood of any point $\tilde m\in B_{Z^{(2)}} (\tilde M^{(1)})$. 
In the same way,  the phase function factorizes locally according to 
\bqn 
\Phi \circ (\id_\xi \otimes (\zeta_k^\rho \circ \zeta_{kl}^{\rho\sigma}))= \,^{(kl)} \tilde \Phi^{tot}=\tau_k \, \tau_l\,  ^{(kl)} \phw,
\eqn
which by \eqref{eq:F1} and \eqref{eq:25} explicitly reads 
\begin{gather*}
\Phi(x,\xi,g)  =\tau_k \Big [ \tau_l \sum_{q=1}^{n-c^{(k)}}  \xi_q \, dp^{(k)}_q (\widetilde A^{(k)}_{ p^{(k)}})  \\ + \sum_{r=1}^{c^{(k)}}   \Big [ \theta_r^{(k)}\Big  (  \exp_{p^{(l)}} \tau_l \tilde v^{(l)} \Big )- \theta_r^{(k)}\Big  ( ( \e{\tau_l  A^{(l)}} h^{(l)} )_{\ast,p^{(k)}} \, \exp_{p^{(l)}} \tau_l \tilde v^{(l)} \Big ) \Big ] \xi_{n-c^{(k)}+r}  +O(|\tau_k\tau_l A^{(k)}|) \Big ]\\ 
=\tau_k \tau_l \left [ \left \langle  \left ( \frac {\gd (p^{(k)} , \theta ^{(k)})(p^{(k)}, 1, p^{(l)}, 0)}{\gd (p^{(k)}, \tau_k, p^{(l)}, \theta^{(l)})}\right )  \,^T \Big (dp^{(k)}_1 (\widetilde A^{(k)}_{ p^{(k)}}), \dots, 0,  dp^{(l)}_1( \tilde A^{(l)} _{p^{(l)}}),  \dots,   \theta^{(l)}_1\big  (  \tilde v^{(l)} \big ) \right. \right. \\
 \left. \left.  -  \theta^{(l)}_1\big  (  (  h^{(l)})_{\ast, p^{(k)}} \tilde v^{(l)} \big ), \dots  \Big )  , \xi \right \rangle +O(|\tau_k \,A^{(k)}|)+O(|\tau_l \,A^{(l)}|)+O(|\tau_l  [\theta^{(l)} ( \tilde v ^{(l)}) - \theta^{(l)} ((h^{(l)})_{\ast, p^{(k)}} \tilde v ^{(l)})]| )\right  ]
\end{gather*}
in the  $(\theta^{(k)},\theta^{(l)})$-charts. A computation now shows that the weak transforms $ ^{(kl)} \phw$ have no critical points in the $(\theta^{(k)},\alpha^{(l)})$-charts. We shall therefore see in Section \ref{sec:MR} that modulo lower order terms $ I_{kl}^\rho(\mu)$ is given by a sum of integrals of the form
\begin{align*}
 I_{kl}^{\rho\sigma}(\mu) =&  \int_{M_k(H_k)\times(-1,1)} \Big [ \int_{ \gamma^{(k)} ((S_k)_{p^{(k)}})_{l}(H_{l})\times (-1,1)} \Big [ \int_{ \gamma^{(l)} ( (S_{kl})_{p^{(l)}}) \times G_{p^{(l)}}\times \stackrel{\circ}{D}_ \iota(\g_{p^{(l)}}^\perp)) \times \stackrel{\circ}{D}_ \iota(\g_{p^{(k)}}^\perp)) \times  \rn}\\ &  e^{i{\mu}{\tau_k\tau_l} \, ^{(kl)} \phw } a_{kl}^{\rho\sigma}  \bar {\mathcal{J}}_{kl}^{\rho\sigma} \, \d \xi \, d A^{(k)}  \,  d A^{(l)}  \, dh^{(l)} \,  d\tilde v^{(l)}    \big ] d\tau_l  \, dp^{(l)}  \Big ]  d\tau_k \,  dp^{(k)}
\end{align*}
for some $ \iota>0$, where $a_{kl}^{\rho\sigma}$ are compactly supported amplitudes, and  $d\tilde v^{(l)}$ is a suitable density on $ \gamma^{(l)} ((S_{kl})_{p^{(l)}})$ such that we have the equality 
$$\d M(x) \d g \equiv \bar {\mathcal{J}}_{kl}^{\rho\sigma} \,  d A^{(k)} \, dA^{(l)} \, dh^{(l)} \,  d\tilde v^{(l)}  \, d\tau_l \, dp^{(l)}\, d\tau_k \, dp^{(k)}.$$
Furthermore, a computation shows that $\bar {\mathcal{J}}_{kl} ^{\rho\sigma}= |\tau_l | ^{c^{(l)} +d^{(k)} +d^{(l)} -1} {\mathcal{J}}_{kl}^\rho \circ \, '\zeta_{kl}^{\rho\sigma}$.

\subsection*{Second reduction} Now, the group $G_{p^{(k)}}$ acts on $ \gamma^{(k)} ((S_k)_{p^{(k)}})_l$ with the isotropy types $(H_l)=(H_{l_j}),(H_{l_{j+1}}), \dots, (H_L)$. By the same arguments given in the first reduction, the isotropy types occuring in $W_{kl}$ constitute a subset of these types, and we shall denote them by
\bqn
(H_l) = (H_{l_{m_1}}), (H_{l_{m_2}}), \dots, (H_L).
\eqn
Consequently, for each $p^{(k)}\in M_k(H_k)$, $G_{p^{(k)}}$ acts on $S_{kl}$ with the isotropy types $(H_{l_{m_2}}), \dots, (H_L)$. If the isotropy groups  $H_{l_{m_2}}, \dots, H_L$ have the same dimensions, we shall see  that the singularities of $(\zeta^{(1)})^{-1}(\mathcal{N})$ can be locally resolved over $Z_{kl}$. Moreover,  if  $G_{p^{(k)}}$ acts on $S_{kl}$ only with type $(H_L)$, the ideal $I_\Phi$ can be partially monomialized in such a way that the critical sets of the corresponding weak transforms are  clean. But since this  is not the case in general,  we have to continue with the iteration.

\subsection*{N-th decomposition}
Denote by $\Lambda\leq L$  the maximal number of elements that a totally ordered subset of the set of isotropy types can have. Assume that  $3 \leq N <\Lambda$, and let   $\mklm{(H_{i_1}), \dots , (H_{i_{N}})}$ be a totally ordered subset of the set of isotropy types  such that $i_1 <  \dots < i_N<L$. Let 
 $f_{i_1}$, $f_{i_1i_2}$, $S_{i_1}$, $S_{i_1i_2}$, as well as  $p^{(i_1)}\in M_{i_1}(H_{i_1}), \quad p^{(i_2)}\in \gamma^{(i_1)}  ((S_{i_1}^+)_{p^{(i_{1})}} ) _{i_2}(H_{i_2}),\dots$
be defined as in the first two iteration steps, and assume that $f_{i_1\dots i_{j}}$,  $S_{i_1\dots i_{j}},p^{(i_{j})},\dots$ have already been defined for $j<N$. 
For every fixed $p^{(i_{N-1})}$, let $ \gamma^{(i_{N-1})} ((S_{i_1\dots i_{N-1}})_{p^{(i_{N-1})}})_{i_{N}}$ be the submanifold with corners of the closed $G_{p^{(i_{N-1})}}$-manifold 
$ \gamma^{(i_{N-1})}  ((S_{i_1\dots i_{N-1}})_{p^{(i_{N-1})}})$ from which all  orbit types  less than $G/H_{i_{N}}$ have been removed. Consider  the invariant tubular neighborhood 
$$f_{i_1\dots i_{N}}=\exp \circ \gamma^{(i_{N})}: \nu _{i_1\dots i_{N}} \rightarrow \gamma^{(i_{N-1})}  ((S_{i_1\dots i_{N-1}})_{p^{(i_{N-1})}})_{i_{N}}$$ of the set  $ \gamma^{(i_{N-1})} ((S_{i_1\dots i_{N-1}})_{p^{(i_{N-1})}})_{i_{N}}(H_{i_{N}})$, and define $S_{i_1\dots i_{N}}$  as the sphere subbundle in $ \nu _{i_1\dots i_{N}}$, while  
 $$S_{i_1 \dots i_N}^+=\mklm { v \in S_{i_1 \dots i_N}:  v=\sum v_i v_i^{(i_1 \dots i_N)}, \, v_{\rho_{i_N}}>0 }$$
for some $\rho_{i_N}$. 
Put $W_{i_1 \dots i_N}= f_{i_1 \dots i_N}(\stackrel \circ D_1(\nu_{i_1 \dots i_N}))$, and denote the corresponding integral in the decomposition of $I_{i_1 \dots i_{N-1}}^{\rho_{i_1} \dots \rho_{i_{N-1}}}(\mu)$ by $I_{i_1 \dots i_N}^{\rho_{i_1} \dots \rho_{i_{N-1}}}(\mu)$. Here we can assume that, modulo terms of lower order,  the $W_{i_1\dots i_N} \times G_{p^{(i_{N-1})}}$-support of the integrand in   $I_{i_1 \dots i_N}^{\rho_{i_1} \dots \rho_{i_{N-1}}}(\mu)$ is contained in a compactum of a tubular neighborhood  of the induced bundle $\pi^\ast _{i_1 \dots i_N} \mathrm{Iso} \,\gamma^{(i_{N-1})} ((S_{i_1\dots i_{N-1}})_{p^{(i_{N-1})}})_{i_{N}}(H_{i_{N}})$, where $\pi _{i_1 \dots i_N}:W_{i_1 \dots i_N}\rightarrow \gamma^{(i_{N-1})} ((S_{i_1\dots i_{N-1}})_{p^{(i_{N-1})}})_{i_{N}}(H_{i_{N}})$ denotes the canonical projection.
For  a given point $p^{(i_{N})}\in \gamma^{(i_{N-1})}   ((S^+_{i_1\dots i_{N-1}})_{p^{(i_{N-1})}})_{i_{N}}(H_{i_{N}})$,  consider further  the decomposition
\bqn 
 \g_{p^{(i_{N-1})}}= \g_{p^{(i_N)}}\oplus \g_{p^{(i_N)}}^\perp, 
\eqn
and set $d^{(i_N)}=\dim \g_{p^(i_N)}^\perp$, $ e^{(i_N)}=\dim \g_{p^(i_N)}$. This yields  the decomposition
\begin{gather}
\label{eq:gdecomp}
\g = \g_{p^{(i_1)}} \oplus \g_{p^{(i_1)}}^\perp =(\g_{p^{(i_2)}}\oplus \g_{p^{(i_2)}}^\perp) \oplus \g_{p^{(i_1)}}^\perp =\dots = \g_{p^{(i_N)}}\oplus \g_{p^{(i_N)}}^\perp \oplus \cdots \oplus \g_{p^{(i_1)}}^\perp.
\end{gather}
Denote by  $\{ A_r^{(i_N)}(p^{(i_1)},\dots,p^{(i_N)})\}$ a basis of $\g_{p^(i_N)}^\perp$, and by $\{ B_r^{(i_N)}(p^{(i_1)},\dots,p^{(i_N)})\}$ a basis of $\g_{p^(i_N)}$. For $A^{(i_N)} \in \g_{p^(i_N)}^\perp$ and $B^{(i_N)}\in \g_{p^(i_N)}$ write  further 
\begin{align*}
A^{(i_N)} &=\sum_{r=1}^{d^{(i_N)}} \alpha^{(i_N)}_r A_r^{(i_N)}(p^{(i_1)},\dots,p^{(i_N)}), \qquad B^{(i_N)} =\sum_{r=1}^{e^{(i_N)}} \beta^{(i_N)}_r B_r^{(i_N)}(p^{(i_1)},\dots,p^{(i_N)}),
\end{align*}
and let   $\mklm{v_r^{(i_1\dots i_N)}(p^{(i_1)}, \dots p^{(i_N)}  )}$ be  an orthonormal frame in $(\nu_{i_1\dots  i_N})_{p^{(i_N)}}$.

 \subsection*{N-th monoidal transformation} 
 
 Let the monoidal transformations $\zeta^{(1)}, \zeta^{(2)}$ be defined as in the first two iteration steps, and assume that the monoidal transformations $\zeta^{(j)}$ have already been defined for $j<N$. Put $\tilde \M^{(j)} = B_{Z^{(j)}}( \tilde \M^{(j-1)})$, $\tilde \M ^{(0)}= \M=M \times G$, and consider the monoidal transformation 
 \bq
 \label{eq:montrans}
 \zeta^{(N)}: B_{Z^{(N)}}( \tilde \M^{(N-1)} )\rightarrow  \tilde \M^{(N-1)}, \qquad Z^{(N)}= \bigcup_{i_1 < \dots < i_N<L} Z_{i_1\dots i_N}, \qquad \mbox{\emph{(disjoint union)}},
 \eq
 where the union is over all totally ordered subsets $\mklm{(H_{i_1}), \dots , (H_{i_{N}})}$ of $N$ elements with $i_1 < \dots < i_N < L$, and 
  \begin{gather*}
 Z_{i_1\dots i_N}\simeq \bigcup _{p^{(i_{1})},   \dots,  p^{(i_{N-1})}}   (-1,1)^{N-1} \times \mathrm{Iso} \, \gamma^{(i_{N-1})}((S_{i_1\dots i_{N-1}})_{p^{(i_{N-1})}})_{i_N} (H_{i_N})
 \end{gather*}
 are the  possible maximal singular loci of $(\zeta^{(1)} \circ \dots \circ \zeta^{(N-1)})^{-1}(\mathcal{N})$.  Denote   by  $ \zeta^{\rho_{i_1}}_{i_1} \circ \dots \circ  \zeta^{\rho_{i_1}\dots \rho_{i_N}}_{i_1\dots i_N}$ a local realization of the sequence of monoidal transformations $ \zeta^{(1)} \circ \dots \circ \zeta^{(N)}$  corresponding to   the totally ordered subset    $\mklm{(H_{i_1}), \dots, (H_{i_{N}})}$  in a set of $(\theta^{(i_1)}, \dots, \theta^{(i_N)})$-charts labeled by the indices $\rho_{i_1},\dots ,\rho_{i_N}$. 
 As a consequence, we obtain a  partial monomialization of the inverse image ideal sheaf $( \zeta^{(1)}\circ \cdots \circ  \zeta^{(N)})^\ast (I_\mathcal{N})$ according to
 \bqn 
 (\zeta^{(1)}\circ \cdots \circ  \zeta^{(N)})^\ast (I_\mathcal{N}) \cdot \E_{\tilde m, \tilde \M^{(N)}}=  \tau_{i_1} \cdots \tau_{i_N} \cdot (\zeta^{(1)}\circ \cdots \circ  \zeta^{(N)})^{-1}_\ast (I_\mathcal{N}) \cdot  \E_{\tilde m, \tilde \M^{(N)}}
  \eqn 
  in a neighborhood of any point $\tilde m\in \tilde \M^{(N)}=B_{Z^{(N)}}(\tilde \M^{(N-1)})$, as well as local factorizations of  the phase function according to 
 \bqn
\Phi \circ (\id_\xi \otimes (\zeta_{i_1}^{\rho_{i_1}} \circ\dots \circ  \zeta_{\sigma_{i_1} \dots \sigma_{i_N}}^{\rho_{i_1} \dots \rho_{i_N}}))=\, ^{(i_1\dots i_N)} \tilde \Phi^{tot}=\tau_{i_1} \cdots \tau_{i_N} \, ^{(i_1\dots i_N)}\tilde \Phi^ {wk},
\eqn
where in the relevant $(\theta^{(i_1)}, \dots, \theta^{(i_N)})$-charts 
\begin{align*}
 ^{(i_1\dots i_N)}\tilde \Phi^ {wk}=&  \Big \langle \Xi \cdot  \, ^T \big ( dp_1 ^{(i_1)} (\tilde A^{(i_1)}_{p^{(i_1)}}), \dots, 0, dp_1 ^{(i_2)} (\tilde A^{(i_2)}_{p^{(i_2)}}),\dots, 0,\dots,  dp_1 ^{(i_N)} (\tilde A^{(i_N)}_{p^{(i_N)}}),\dots, \\ &  \theta_1 ^{(i_N)} \big ( \tilde v^{(i_N)} \big )-  \theta_1 ^{(i_N)} \big ((h^{(i_N)})_{\ast, p^{(i_1)}}  \tilde v^{(i_N)} \big ), \dots \big ), \xi \Big \rangle \\ 
 &+ \sum _{j=1}^N O(|\tau_{i_j} A^{(i_j)}|)+O(| \tau_{i_N} [ \theta^{(i_N)}( \tilde v^{(i_N)}) - \theta^{(i_N)}( (h^{(i_N)})_{\ast, p^{(i_1)}} \tilde v^{(i_N)})] |),
\end{align*}
the $\{p^{(i_j)}_s\}$ being local coordinates,  $\tilde v^{(i_N)}(p^{(i_j)},\theta^{(i_N)}) \in  \gamma^{(i_N)} ((S_{i_1\dots i_N}^+)_{p^{(i_N)}})$, $h^{(i_N)} \in G_{p^{(i_N)}}$, and 
\begin{align*}
\Xi =&\Xi ^{(i_1)} \cdot \Xi^{(i_1i_2)} \cdot \cdots \cdot \Xi^{(i_1 \dots i_{N-1})}, \\
\Xi^{(i_1 \dots i_{j})}=& \frac{\gd ( p^{(i_1)}, \tau_{i_1}, p^{(i_2)}, \tau_{i_2}, \dots , p^{(i_j)}, \theta ^{(i_j)}))}{\gd ( p^{(i_1)}, \tau_{i_1}, p^{(i_2)}, \tau_{i_2}, \dots , p^{(i_j)},\tau_{i_j}, p^{(i_{j+1})},  \theta ^{(i_{j+1})} )} ( p^{(i_1)},1, p^{(i_2)},1, \dots , p^{(i_j)},1, p^{(i_{j+1})},0). 
\end{align*}
Here $\Xi^{(i_1 \dots i_{j})}$ corresponds to  the Jacobian of the coordinate change given by 
$$\theta^{(i_j)} = \theta^{(i_j)}\big ( \tau_{i_j} \exp _{p^{(i_{j+1})}} \gamma^{(i_{j+1})}(v^{(i_{j+1})})\big ).$$
Modulo lower order terms, $I(\mu)$ is then given by a sum of integrals of the form
\begin{align}
\label{eq:N}
\begin{split}
 & \qquad\qquad\qquad\qquad \qquad \qquad \qquad I_{i_1\dots i_N}^{\rho_{i_1}\dots \rho_{i_N}}(\mu) =\\ &\int_{M_{i_1}(H_{i_1})\times (-1,1)} \Big [ \int_{\gamma^{(i_{1})}((S_{i_1})_{p^{(i_1)}})_{i_2}(H_{i_2})\times (-1,1) } \dots \Big [  \int_{\gamma^{(i_{N-1})}((S_{i_1\dots i_{N-1}})_{p^{(i_{N-1})}})_{i_{N}}(H_{i_{N}})\times (-1,1)} \\
&\Big [ \int_{\gamma^{(i_{N})}((S_{i_1\dots i_{N}})_{p^{(i_N)}})\times G_{p^{(i_{N})}}\times \stackrel \circ D_ \iota( \g_{p^{(i_{N})}}^\perp) \times \cdots \times \stackrel \circ D_ \iota (\g_{p^{(i_{1})}}^\perp)\times \rn }  e^{i\mu {\tau_1 \dots \tau_N} \, ^{(i_1\dots i_N)} \tilde \Phi ^{wk}}   \\ & \, a_{i_1\dots i_N}^{\rho_{i_1}\dots \rho_{i_N}} \,   \bar {\mathcal{J}}_{i_1\dots i_N}^{\rho_{i_1}\dots \rho_{i_N}}  \d \xi  \d A^{(i_1)} \dots  \d A^{(i_N)}  \d h^{(i_N)} \d \tilde v^{(i_N)} \Big ]  \d \tau_{i_N} \d p^{(i_{N})} \dots  \Big ] \d \tau_{i_2} \d p^{(i_{2})} \Big ]\d \tau_{i_1} \d p^{(i_{1})}.
\end{split}
\end{align}
Here $a_{i_1\dots i_N}^{\rho_{i_1}\dots \rho_{i_N}} $
are amplitudes with compact support in a system of $(\theta^{(i_1)}, \dots, \theta^{(i_N)})$-charts labeled  by the indices $ \rho_{i_1}, \dots, \rho_{i_N}$, while
\begin{align*}
\bar {\mathcal{J}}_{i_1\dots i_N}^{\rho_{i_1}\dots \rho_{i_N}} &=\prod_{j=1}^N |\tau_{i_j}|^{c^{(i_j)}+\sum_{r=1}^j d^{(i_r)}-1}{\mathcal{J}}_{i_1\dots i_N}^{\rho_{i_1}\dots \rho_{i_N}},
\end{align*}
where ${\mathcal{J}}_{i_1\dots i_N}^{\rho_{i_1}\dots \rho_{i_N}}$ are functions which do not depend on the variables $\tau_{i_j}$.

\subsection*{N-th reduction} For each $p^{(i_{N-1})}$, the isotropy group $G_{p^{(i_{N-1})}}$ acts on $ \gamma^{(i_{N-1})} ((S_{i_1\dots i_{N-1}})_{p^{(i_{N-1})}})_{i_{N}}$ by the types $(H_{i_N}), \dots, (H_L)$. The types occuring in $W_{i_1 \dots i_N}$ constitute a subset of these, and   $G_{p^{(i_{N-1})}}$ acts on  the sphere bundle $S_{i_1\dots i_N}$ over the submanifold $\gamma^{(i_{N-1})}((S_{i_1\dots i_{N-1}})_{p^{(i_{N-1})}})_{i_N} (H_{i_N})\subset W_{i_1 \dots i_N}$  with one type less.

 \medskip

\subsection*{End of iteration}

As before, let $\Lambda\leq L$ be the maximal number of elements of a totally ordered subset of the set of isotropy types. After  $N=\Lambda-1$ steps, the end of the iteration is reached. In particular, we will have achieved a desingularization of $\mathcal{N}$. For this, it is actually sufficient to consider only monoidal transformations \eqref{eq:montrans} whose centers $Z^{(N)}$ are unions  over  totally ordered subsets $\mklm{(H_{i_1}),  \dots , (H_{i_{N}})}$ for which the corresponding orbit types $G/H_{i_j}$ are singular. 
\begin{theorem}
\label{thm:desing} Consider  a compact, connected $n$-dimensional Riemannian manifold $M$, together with  a compact, connected Lie groups $G$ acting effectively and isometrically on $M$, and put 
\bqn
\mathcal{N}=\mklm{(x,g) \in \M: gx=x}.
\eqn 
For every $1 \leq N \leq \Lambda -1$, let the monoidal transformation 
 $\zeta^{(N)}$  be  defined as in \eqref{eq:montrans}, where  $Z^{(N)}$ is a union over totally ordered subsets    $\mklm{(H_{i_1}), \dots, (H_{i_{N}})}$ of singular isotropy types of $N$ elements.  
 Denote the sequence of monoidal transformations $ \zeta^{(1)} \circ \dots \circ \zeta^{(\Lambda-1)}$ by $\zeta$, and put  $\tilde \M = \tilde \M^{(\Lambda-1)}$.  Then $\zeta: \tilde \M \rightarrow \M$ yields  a  strong  resolution of  $\mathcal{N}$.
\end{theorem}
\begin{proof} If all $G$-orbits on $M$ have the same dimension, $\mathcal{N}$ is a manifold, and $\zeta : \M \rightarrow \M$ is the identity. Let us therefore assume that there are singular orbits, and begin by recalling the covering
\bqn 
\mathcal{N}  =   \mathcal{N}_1\cup \dots \cup \mathcal{N}_L,
\eqn
  where  $\mathcal{N}_L=\mathrm{Iso } \, W_L$ is a manifold, and the $\mathcal{N}_k=\mathcal{N} \cap U_k$ are in general singular for $k <L$. Let $1 \leq N \leq \Lambda-1$, and consider a totally ordered subset    $\mklm{(H_{i_1}), \dots, (H_{i_{N}})}$ of isotropy types such that $i_1 < \dots<  i_{N}$. In case that $(H_{i_j})$ is exceptional, all types $(H_{i_{j'}})$ with $j < j'$ are exceptional, or principal. Indeed, if $H_{i_j}/H_L$ is finite and non-trivial, $  H_{i_{j'}}/H_L$ is also finite. In  particular, if  $(H_{i_1})$ is exceptional, $\mathcal{N}_{i_1}$ is a manifold. In what follows, let us therefore restrict to the case where $\mklm{(H_{i_1}), \dots, (H_{i_{N}})}$ is a totally ordered subset of singular isotropy types which is maximal in the sense that there is no singular isotropy type $(H_{i_{N+1}})$ with $i_N< i_{N+1}$ such that  $\mklm{(H_{i_1}), \dots, (H_{i_{N+1}})}$ is a totally ordered subset.   Let  $ \zeta^{\rho_{i_1}}_{i_1} \circ \dots \circ  \zeta^{\rho_{i_1}\dots \rho_{i_N}}_{i_1\dots i_N}$ be a local realization of the sequence of monoidal transformations $ \zeta^{(1)} \circ \dots \circ \zeta^{(N)}$  corresponding to   the totally ordered subset    $\mklm{(H_{i_1}), \dots, (H_{i_{N}})}$  in a set of $(\theta^{(i_1)}, \dots, \theta^{(i_N)})$-charts labeled by the indices $\rho_{i_1},\dots ,\rho_{i_N}$. The preimage of $\mathcal{N}_{i_1}$ under $ \zeta^{\rho_{i_1}}_{i_1} \circ \dots \circ  \zeta^{\rho_{i_1}\dots \rho_{i_N}}_{i_1\dots i_N}$  is given by all points  $$(\tau_{i_1}, \dots, \tau_{i_N}, p^{(i_1)}, \dots, p^{(i_N)},   \tilde v ^{(i_N)},  A^{(i_1)}, \dots, A^{(i_N)}, h^{(i_N)}) $$
satisfying 
\bqn 
(x^{(i_1\dots i_{N})},g^{(i_1\dots i_{N})}) \in \mathcal{N},
\eqn
 where for  $j=1,\dots ,N$ we set
\begin{align}
\label{eq:70}
\begin{split}
x^{(i_j\dots i_{N})}&=\exp_{p^{(i_j)}}[\tau_{i_j} \exp_{p^(i_{j+1})}[\tau_{i_{j+1}}\exp_{p^(i_{j+2})}[\dots [ \tau_{i_{N-2}}\exp_{p^(i_{N-1})}[\tau_{i_{N-1}}\exp_{p^{(i_N)}}[ \tau_{i_N} \tilde v ^{(i_N)}]]] \dots ]]], \\
g^{(i_j\dots i_{N})}&=\e{\tau_{i_j} \cdots \tau_{i_N}A^{(i_j)}}\e{\tau_{i_{j+1}} \cdots \tau_{i_N}A^{(i_{j+1})}}\cdots\e{\tau_{i_{N-1}}  \tau_{i_N}A^{(i_{N-1})}} \e{\tau_{i_N} A^{(i_N)}}h^{(i_N)}.
\end{split}
\end{align}
Assume now that $\tau_{i_1}\cdots \tau_{i_N}\not=0$. Since the point $x^{(i_1\dots i_N)}$ lies in a slice around  $G \cdot p^{(i_1)}$, the condition  $g^{(i_1\dots i_N)} \in G_{x^{(i_1\dots i_N)}}$ implies that   
  $g^{(i_1\dots i_N)}$ must stabilize $p^{(i_1)}$ as well. Frome the inclusions
\begin{align}
\label{eq:isotr}
G_{p^{(i_N)}} \subset G_{p^{(i_{N-1})}} \subset \dots \subset G_{p^{(i_1)}}
\end{align}
and  $\g_{p^{(i_{j+1})}}^\perp \subset \g_{p^{(i_j)}}$ one deduces 
$g^{(i_2\dots i_N)} \in G_{p^{(i_1)}}$, and we obtain
\bqn
g^{(i_1\dots i_N)}p^{(i_1)}= \e{\tau_{i_1} \dots \tau_{i_N}\sum \alpha_r^{(i_1)} A_r ^{(i_1)}} p^{(i_1)} =p^{(i_1)}.
\eqn
Thus we conclude $\alpha^{(i_1)} =0$, which implies $g^{(i_2\dots i_N)} \in G_{x^{(i_1\dots i_N)}}$, and consequently $g^{(i_2\dots i_N)} \in G_{x^{(i_2\dots i_N)}}$. Repeating the above argument  we see that 
\bqn 
(x^{(i_1\dots i_{N})},g^{(i_1\dots i_{N})}) \in \mathcal{N} \quad \Longleftrightarrow \quad A^{(i_j)}=0, \quad  h^{(i_N)} \in G_{\tilde v^{(i_N)}}
\eqn
in case that  $\tau_{i_1}\cdots \tau_{i_N}\not= 0$. Actually we have shown that  if $\tau_{i_1}\cdots \tau_{i_N}\not= 0$
\bq
\label{eq:G}
G_{x^{(i_1,\dots,i_N)}} =G_{\tilde v^{(i_N)}},
\eq
 since $G_{\tilde v^{(i_N)}}\subset G_{p^{(i_N)}}$. The preimage of $\mathcal{N}_{i_1}$ under $ \zeta^{\rho_{i_1}}_{i_1} \circ \dots \circ  \zeta^{\rho_{i_1}\dots \rho_{i_N}}_{i_1\dots i_N}$ is therefore given by
\bqn 
\mklm{\tau_{i_1}\cdots \tau_{i_N}\not=0, \quad A^{(i_j)}=0, \quad  h^{(i_N)} \in G_{\tilde v^{(i_N)}} } \cup \bigcup_{j=1}^N \mklm{\tau_{i_j}=0}.
\eqn
By assumption, $G_{p^{(i_N)}}$ acts on $(S_{i_1 \dots i_N})_{p^{(i_N)}}$ with orbits of the same dimension, so that   
\bq
\label{eq:36}
\{ A^{(i_j)}=0, \quad  h^{(i_N)} \in G_{\tilde v^{(i_N)}} \}
\eq
 is a smooth submanifold, being equal  to the total space of the isotropy  bundle given by the local trivialization
\bqn 
(\tau_{i_j}, p^{(i_j)},  \tilde v ^{(i_N)}, G_{\tilde v ^{(i_N)}})\mapsto (\tau_{i_j}, p^{(i_j)},  \tilde v ^{(i_N)}).
\eqn
Now, for $1\leq N \leq \Lambda -1$, let $\zeta^{(N)}$ be defined as in \eqref{eq:montrans}, where $Z^{(N)}$ is a union over 
totally ordered subsets  of singular isotropy types of $N$ elements, and put 
$\zeta = \zeta^{(1)} \circ \dots \circ \zeta^{(\Lambda-1)}$. By construction, $\zeta$ is given locally by sequences of local transformations 
 $ \zeta^{\rho_{i_1}}_{i_1} \circ \dots \circ  \zeta^{\rho_{i_1}\dots \rho_{i_{N}}}_{i_1\dots i_{N}}$ corresponding to maximal, totally ordered subsets   $\mklm{(H_{i_1}), \dots, (H_{i_{N}})}$ of singular isotropy types of $N \leq \Lambda -1$ elements. Taking the union over all the corresponding sets \eqref{eq:36}   yields a 
 smooth submanifold $ \mathcal{\tilde N}$ which has normals crossings with the exceptional divisor  $\zeta^{-1} (\mathrm{Sing }\, \mathcal{N}) \subset \tilde \M$. Furthermore, $\zeta$ maps the union of the sets $\mklm{\tau_{i_1}\cdots \tau_{i_N}\not=0, \quad A^{(i_j)}=0, \quad  h^{(i_N)} \in G_{\tilde v^{(i_N)}} }$ bijectively onto  the non-singular part $\mathrm{Reg} \, \mathcal{N}$ of $\mathcal{N}$. However, $\mathrm{Reg} \, \mathcal{N}$ is not necessarily dense in $\mathcal{N}$, nor is $\zeta(\mathcal{\tilde N})$, so that $\zeta:\mathcal{\tilde N}\rightarrow \mathcal{N}$ might  not be a birational map in general. Nevertheless, by sucessively blowing up the intersections of $\mathcal{\tilde N}$ with $\zeta^{-1} (\mathrm{Sing }\, \mathcal{N})$ one finally obtains a strong resolution of $\mathcal{N}$.
\end{proof}

 The resolution of $\mathcal{N}$ constructed in Theorem \ref{thm:desing} was deduced from   a  monomialization of the ideal sheaf $I_\mathcal{N}$ 
 \bqn 
 \zeta^\ast (I_\mathcal{N}) \cdot \E_{\tilde m, \tilde \M}=  \tau_{i_1} \cdots \tau_{i_{\Lambda-1}} \cdot \zeta^{-1}_\ast (I_\mathcal{N}) \cdot  \E_{\tilde m, \tilde \M},   \qquad \tilde m\in \tilde \M,
  \eqn 
 where $\zeta^{-1}_\ast (I_\mathcal{N})$ is a resolved ideal sheaf.  
  In the following two sections, we shall derive from this a partial monomialization of  the local ideal $I_\Phi=(\Phi)$ such  that  the corresponding weak transforms  of $\Phi$ have clean critical sets. This will  allow us to derive  asymptotics for the integrals $I_{i_1\dots i_N}^{\rho_{i_1}\dots \rho_{i_N}}(\mu)$ in Section \ref{sec:INT} via the stationary phase theorem.

\section{Phase analysis of the weak transforms. The first main theorem}
\label{sec:MT1}

We continue with  the notation of the previous sections and recall   that the sequence of monoidal transformations 
$\zeta = \zeta^{(1)} \circ \dots \circ \zeta^{(\Lambda-1)}$ is given locally by sequences of local transformations 
 $ \zeta^{\rho_{i_1}}_{i_1} \circ \dots \circ  \zeta^{\rho_{i_1}\dots \rho_{i_{N}}}_{i_1\dots i_{N}}$ corresponding to totally ordered subsets  $\mklm{(H_{i_1}), \dots, (H_{i_{N}})}$ of non-principal isotropy types  that are maximal in the sense that, if there is an isotropy type $(H_{i_{N+1}})$ with $i_N < i_{N+1}$ such that $\mklm{(H_{i_1}), \dots , (H_{i_{N+1}}) }$ is a totally ordered subset, then  $(H_{i_{N+1}})= (H_L)$. Let  now $x \in M$ be fixed, and   $Z_x \subset T_xM$ be a neighborhood of zero such that $\exp_x: Z_x \longrightarrow M$ is a diffeomorphism onto its image. One has
\bqn
\qquad (\exp_x)_{\ast,v}: T_v Z_x \longrightarrow T_{\exp_xv}M, \quad v \in Z_x, 
\eqn
and under the identification $T_xM \simeq T_0 Z_x$ one computes $(\exp_x)_{\ast, 0} \equiv \id$. Furthermore,  for $g \in G$ we have 
$g \cdot \exp_x v= L_g (\exp_x v)=\exp_{L_g(x)}(L_g)_{\ast,x}( v)$.  Consider next a maximal,  totally ordered subset $\mklm{(H_{i_1}), \dots , (H_{i_{N}}) }$  of isotropy types with $i_1< \dots< i_N<L$, and denote by 
\bqn
\lambda: \g_{p^{(i_1)}} \longrightarrow \mathfrak{gl}(\nu_{i_1,p^{(i_1)}}), \quad B^{(i_1)} \mapsto \frac d{dt} (L_{\e{-tB^{(i_1)}}})_{\ast, p^{(i_1)}|t=0}, 
\eqn
the linear representation of $ \g_{p^{(i_1)}}$ in $\nu_{i_1,p^{(i_1)}}$,  where $p^{(i_1)} \in M _{i_1} (H_{i_1}) $. 
 For an arbitrary element $A ^{(i_j)}\in \g_{i_j}^\perp$ with $2 \leq j \leq N$, and $x^{(i_1\dots i_{N})}$ given as in \eqref{eq:70}, one  computes
 \begin{align*}
 (\widetilde A ^{(i_j)})_{ x^{(i_1\dots i_{N})}}&=\frac d{dt} \e{-t A^{(i_j)}} \cdot  x^{(i_1\dots i_{N})}_{|t=0}= \frac d {dt} \exp_{p^{(i_1)}} \big [(\e{-t A^{(i_j)}})_{\ast, p^{(i_1)}} [\tau_{i_1}  x^{(i_2\dots i_{N})}]\big ]_{|t=0}\\
 &=(\exp_{p^{(i_1)}})_{\ast, \tau_{i_1}  x^{(i_2\dots i_{N})}}[ \lambda(A^{(i_j)})\tau_{i_1}  x^{(i_2\dots i_{N})}],
 \end{align*}
 successively obtaining
  \begin{align}
  \label{eq:15}
  \begin{split}
 (\widetilde A ^{(i_j)})&_{ x^{(i_1\dots i_{N})}}= \frac d {dt} \exp_{p^{(i_1)}} \big [\tau_{i_1} \exp_{p^{(i_2)}} [\dots [\tau_{i_{j-1}}  (\e{-t A^{(i_j)}})_{\ast, p^{(i_1)}} x^{(i_j\dots i_{N})}]\dots ]\big ]_{|t=0}\\
 &=(\exp_{p^{(i_1)}})_{\ast, \tau_{i_1}  x^{(i_2\dots i_{N})}} \big [ \tau_{i_1} (\exp_{p^{(i_2)}})_{\ast, \tau_{i_2} x^{(i_3\dots i_N)}}[ \dots [\tau_{i_{j-1}}  \lambda(A^{(i_j)})  x^{(i_j\dots i_{N})}]\dots ]\big ],
 \end{split}
 \end{align}
where we made the canonical identification $T_v(\nu_{i_1,p^{(i_1)}}) \equiv \nu_{i_1,p^{(i_1)}}$ for any $v \in(\nu_{i_1})_{p^{(i_1)}}$. 
We shall next  define certain geometric distributions $E^{(i_j)}$ and $F^{(i_N)}$ on $M$ by setting
\begin{equation}
\label{eq:EF}
\begin{split}
E^{(i_1)}_{x^{(i_1 \dots i_N)}}&=\mathrm{Span} \{  \tilde Y _{x^{(i_1 \dots i_N)}}: Y \in \g_{p^{(i_1)}} ^\perp\}, \\ 
E^{(i_j)} _{x^{(i_1 \dots i_N)}}&= (\exp_{p^{(i_1)}})_{\ast, \tau_{i_1}  x^{(i_2\dots i_N)} } \dots  (\exp_{p^{(i_{j-1})}})_{\ast, \tau_{i_{j-1}}x^{(i_j\dots i_N)}}[\lambda(\g_{p^{(i_j)}}^\perp)  x ^{(i_j\dots i_N)}]
, \\
F^{(i_N)}_{x^{(i_1 \dots i_N)}} &=  (\exp_{p^{(i_1)}})_{\ast, \tau_{i_1}  x^{(i_2\dots i_N)} } \dots (\exp_{p^{(i_{N})}})_{\ast, \tau_{i_{N}}\tilde v^{(i_N)}}[\lambda(\g_{p^{(i_N)}})  \tilde v ^{(i_N)}],
\end{split}
\end{equation}
where $2 \leq j \leq N$. By construction, if $\tau_{i_1} \cdots \tau_{i_N}\not=0$, the $G$-orbit through $x^{(i_1\dots i_N)}$ is of principal type $G/H_L$, which amounts to the fact that $G_{p^{(i_{N-1})}}$ acts on $S_{i_1\dots i_N}$ only with  isotropy type $(H_L)$, where we understand  that $G_{p^{(i_{0})}}=G$.  Furthermore,  \eqref{eq:gdecomp} and \eqref{eq:15} imply that 
\begin{align}
\label{eq:32}
T_{x^{(i_1 \dots i_N)}} ( G\cdot x^{(i_1 \dots i_N)})=E^{(i_1)}_{x^{(i_1 \dots i_N)}}\oplus  \bigoplus _{j=2}^N \tau_{i_1}\dots \tau_{i_{j-1}} E^{(i_j)} _{x^{(i_1 \dots i_N)}}  \oplus  \tau_{i_1}\dots \tau_{i_N} F^{(i_N)}_{x^{(i_1 \dots i_N)}}. 
\end{align}
The main result of this  section is the following 
 \begin{theorem}[First Main  Theorem]
\label{thm:MT1}
Let  $\mklm{(H_{i_1}), \dots, (H_{i_{N}})}$ be a maximal, totally ordered subset of non-principal isotropy types, and 
  $ \zeta^{\rho_{i_1}}_{i_1} \circ \dots \circ  \zeta^{\rho_{i_1}\dots \rho_{i_N}}_{i_1\dots i_N}$  a corresponding local realization of the sequence of monoidal transformations $ \zeta^{(1)} \circ \dots \circ \zeta^{(N)}$  in a set of $(\theta^{(i_1)}, \dots, \theta^{(i_N)})$-charts labeled by the indices $\rho_{i_1},\dots ,\rho_{i_N}$. Consider the corresponding factorization 
\bqn
\Phi \circ (\id_\xi \otimes (\zeta_{i_1}^{\rho_{i_1}} \circ\dots \circ  \zeta_{\sigma_{i_1} \dots \sigma_{i_N}}^{\rho_{i_1} \dots \rho_{i_N}}))=\, ^{(i_1\dots i_N)} \tilde \Phi^{tot}=\tau_{i_1} \cdots \tau_{i_N} \, ^{(i_1\dots i_N)}\tilde \Phi^ {wk, \, pre}
\eqn
of the phase function \eqref{eq:phase} where \footnote{Note that $^{(i_1\dots i_N)}\tilde \Phi^ {wk,pre}$ was formerly denoted by $^{(i_1\dots i_N)}\tilde \Phi^ {wk}$. }
\begin{align*}
 ^{(i_1\dots i_N)}\tilde \Phi^ {wk,pre}=&  \Big \langle \Xi \cdot  \, ^T \big ( dp_1 ^{(i_1)} (\tilde A^{(i_1)}_{p^{(i_1)}}), \dots, 0, dp_1 ^{(i_2)} (\tilde A^{(i_2)}_{p^{(i_2)}}),\dots, 0,\dots,  dp_1 ^{(i_N)} (\tilde A^{(i_N)}_{p^{(i_N)}}),\dots, \\ &  \theta_1 ^{(i_N)} \big (  \tilde v^{(i_N)} \big ) - \theta_1 ^{(i_N)} \big ( ( h^{(i_N)})_{\ast, p^{(i_1)}} \tilde v^{(i_N)} \big ), \dots \big ), \xi \Big \rangle + \sum _{j=1}^N O(|\tau_{i_j} A^{(i_j)}|) \\
 +&O(| \tau_{i_N} [ \theta^{(i_N)}( \tilde v^{(i_N)}) - \theta^{(i_N)}( (h^{(i_N)})_{\ast, p^{(i_1)}} \tilde v^{(i_N)})] |).
\end{align*}
Let further 
$ 
 \, ^{(i_1\dots i_N)}\tilde \Phi^ {wk}
$
denote the pullback of  $ \, ^{(i_1\dots i_n)}\tilde \Phi^ {wk,\,pre}$ along the  substitution $\tau=\delta_{i_1\dots i_N}(\sigma)$ given by the sequence of local quadratic transformations
\begin{align*}
\delta_{i_1\dots i_N}: (\sigma_{i_1}, \dots \sigma_{i_N}) &\mapsto \sigma_{i_1}( 1, \sigma_{i_2}, \dots, \sigma_{i_N})= (\sigma_{i_1}', \dots ,\sigma_{i_N}')\mapsto \sigma_{i_2}'(\sigma_{i_1}',1,\dots, \sigma_{i_N}')= (\sigma_{i_1}'', \dots, \sigma_{i_N}'')\\
 &\mapsto \sigma_{i_3}''(\sigma_{i_1}'',\sigma_{i_2}'', 1,\dots, \sigma_{i_N}'')= \cdots \mapsto \dots = (\tau_{i_1}, \dots ,\tau_{i_N}).
\end{align*}
Then the critical set $\Crit(\, ^{(i_1\dots i_N)} \phw)$ of $\, ^{(i_1\dots i_N)} \phw$ is given by all points 
$$(\sigma_{i_1}, \dots, \sigma_{i_N}, p^{(i_1)}, \dots, p^{(i_N)},   \tilde v ^{(i_N)},  A^{(i_1)}, \dots, A^{(i_N)}, h^{(i_N)}, \xi) $$
 satisfying  the conditions

\medskip
\begin{tabular}{ll}
\emph{(I)} &  $A^{(i_j)} =0$ for all $j=1,\dots,N$, and $(h^{(i_N)})_{\ast, p^{(i_1)}}\tilde v^{(i_N)}= \tilde v^{(i_N)}$, \\[2pt]
\emph{(II)} & $\eta_{x^{(i_1 \dots i_N)}} \in \mathrm{Ann}\big (E^{(i_j)} _{x^{(i_1 \dots i_N)}}\big )$ for all $j=1,\dots, N$, \\[2pt]
\emph{(III)} &  $\eta_{x^{(i_1 \dots i_N)}}  \in\mathrm{Ann} \big (  F^{(i_N)} _{x^{(i_1 \dots i_N)}}\big )$,
\end{tabular}
\medskip

\noindent where  $\eta$ denotes the $1$-form $\sum_{i=1}^{n}  \xi_i\,  d\tilde x_i$. Furthermore,  $\Crit(\, ^{(i_1\dots i_N)} \phw)$ is a $\Cinft$-submanifold of codimension $2\kappa$, where $\kappa=\dim G/H_L$ is the dimension of a principal orbit.
\end{theorem}
\begin{proof}
In what follows, set
\bq
\label{eq:40}
\mathcal{Z}^{\rho_{i_1}\dots \rho_{i_N}}_{i_1\dots i_N}=  (\zeta^{\rho_{i_1}}_{i_1} \circ \dots \circ  \zeta^{\rho_{i_1}\dots \rho_{i_N}}_{i_1\dots i_N} \circ (\delta_{i_1\dots i_N}\otimes \id) )\otimes \id_\xi
\eq
so that 
\bqn
\Phi \circ \mathcal{Z}^{\rho_{i_1}\dots \rho_{i_N}}_{i_1\dots i_N} =\, ^{(i_1\dots i_N)} \tilde \Phi^{tot}=\tau_{i_1}(\sigma) \dots \tau_{i_N}(\sigma) \, ^{(i_1\dots i_N)}\tilde \Phi^ {wk}, 
\eqn
 and   let  $\sigma_{i_1} \cdots \sigma_{i_N}\not=0$. In this case, $\mathcal{Z}^{\rho_{i_1}\dots \rho_{i_N}}_{i_1\dots i_N}$ constitutes a diffeomorphism, so that  
 \begin{align*}
 \mathrm{Crit}(\, ^{(i_1\dots i_N)} \tilde \Phi^{tot})_{\sigma_{i_1} \cdots \sigma_{i_N}\not=0}=\{&(\sigma_{i_1}, \dots, \sigma_{i_N}, p^{(i_1)}, \dots, p^{(i_N)},  \tilde v ^{(i_N)}, A^{(i_1)}, \dots, A^{(i_N)}, h^{(i_N)}, \xi):  \\ &(x^{(i_1 \dots i_N)},\xi,   g^{(i_1\dots i_N)}) \in  {\mathcal{C}}, \quad  {\sigma_{i_1} \cdots \sigma_{i_N}\not=0}  \},
  \end{align*}
where we employed the notation of \eqref{eq:70}.     Now, by \eqref{eq:crit},
  \begin{align*} 
  (x ^{(i_1\dots i_N)}, \xi, g ^{(i_1\dots i_N)}) \in \mathcal{C} \quad \Longleftrightarrow &\quad \eta_{x ^{(i_1\dots i_N)}}=\sum_{i=1}^{n}  \xi_i\,  (d\tilde x_i)_{x ^{(i_1\dots i_N)}} \in \Omega,\\ & \quad g ^{(i_1\dots i_N)} \in G_{(x ^{(i_1\dots i_N)},\eta_{x ^{(i_1\dots i_N)}})}.
  \end{align*}
The reasoning which led to  \eqref{eq:G}  in particular implies that condition (I)  is equivalent to $g^{(i_1\dots i_N)} \in G_{x^{(i_1\dots i_N)}}$ in case that all $\sigma_{i_j}$ are different from zero.  Now, $ \eta_{{x ^{(i_1\dots i_N)}}} \in \Omega$ means that 
\bqn
 \sum \xi_i (d\tilde x_i)_{x^{(i_1\dots i_ N)}} \in \mathrm{Ann}(T_{x^{(i_1\dots i_ N)}} (G \cdot {x^{(i_1\dots i_ N)}})).
 \eqn
But if $\sigma_{i_j}\not=0$ for all $j=1,\dots,N$, (II) and (III) imply that 
\bqn 
\eta_{x^{(i_1\dots i_ N)}}\Big ( (\exp_{p^{(i_1)}})_{\ast, \tau_{i_1}  x^{(i_2\dots i_N)} } [\dots  (\exp_{p^{(i_{j-1})}})_{\ast, \tau_{i_{N-1}}x^{(i_N)}}[\lambda(Z)  x^{(i_N)}] \dots \big ] \Big )=0 \quad \forall Z \in \g_{p^{(i_{N-1})}},
\eqn
since $ \g_{p^{(i_{N-1})}} =  \g_{p^{(i_{N})}}\oplus  \g_{p^{(i_{N})}}^\perp$. By repeatedly using this argument, we  conclude  that under the assumption $\sigma_{i_1}\cdots \sigma_{i_N}\not= 0$
\bq
\label{eq:IVb}
\mathrm{(II), \, (III)} \quad \Longleftrightarrow \quad \eta_{x^{(i_1 \dots i_N)}} \in \mathrm{Ann}(T_{x^{(i_1\dots i_ N)}} (G \cdot {x^{(i_1\dots i_ N)}})).
\eq
Taking everything together therefore gives
\begin{align}
\begin{split}
\label{eq:XX}
\mathrm{Crit}(\, ^{(i_1\dots i_N)}& \tilde \Phi^{tot})_{\sigma_{i_1}\cdots \sigma_{i_N}\not=0}\\ 
 &=\{(\sigma_{i_1}, \dots, \sigma_{i_N}, p^{(i_1)}, \dots, p^{(i_N)},  \tilde v ^{(i_N)}, A^{(i_1)}, \dots, A^{(i_N)}, h^{(i_N)}, \xi):  \\    &{\sigma_{i_1} \cdots \sigma_{i_N}\not=0}, \, \text{(I)-(III) are fulfilled and $h^{(i_N)}\cdot  \eta_{{x ^{(i_1\dots i_N)}}}= \eta_{{x ^{(i_1\dots i_N)}}}$}  \},
\end{split}
  \end{align}  
and we assert that   
$$ \mathrm{Crit}(\, ^{(i_1\dots i_N)} \tilde \Phi^{wk})=\overline{\mathrm{Crit}(\, ^{(i_1\dots i_N)} \tilde \Phi^{tot})_{\sigma_{i_1} \cdots \sigma_{i_N}\not=0}}.$$
To show this, let us  still assume that all $\sigma_{i_j}$ are different from zero. Then all $\tau_{i_j}$ are different from zero, too, and $\gd_\xi \, ^{(i_1\dots i_N)} \phw=0$ is equivalent to
\begin{gather*}
\gd _\xi \Phi(x^{(i_1\dots i_N)},\xi,g^{(i_1\dots i_N)})=0,
\end{gather*}
which gives us the condition $g^{(i_1\dots i_N)} \in G_{x^{(i_1\dots i_N)}}$. By the reasoning which led to \eqref{eq:G} we therefore obtain condition (I) in the case that all $\sigma_{i_j}$ are different from zero. 
Let now one of the  $\sigma_{i_j}$ be equal to zero.  Then all $\tau_{i_j}$ are zero, too, and $\gd_\xi \, ^{(i_1\dots i_N)} \phw=0$ is equivalent to
\begin{align}
\label{eq:B}
\begin{split}
 \widetilde A^{(i_j)}_{ p^{(i_j)}}(p^{(i_j)}_q) &=0  \quad  \text{for all $ 1 \leq j \leq N$ and $q$}, \quad (\1 - h^{(i_N)} )_{\ast,p^{(i_1)}} \, \tilde v^{(i_N)} =0,
\end{split}
\end{align}
since the $(n\times n)$-matrix $\Xi$ is invertible, so that the kernel of the corresponing linear transformation is trivial. Denote by  $N_{p^{(i_1)}} ( G \cdot p^{(i_1)})$ the normal space in $T_{p^{(i_1)}}M$ to the orbit $G \cdot p^{(i_1)}$, on which $G_{p^{(i_1)}}$ acts, and define  $N_{p^{(i_{j+1})}} ( G_{p^{(i_{j})}} \cdot p^{(i_{j+1})})$ successively as the normal space to the orbit $ G_{p^{(i_{j})}} \cdot p^{(i_{j+1})}$  in the $ G_{p^{(i_{j})}}$-space $N_{p^{(i_j)}} ( G_{p^{(i_{j-1})}} \cdot p^{(i_j)})$, where we understand that $G_{p^{(i_0)}}=G$.  Since smooth actions of compact Lie groups are locally smooth, the aforementioned  actions can be assumed to be orthogonal, see   \cite{bredon}, pages 171 and  308. Since  $\widetilde A^{(i_1)}_{ p^{(i_1)}} \in   T_{p^{(i_1)}} ( G \cdot p^{(i_1)})$ is tangent to $M_{i_j}(H_{i_j})$, and 
$
\widetilde A^{(i_j)}_{ p^{(i_j)}} \in   T_{p^{(i_j)}} ( G_{p^{(i_{j-1})}} \cdot p^{(i_j)})
$ 
is tangent to $\gamma^{(i_{j-1})}   ((S^+_{i_1\dots i_{j-1}})_{p^{(i_{j-1})}})_{i_{j}}(H_{i_{j}})$, we finally obtain
\bq
\label{eq:I}
\gd_\xi \, ^{(i_1\dots i_N)} \phw=0 \quad \Longleftrightarrow \quad \text{(I)}
\eq
for arbitrary $\sigma_{i_j}$.  In particular, one concludes that $\, ^{(i_1\dots i_N)} \phw$ must vanish on its critical set. Since 
\bqn 
d(\, ^{(i_1\dots i_N)} \tilde \Phi^{tot})= d(\tau_{i_1}\dots \tau_{i_N}) \cdot \, ^{(i_1\dots i_N)} \tilde \Phi^{wk} +  \tau_{i_1}\dots \tau_{i_N} d\,( ^{(i_1\dots i_N)} \tilde \Phi^{wk}),
\eqn
one sees that 
\bqn 
 \mathrm{Crit}(\, ^{(i_1\dots i_N)} \tilde \Phi^{wk})\subset \mathrm{Crit}(\, ^{(i_1\dots i_N)} \tilde \Phi^{tot}).
\eqn
In turn, the vanishing of $\Phi$ on its critical set implies
\bq
\label{eq:38}
 \mathrm{Crit}(\, ^{(i_1\dots i_N)} \tilde \Phi^{wk})_{\sigma_{i_1}\dots \sigma_{i_N}\not=0}=  \mathrm{Crit}(\, ^{(i_1\dots i_N)} \tilde \Phi^{tot})_{\sigma_{i_1}\dots \sigma_{i_N}\not=0}.
\eq
Therefore, by continuity, 
\bq
\label{eq:XXbis}
\overline{\mathrm{Crit}(\, ^{(i_1\dots i_N)} \tilde \Phi^{tot})_{\sigma_{i_1}\dots \sigma_{i_N}\not=0}} \subset  \mathrm{Crit}(\, ^{(i_1\dots i_N)} \tilde \Phi^{wk}).
\eq
In order to see the converse inclusion we shall henceforth  assume that $\gd_\xi \, ^{(i_1\dots i_N)} \phw=0$, and consider next the $\alpha$-derivatives, where we shall again take $\sigma_{i_1}\cdots \sigma_{i_N}\not=0$. Taking into account \eqref{eq:G} and  \eqref{eq:I}, one sees that 
\begin{gather*}
\gd_{\alpha_r^{(i_j)}} \, ^{(i_1\dots i_N)} \phw=0 \\  \Longleftrightarrow \quad \frac 1 {\tau_{i_1}\cdots \tau_{i_N}} \gd_{\alpha_r^{(i_j)}} \Phi(x^{(i_1\dots i_N)},\xi,g^{(i_1\dots i_N)})=
 \frac 1 {\tau_{i_1}\cdots \tau_{i_{j-1}}} \sum_{q=1}^{n}  \xi_q\,  (\widetilde A^{(i_j)}_r)_{ x^{(i_1\dots i_N)}} (\tilde x_q) =0.
\end{gather*}
By  \eqref{eq:15} we therefore obtain for arbitrary $\sigma$ and $1 \leq j \leq N$
\begin{align*}
\gd_{\alpha^{(i_j)}} \, ^{(i_1\dots i_N)} \phw=0 \quad & \Longleftrightarrow \quad \sum_{q=1}^{n}  \xi_q\,  (d\tilde x_q)_{ x^{(i_1\dots i_N)}} \in \mathrm{Ann} (E_{ x^{(i_1\dots i_N)}}^{(i_j)}).
\end{align*}
Consequently, 
\bq
\label{eq:II}
\gd_{\alpha} \, ^{(i_1\dots i_N)} \phw=0 \quad  \Longleftrightarrow \quad \mathrm{(II)}.
\eq
In a similar way, one  sees that 
\begin{align}
\label{eq:III}
\gd_{h^{(i_N)}} \, ^{(i_1\dots i_N)} \phw=0 \quad & \Longleftrightarrow \quad \mathrm{(III)},
\end{align}
by which the necessity of the conditions (I)--(III) is established. In order to see their sufficiency, let them be fulfilled, and assume again that $\sigma_{i_j}\not=0$ for all $j=1,\dots,N$. Then \eqref{eq:IVb} implies that  $ \eta_{x^{(i_1 \dots i_N)}} \in \mathrm{Ann}(T_{x^{(i_1\dots i_ N)}} (G \cdot {x^{(i_1\dots i_ N)}}))$. Now, if $\sigma_{i_1}\cdots \sigma_{i_N}\not=0$, $G  \cdot x^{(i_1\dots i_N)}$ is of principal type $G/ H_L$ in $M$, so that the isotropy group of $x^{(i_1\dots i_N)}$ must act trivially on $N_{x^{(i_1\dots i_N)}}(G\cdot x^{(i_1\dots i_N)})$, compare \cite{bredon}, page 181. If therefore $ \mathfrak{X}= \mathfrak{X}_T+ \mathfrak{X}_N$ denotes an arbitrary element in $ T_{x^{(i_1\dots i_N)}}M = T_{x^{(i_1\dots i_ N)}} (G \cdot {x^{(i_1\dots i_ N)}}))\oplus N_{x^{(i_1\dots i_ N)}} (G \cdot {x^{(i_1\dots i_ N)}}))$, and $g \in G_{x^{(i_1\dots i_N)}}$, one computes 
\begin{align*} 
g \cdot \eta_{x^{(i_1\dots i_N)}}( \mathfrak{X})&= [ ( L_{g^{-1}})^\ast _{gx^{(i_1\dots i_N)}} \eta_{x^{(i_1\dots i_N)}}]( \mathfrak{X})= \eta_{x^{(i_1\dots i_N)}} ( ( L_{g^{-1}}) _{\ast, x^{(i_1\dots i_N)}}( \mathfrak{X}_N))\\
&= \eta_{x^{(i_1\dots i_N)}} (  \mathfrak{X}_N)= \eta_{x^{(i_1\dots i_N)}} (  \mathfrak{X}).
\end{align*}
With \eqref{eq:G} we then conclude that $ h^{(i_N)}\cdot \eta_{x^{(i_1\dots i_N)}}=\eta_{x^{(i_1\dots i_N)}}$, since   $( h^{(i_N)})_{\ast,p^{(i_1)}} \tilde v^{(i_N)}=\tilde v^{(i_N)}$ by \eqref{eq:I}.  Set next   
\bq
\label{eq:V}
V^{(i_1\dots i_{j})}= N_{p^{(i_j)}} ( G_{p^{(i_{j-1})}} \cdot p^{(i_j)}).
\eq
We then have the following 
\begin{lemma}
The orbit of the point $\tilde v^{(i_N)}$ in the $G_{p^{(i_N)}}$-space $V^{(i_1\dots i_N)}$ is of principal type.
\end{lemma}
\begin{proof}[Proof of the lemma]
By assumption, for  $\sigma_{i_j}\not=0$,  $1 \leq j \leq  N$, the $G$-orbit of $x^{(i_1\dots i_N)}$ is of principal type $G/ H_L$ in $M$. The theory of compact group actions then implies that this is equivalent to the fact that $x^{(i_2 \dots i_N)} \in V^{(i_1)}$ is of principal type in the $G_{p^{(i_1)}}$-space $V^{(i_1)}$, see  \cite{bredon}, page 181, which in turn is equivalent to the fact that $x^{(i_3 \dots i_N)} \in V^{(i_1i_2)}$ is of principal type in the $G_{p^{(i_2)}}$-space $V^{(i_1i_2)}$, and so forth. Thus, $x^{(i_j \dots i_N)} \in V^{(i_1 \dots i_{j-1})}$ must be of principal type in the $G_{p^{(i_{j-1})}}$-space $V^{(i_1\dots i_{j-1})}$ for all $j=1,\dots N$, and the assertion follows.
\end{proof}
Let us now assume that one of the $\sigma_{i_j}$ vanishes. Then  
\begin{align}
\label{eq:VII}
\mathrm{(II), \, (III)} \quad \Leftrightarrow \quad & \left \{
\begin{array}{l}
\eta_{p^{(i_1)}} \in \mathrm{Ann}(E_{p^{(i_1)}}^{(i_j)})  \quad \forall \, j=1, \dots, N, \\
\eta_{p^{(i_1)}} \in \mathrm{Ann}(F_{p^{(i_1)}}^{(i_N)} ),
\end{array} \right.
\end{align}
where $E^{(i_1)}_{p^{(i_1)}}  = T_{p^{(i_1)}}(G\cdot p^{(i_1)})$, and  
\bq
\label{eq:isom} 
E^{(i_j)}_{p^{(i_1)}} \simeq T_{p ^{(i_j)}}(G_{p^{(i_{j-1})}} \cdot p ^{(i_j)}) \subset  V^{(i_{1} \dots i_{j-1})}, \qquad 2 \leq j \leq N, 
\eq
while $F^{(i_N)}_{p^{(i_1)}}\simeq T_{\tilde v^{(i_N)}} (G_{p^{(i_{N})}}\cdot \tilde v^{(i_N)})  \subset V^{(i_1 \dots i_{N})}$.  Consequently,  we obtain the direct sum of vector spaces
\bqn 
E_{p^{(i_1)}}^{(i_1)}\oplus E_{p^{(i_1)}}^{(i_2)}\oplus \dots  \oplus E_{p^{(i_1)}}^{(i_N)}\oplus F_{p^{(i_1)}}^{(i_N)} \subset T_{p^{(i_1)}}M.
\eqn 
Now, as a consequence of the previous lemma,  the stabilizer of $\tilde v^{(i_N)} $  must act trivially on $N_{\tilde v^{(i_N)}} ( G_{p^{(i_N)}} \cdot\tilde v^{(i_N)} )$. If therefore  $ \mathfrak{X}= \mathfrak{X}_{T}+  \mathfrak{X}_N$ denotes an arbitrary element in 
\begin{align*}
T_{p^{(i_1)}} M 
& \simeq \bigoplus_{j=1}^N T_{p^{(i_j)}} (G_{p^{(i_{j-1})}} \cdot {p^{(i_j)}}) \oplus  T_{\tilde v^{(i_N)}} (G_{p^{(i_{N})}}\cdot \tilde v^{(i_N)}) \oplus  N_{\tilde v^{(i_N)}} (G_{p^{(i_{N})}}\cdot \tilde v^{(i_N)})\\
&\simeq \bigoplus_{j=1}^N E^{(i_j)}_{p^{(i_1)}} \oplus F^{(i_N)}_{p^{(i_1)}}\oplus N_{\tilde v^{(i_N)}} (G_{p^{(i_{N})}}\cdot \tilde v^{(i_N)}),
\end{align*}
  \eqref{eq:isotr},  \eqref{eq:VII}, and $  G_{\tilde v ^{(i_N)}} \subset G_{p^{(i_N)}}$  imply that  for $g \in G_{\tilde v^{(i_N)}}$ 
\begin{align*} 
g \cdot \eta_{p^{(i_1)}}( \mathfrak{X})&= [ ( L_{g^{-1}})^\ast _{gp^{(i_1)}} \eta_{p^{(i_1)}}]( \mathfrak{X})= \eta_{p^{(i_1)}} ( ( L_{g^{-1}}) _{\ast, p^{(i_1)}}( \mathfrak{X}_N))\\
&= \eta_{p^{(i_1)}} (  \mathfrak{X}_N)= \eta_{p^{(i_1)}} (  \mathfrak{X}).
\end{align*}
Collecting everything together we have shown for arbitrary $\sigma=(\sigma_{i_1}, \dots, \sigma_{i_N})$ that 
\begin{align}
\label{eq:IV}
 \gd_{\xi, \alpha^{(i_1)}, \dots , \alpha^{(i_N)} , h^{(i_N)}} \, ^{(i_1\dots i_N)} \phw=0 \quad \Longleftrightarrow \quad \mathrm{(I), \,(II), \,(III)} \quad & \Longrightarrow \quad h^{(i_N)} \in G_{\eta_{x^{(i_1\dots i_N)}}}.
\end{align}
By \eqref{eq:XX} and \eqref{eq:XXbis} we therefore conclude 
\bq
\label{eq:CC}
\overline{\mathrm{Crit}(\, ^{(i_1\dots i_N)} \tilde \Phi^{tot})_{\sigma_{i_1}\dots \sigma_{i_N}\not=0}} =  \mathrm{Crit}(\, ^{(i_1\dots i_N)} \tilde \Phi^{wk}).
\eq
We have thus computed the critical set of $\, ^{(i_1\dots i_N)} \phw$, and it remains to show that it is a $\Cinft$-submanifold of codimension $2\kappa$. By the previous considerations, 
\begin{align}
\label{eq:C}
\begin{split}
&\Crit(\, ^{(i_1\dots i_N)} \phw)\\ 
= \Big \{ A^{(i_j)}=0, \quad h^{(i_N)} \in & \, G_{\tilde v^{(i_N)}}, \quad  \eta_{x^{(i_1\dots i_N)}}\in \mathrm{Ann} \Big(\bigoplus_{j=1}^N E^{(i_j)}_{x^{(i_1\dots i_N)}}\oplus F^{(i_N)}_{x^{(i_1\dots i_N)}}\Big ) \Big \}.
\end{split}
\end{align}
Now, since for  $\sigma_{i_1}\cdots \sigma_{i_N}\not=0$ the $G$-orbit of $x^{(i_1\dots i_N)}$ is of principal type $G/ H_L$ in $M$, \eqref{eq:32} implies in this case that 
\begin{align*}
\kappa =&  \dim T_{x^{(i_1 \dots i_N)}} ( G\cdot x^{(i_1 \dots i_N)})
=\dim \Big [E^{(i_1)}_{x^{(i_1 \dots i_N)}}\oplus  \bigoplus _{j=2}^N \tau_{i_1}\dots \tau_{i_{j-1}} E^{(i_j)} _{x^{(i_1 \dots i_N)}}  \oplus  \tau_{i_1}\dots \tau_{i_N} F^{(i_N)}_{x^{(i_1 \dots i_N)}}\Big ] \\
 =& \sum_{j=1} ^N \dim E^{(i_j)}_{x^{(i_1\dots i_N)}} + \dim F^{(i_N)}_{x^{(i_1\dots i_N)}}.
\end{align*}
But  $\dim E^{(i_j)}_{x^{(i_1\dots i_N)}} = \dim G_{p^{(i_{j-1})}} \cdot p ^{(i_j )}$ in particular shows that  the dimensions of the spaces $E^{(i_j)}_{x^{(i_1 \dots i_N)}}$   do not depend on the variables $\sigma_{i_j}$. A similar argument applies to $F^{(i_N)}_{x^{(i_1 \dots i_N)}}$, so that we obtain the equality 
\bq
\label{eq:kappa}
\kappa=\sum_{j=1} ^N \dim E^{(i_j)} _{x^{(i_1 \dots i_N)}}+ \dim F^{(i_N)}_{x^{(i_1 \dots i_N)}}
\eq
for arbitrary ${x^{(i_1 \dots i_N)}}$. Note that, in contrast, the dimension of  $T_{x^{(i_1 \dots i_N)}} ( G\cdot x^{(i_1 \dots i_N)})$ collapses, as soon as one of the $\tau_{i_j}$ becomes zero. Since the annihilator of a subspace of $T_x M$ is a linear subspace of $T^\ast_x M$,  we arrive at a vector bundle with $(n-\kappa)$-dimensional fiber that is locally given by the trivialization 
\bqn 
\Big (\sigma_{i_j}, p^{(i_j)},  \tilde v ^{(i_N)}, \mathrm{Ann}  \Big(\bigoplus _{j=1}^N  E^{(i_j)} _{x^{(i_1 \dots i_N)}}  \oplus  F^{(i_N)}_{x^{(i_1 \dots i_N)}}\Big ) \Big )\mapsto (\sigma_{i_j}, p^{(i_j)},  \tilde v ^{(i_N)}).
\eqn
Consequently, by \eqref{eq:IV} and  \eqref{eq:C} we see that $\Crit(\, ^{(i_1\dots i_N)} \phw)$ is equal to the total space of the fiber product of the mentioned vector bundle with the isotropy  bundle given by the local trivialization
\bqn 
(\sigma_{i_j}, p^{(i_j)},  \tilde v ^{(i_N)}, G_{\tilde v ^{(i_N)}})\mapsto (\sigma_{i_j}, p^{(i_j)},  \tilde v ^{(i_N)}).
\eqn
Lastly, equation \eqref{eq:G} implies  $\dim \g_{\tilde v ^{(i_N)}}=d-\kappa$, which concludes the proof of Theorem \ref{thm:MT1}.
\end{proof}

\section{Phase analysis of the weak transforms. The second main theorem}
\label{sec:MT2}

In this section, we shall prove  that the Hessians of the weak transfoms  $ \, ^ {(i_1\dots i_N)} \tilde \Phi ^{wk}$ are transversally non-degenerate at each point of their critical sets.  We begin with  the following general observation. Let $M$ be a $n$-dimensional $\Cinft$-manifold, and $C$ the critical set of a function $\psi \in \Cinft(M)$, which is assumed to be a smooth submanifold in a chart $\mathcal{O} \subset M$. Let further 
\bqn
\alpha:(x,y) \mapsto m, \qquad \beta:(q_1,\dots, q_n) \mapsto m, 	\qquad m \in \mathcal{O},
\eqn
be two systems of  local coordinates on $\mathcal{O}$, such that $\alpha(x,y) \in C$ if and only if $y=0$. One computes
\begin{align*}
\gd_{y_l}(\psi \circ \alpha)(x,y)=\sum_{i=1}^n \frac{\gd(\psi \circ \beta)}{\gd q_i} (\beta^{-1}\circ \alpha(x,y)) \, \gd _{y_l} (\beta^{-1} \circ \alpha)_i (x,y),
\end{align*}
as well as
\begin{align*}
\gd_{y_k} \gd_{y_l} (\psi \circ \alpha)(x,y)&=\sum_{i=1}^n \frac{\gd(\psi \circ \beta)}{\gd q_i} (\beta^{-1}\circ \alpha(x,y)) \, \gd_{y_k}  \gd _{y_l} (\beta^{-1} \circ \alpha)_i (x,y)\\
&+ \sum_{i,j=1}^n \frac{\gd^2(\psi \circ \beta)}{\gd q_i \gd{q_j}} (\beta^{-1}\circ \alpha(x,y))  \gd_{y_k}(\beta^{-1} \circ\alpha)_j(x,y) \,  \gd_{y_l}(\beta^{-1} \circ\alpha)_i(x,y).
\end{align*}
Since
\bqn
\alpha_{\ast,(x,y)}(\gd_{y_k})=\sum_{j=1}^n \gd_{y_k} (\beta^{-1} \circ \alpha)_j(x,y) \, \beta_{\ast,(\beta^{-1} \circ \alpha)(x,y)}(\gd_{q_j}),
\eqn
this implies
\bq
\label{eq:Hess}
\gd_{y_k} \gd_{y_l} (\psi \circ \alpha)(x,0)= \mathrm{Hess}\,  \psi_{|\alpha(x,0)} (\alpha_{\ast,(x,0)}(\gd_{y_k}),\alpha_{\ast,(x,0)}(\gd_{y_l})),
\eq
by definition of the Hessian \cite{milnor63}, Section 2. Let us now write $x=(x',x'')$, and consider the restriction of $\psi$ onto  the $\Cinft$-submanifold
$
M_{c'}=\mklm { m \in \mathcal{O}: m=\alpha(c',x'',y)}.
$
We write $\psi_{c'}=\psi_{|M_{c'}}$, and denote the critical set of $\psi_{c'}$ by $C_{c'}$, which contains $C \cap M_{c'}$ as a subset. Introducing on $M_{c'}$ the local coordinates
$
\alpha':(x'',y) \mapsto \alpha(c',x'',y),
$
we obtain 
\bqn
\gd_{y_k} \gd_{y_l} (\psi_{c'} \circ \alpha')(x'',0)= \mathrm{Hess}\,  \psi_{c'|\alpha(x'',0)} (\alpha'_{\ast,(x'',0)}(\gd_{y_k}),\alpha'_{\ast,(x'',0)}(\gd_{y_l})).
\eqn
Let us now assume $C_{c'}=C \cap M_{c'}$, a transversal intersection.  Then $C_{c'}$ is a submanifold of $M_{c'}$, and the complement of  $T_{\alpha'(x'',0)}C_{c'}$ in  $T_{\alpha'(x'',0)} M_{c'}$ at a point $\alpha'(x'',0)$ is spanned by the vector fields $\alpha'_{\ast,(x'',0)} (\gd _{y_k})$.
Since clearly
\bqn
\gd_{y_k} \gd_{y_l} (\psi_{c'} \circ \alpha')(x'',0)=\gd_{y_k} \gd_{y_l} (\psi \circ \alpha)(x,0),\qquad x=(c',x''),
\eqn
we thus have proven the following
\begin{lemma}
\label{lemma:A}
Assume that $C_{c'}=C \cap M_{c'}$. Then $\mathrm{Hess} \, \psi$ is transversally non-degenerate at $\alpha(c', x'',0) \in C$  if, and only if  $\mathrm{Hess} \, \psi_{c'}$ is transversally non-degenerate at $\alpha'( x'',0) \in C_{c'}$. That is, $\mathrm{Hess} \, \psi$  defines a non-degenerate quadratic form on 
\bqn 
T_{\alpha(c', x'',0)} M/ T_{\alpha (c',x'',0)}C
\eqn
if, and only if $\mathrm{Hess} \, \psi_{c'}$  defines a non-degenerate quadratic form on
\bqn
T_{{\alpha'(x'',0)}}M_{c'}/{T_{\alpha'(x'',0)}C_{c'}}.
\eqn
\end{lemma} \qed

Let us now state the main result of this section, the notation being the same as in the previous sections.

\begin{theorem}[Second Main Theorem]
\label{thm:II}
Let  $\mklm{(H_{i_1}), \dots, (H_{i_{N}})}$ be a maximal, totally ordered subset of non-principal isotropy types of the $G$-action on $M$, and  $\mathcal{Z}^{\rho_{i_1}\dots \rho_{i_N}}_{i_1\dots i_N}$ be  defined as in \eqref{eq:40}.  Consider the corresponding factorization 
\bqn
\Phi \circ \mathcal{Z}^{\rho_{i_1}\dots \rho_{i_N}}_{i_1\dots i_N} =\, ^{(i_1\dots i_N)} \tilde \Phi^{tot}=\tau_{i_1}(\sigma) \dots \tau_{i_N}(\sigma) \, ^{(i_1\dots i_N)}\tilde \Phi^ {wk}
\eqn
 of the phase function \eqref{eq:phase}. 
 Then, at each point of the critical manifold $\Crit(\, ^{(i_1\dots i_N)}\tilde \Phi^ {wk})$, the Hessian $\mathrm{Hess} \, ^{(i_1\dots i_N)}\tilde \Phi^ {wk}$ is transversally non-degenerate.
\end{theorem}
For the proof of Theorem \ref{thm:II} we need  the following
\begin{lemma}
\label{lemma:Reg}
 Let  $(x,\xi,g) \in \Crit{(\Phi)}$, and  $ x \in M({H_L})$. Then $(x,\xi,g) \in \mathrm{Reg} \,\Crit{(\Phi)}$, and $\mathrm{Hess} \, \Phi$ is transversally non-degenerate at $(x,\xi,g)$. 
 \end{lemma}

\begin{proof}
The first assertion is clear from \eqref{eq:x} - \eqref{eq:z}.  To see the second, consider the $1$-form  $\eta = \sum \xi_i d\tilde x_i$, and note that by \eqref{eq:y}
\bqn 
\eta_x \in \Omega \cap T_x^\ast Y,\, x \in M(H_L), \, g \in G_{{x}} \quad \Longrightarrow \quad g \cdot \eta_x = \eta_x.
\eqn
Since by \eqref{eq:vanx} the condition $\gd_x \Phi(x,\xi,g)=0$ is equivalent to $g \cdot \eta_x = \eta_x$, and 
\bqn 
\gd_\xi \Phi (x,\xi, g) =0 \quad \Longleftrightarrow  \quad gx=x, \qquad \quad \gd_g \Phi(x,\xi, g) =0 \quad \Longleftrightarrow \quad \eta_{gx} \in \Omega,
\eqn 
we obtain on  $T^\ast (Y\cap M(H_{L}))\times G$ the implication 
\bqn 
\gd_{\xi,g} \Phi (x,\xi,g) =0 \quad \Longrightarrow  \quad \gd_x \Phi(x,\xi,g) =0.
\eqn 
Let $\Phi_x(\xi,g)$ denote the phase function \eqref{eq:phase} regarded as a function of the coordinates $\xi,g$ alone, while $x$ is regarded as a parameter. 
Lemma \ref{lemma:A} then implies that  on $T^\ast (Y \cap M(H_{L}))\times G$ the study of the transversal Hessian of $\Phi$ can be reduced to the study of the transversal Hessian of $\Phi_x$. Now, with respect to the coordinates $\xi,g$, the Hessian of $\Phi_x$ is given by 
\bqn 
\left ( \begin{array}{cc}
0 & (d\tilde x_i)_x((\widetilde X_{j})_x) \\ (d\tilde x_j)_x((\widetilde X_{i})_x) &  (\gd_{X_i} \gd_{X_j} \Phi_x)(\xi,g) \\
\end{array} \right ),
\eqn 
where $\mklm{X_1, \dots, X_d}$ denotes a basis of $\g$. 
A computation then shows that the kernel of the corresponding linear transformation is equal  to $\mklm{(\tilde \xi,\tilde s): \sum \tilde \xi_i (d\tilde x_i)_{x} \in \mathrm{Ann}(T_{x} ( G \cdot x)), \sum \tilde s_j (\widetilde X_j)_x =0}\simeq T_{\xi,g}( \Crit \, \Phi_x)$. The lemma now follows by the following general observation.
Let $\mathcal{B}$ be a symmetric bilinear form on an $n$-dimensional $\mathbb{K}$-vector space $V$, and $B=(B_{ij})_{i,j}$ the corresponding Gramsian matrix with respect to a basis $\mklm{v_1,\dots,v_n}$ of $V$ such that 
\bqn 
\mathcal{B}(u,w) = \sum_{i,j} u_i w_j B_{ij}, \qquad  u=\sum u_i v_i, \quad w=\sum w_i v_i.
\eqn
We denote the linear operator given by $B$ with the same letter, and write 
\bqn 
V =\ker B \oplus W.
\eqn
Consider the restriction $\mathcal{B}_{|W \times W}$ of $\mathcal{B}$ to $W\times W$, and assume that $\mathcal{B}_{|W\times W}(u,w) =0$ for all $u \in W$, but $w\not=0$. Since the Euclidean scalar product in $V$ is non-degenerate, we necessarily must have $Bw=0$, and consequently $ w \in \ker  B \cap W=\mklm{0}$, which is a contradiction. Therefore $\mathcal{B}_{|W \times W}$ defines a non-degenerate symmetric bilinear form. 
\end{proof}

\begin{proof}[Proof of Theorem \ref{thm:II}]  For $\sigma_{i_1} \cdots \sigma_{i_N}\not=0$, the sequence of monoidal transformations  $\mathcal{Z}^{\rho_{i_1}\dots \rho_{i_N}}_{i_1\dots i_N}$ defined in \eqref{eq:40}  is a diffeomorphism, so that by the previous lemma 
\bqn 
\mathrm{Hess} ^{(i_1\dots i_N)} \tilde \Phi^{tot} (\sigma_{i_j},p^{(i_j)},\tilde v^{(i_N)}, \alpha^{(i_j)}, h^{(i_N)},\xi)
\eqn
is transversally non-degenerate at each point of $ \mathrm{Crit}(\, ^{(i_1\dots i_N)} \Phi^{tot})_{\sigma_{i_1}\cdots \sigma_{i_N}\not=0}$.  Next, one computes
\begin{align*}
\left (\frac{\gd^2 \, ^{(i_1\dots i_N)}\tilde \Phi^ {tot}}{\gd \gamma_k\gd \gamma_l}  \right )_{k,l} & = \tau_{i_1}(\sigma) \cdots  \tau_{i_N }(\sigma) 
\left (\frac{\gd^2 \, ^{(i_1\dots i_N)}\tilde \Phi^ {wk}}{\gd \gamma_k\gd \gamma_l}  \right )_{k,l} \\ &+
\left (\begin{array}{cc}
\left (   \frac{ \gd ^2 (\tau_{i_1}(\sigma) \cdots  \tau_{i_N }(\sigma))}{\gd \sigma_{i_r}\sigma_{i_s}} \right )_{r,s} & 0 \\ 0 & 0
\end{array}\right ) \, ^{(i_1\dots i_N)}\tilde \Phi^ {wk} +R,
\end{align*}
where $\gamma_k$ stands for any of the coordinates, and $R$ represents a matrix whose entries contain first order derivatives of $^{(i_1\dots i_N)}\tilde \Phi^ {wk}$ as factors. But since by \eqref{eq:38}
\bqn 
\mathrm{Crit}(\, ^{(i_1\dots i_N)} \tilde \Phi^{tot})_{\sigma_{i_1}\cdots \sigma_{i_N}\not=0} =\Crit(^{(i_1\dots i_N)}\tilde \Phi^ {wk})_{|\sigma_{i_1}\cdots \sigma_{i_N}\not=0}, 
\eqn
 we conclude that  the transversal Hessian of $^{(i_1\dots i_N)}\tilde \Phi^ {wk}$ does not degenerate along the manifold $\Crit(^{(i_1\dots i_N)}\tilde \Phi^ {wk})_{|\sigma_{i_1}\cdots \sigma_{i_N}\not=0}$. Therefore, it  remains to study the transversal Hessian of $^{(i_1\dots i_N)}\tilde \Phi^ {wk}$ in the case that any of the $\sigma_{i_j}$ vanishes, that is, along the exceptional divisor. Now, the proof of the first main theorem,  in particular  \eqref{eq:IV},  showed that
\bqn 
\gd _{\xi, \alpha^{(i_1)}, \dots, \alpha^{(i_N)}, h^{(i_N)}} \, ^{(i_1\dots i_N)}\tilde \Phi^ {wk}=0 \quad \Longrightarrow \quad \gd _{\sigma_{i_1}, \dots, \sigma_{i_N}, p^{(i_1)}, \dots, p^{(i_N)},\tilde v^{(i_N)}} \, ^{(i_1\dots i_N)}\tilde \Phi^ {wk}=0.
\eqn
 If therefore 
$$
\, ^{(i_1\dots i_N)}\tilde \Phi^ {wk}_{\sigma_{i_j}, p^{(i_j)},\tilde v^{(i_N)}}(\alpha^{(i_j)},  h^{(i_N)},\xi)
$$ 
denotes the weak transform of the phase function $\Phi$ regarded as a function of the variables $(\alpha^{(i_1)},\dots, \alpha^{(i_N)}, h^{(i_N)},\xi)$ alone, while the variables $(\sigma_{i_1},\dots,\sigma_{i_N}, p^{(i_1)},\dots, p^{(i_N)},\tilde v^{(i_N)})$ are kept fixed at constant values,
\bqn 
\Crit \big ( \, ^{(i_1\dots i_N)}\tilde \Phi^ {wk}_{\sigma_{i_j}, p^{(i_j)},\tilde v^{(i_N)}}\big )=\Crit \big ( \, ^{(i_1\dots i_N)}\tilde \Phi^ {wk}\big )  \cap \mklm{\sigma_{i_j}, p^{(i_j)},\tilde v^{(i_N)} = \, \, \text{constant}}. 
\eqn
Thus, the critical set of $\, ^{(i_1\dots i_N)}\tilde \Phi^ {wk}_{\sigma_{i_j}, p^{(i_j)},\tilde v^{(i_N)}}$ is equal to the fiber over $(\sigma_{i_j}, p^{(i_j)},\tilde v^{(i_N)})$ of the vector bundle
\bqn 
\Big ((\sigma_{i_j}, p^{(i_j)},\tilde v^{(i_N)}), G _{\tilde v^{(i_N)}} \times \mathrm{Ann} \big ( \bigoplus\limits_{j=1}^N E^{(i_j)}_{x^{(i_1\dots i_N)}} \oplus F^{(i_N)}_{x^{(i_1\dots i_N)}} \big ) \Big ) \mapsto  (\sigma_{i_j}, p^{(i_j)},\tilde v^{(i_N)}),
\eqn
and in particular  a smooth submanifold. Lemma \ref{lemma:A} then implies that the study of the transversal Hessian  of $\, ^{(i_1\dots i_N)}\tilde \Phi^ {wk}$ can be reduced to the study of the transversal Hessian of $\, ^{(i_1\dots i_N)}\tilde \Phi^ {wk}_{\sigma_{i_j}, p^{(i_j)},\tilde v^{(i_N)}}$. The crucial fact is now contained in the following 
\begin{proposition}
\label{prop:1}
Assume that 
$\sigma_{i_1} \cdots \sigma_{i_N}=0$. Then 
\bqn 
\ker  \mathrm{Hess} \, ^{(i_1\dots i_N)}\tilde \Phi^ {wk}_{\sigma_{i_j}, p^{(i_j)},\tilde v^{(i_N)}}(0,\dots, 0, h^{(i_N)},\xi)\simeq T_{(0, \dots, 0,h^{(i_N)},\xi)}  \mathrm{Crit} \big (\,^{(i_1\dots i_N)}\tilde \Phi^ {wk}_{\sigma_{i_j}, p^{(i_j)},\tilde v^{(i_N)}} \big )
\eqn
for all $(0,\dots, 0, h^{(i_N)},\xi) \in \mathrm{Crit} \big (\,^{(i_1\dots i_N)}\tilde \Phi^ {wk}_{\sigma_{i_j}, p^{(i_j)},\tilde v^{(i_N)}} \big )$, and arbitrary $p^{(i_j)}$, $\tilde v^{(i_j)}$.
\end{proposition}
\begin{proof}
Since $\sigma_{i_1} \cdots \sigma_{i_N}=0$, we have
\begin{align*}
 ^{(i_1\dots i_N)}\tilde \Phi^ {wk}_{\sigma_{i_j}, p^{(i_j)},\tilde v^{(i_N)}}&=  \Big \langle \Xi \cdot  \, ^T \big ( dp_1 ^{(i_1)} (\tilde A^{(i_1)}_{p^{(i_1)}}), \dots, 0, dp_1 ^{(i_2)} (\tilde A^{(i_2)}_{p^{(i_2)}}),\dots, 0,\dots,  dp_1 ^{(i_N)} (\tilde A^{(i_N)}_{p^{(i_N)}}),\dots, \\ &  \theta_1 ^{(i_N)} \big (  \tilde v^{(i_N)} \big ) -\theta_1 ^{(i_N)} \big ( (h^{(i_N)})_{\ast, p^{(i_1)}} \tilde v^{(i_N)} \big ), \dots \big ), \xi \Big \rangle.
\end{align*}
The only non-vanishing second order derivatives at a critical point $(0, \dots, 0, h^{(i_N)},\xi)$ therefore read
\begin{align*}
\gd_{\alpha^{(i_j)}_s} \gd _{\xi_r} \, ^{(i_1\dots i_N)}\tilde \Phi^ {wk}_{\sigma_{i_j}, p^{(i_j)},\tilde v^{(i_N)}} =&\Big [ \Xi \cdot \, ^T \big (0, \dots, 0, dp_1^{(i_j)}((\widetilde A^{(i_j)}_{s})_{p^{(i_j)}}),\dots, 0, \dots, 0 \big )  \Big ] _r, \\
\gd_{\beta^{(i_N)}_s} \gd_{\xi_{r}} \, ^{(i_1\dots i_N)}\tilde \Phi^ {wk}_{\sigma_{i_j}, p^{(i_j)},\tilde v^{(i_N)}} =&\Big [ \Xi \cdot \, ^T \big (0, \dots, 0, d\theta_1^{(i_N)}((\widetilde B^{(i_N)}_{s})_{\tilde v^{(i_N)}}),\dots \big )  \Big ] _r, \\
\gd_{\beta^{(i_N)}_r} \gd_{\beta^{(i_N)}_{s}} \, ^{(i_1\dots i_N)}\tilde \Phi^ {wk}_{\sigma_{i_j}, p^{(i_j)},\tilde v^{(i_N)}} =&-\Big \langle \Xi \cdot  \, ^T \big ( 0, \dots, 0,     \theta_1 ^{(i_N)} \big (\lambda(B_r^{(i_N)}) \lambda(B_s^{(i_N)}) \,  \tilde v^{(i_N)} \big ), \dots \big ), \xi \Big \rangle,
\end{align*}
where we used canonical coordinates of the first kind on $G_{p^{(i_N)}}$ of the form $\e{\sum \beta_m^{(i_N)} B^{(i_N)}_m} \cdot h^{(i_N)}$. 
Consequently, the Hessian of the function $\, ^{(i_1\dots i_N)}\tilde \Phi^ {wk}_{\sigma_{i_j}, p^{(i_j)},\tilde v^{(i_N)}}$  with respect to the coordinates $\xi, \alpha^{(i_j)}, \beta^{(i_N)}$ is given on its 
critical set by the matrix
\bq
\label{eq:Hesswk}
\left ( \begin{array}{cc} 
\Xi  &0 \\ 0  & \1_d  
\end{array} \right )
\left ( \begin{array}{cc} 
0 &\mathcal{E} \\ ^T\mathcal{E} & \mathcal{F}
\end{array} \right )
\left ( \begin{array}{cc} 
^T \Xi  & 0  \\ 0 & \1_d  
\end{array} \right ),
\eq
where the $(n \times d)$-matrix $\mathcal{E}$  is defined by 
\bqn 
\mathcal{E}= \left ( \begin{array}{ccccc}
 dp_r^{(i_1)}((\widetilde A^{(i_1)}_{s})_{p^{(i_1)}}) &0 & 0 & \dots & 0 \\
 0 & 0 & 0 & \dots  & 0\\
 0 &  0 & dp_r^{(i_2)} \big ( (\widetilde A_s^{(i_2)})_{ p^{(i_2)}}\big ) &\dots  & 0\\
 \vdots & \vdots & \vdots & \ddots & \vdots \\
 0 &  0 &  0 &  \dots  & d\theta^{(i_N)}_r((\widetilde B_s^{(i_N)})_{\tilde v^{(i_N)}}) \\
\end{array} \right ),
\eqn
and the $(d \times d)$-matrix $\mathcal{F}$ by
\bqn 
\mathcal{F}= \left ( \begin{array}{cc}
0 & 0 \\ 0 & -\Big \langle \Xi \cdot  \, ^T \big ( 0, \dots, 0,    \theta_1 ^{(i_N)} \big (\lambda(B_r^{(i_N)}) \lambda(B_s^{(i_N)}) \,  \tilde v^{(i_N)} \big ), \dots \big ), \xi \Big \rangle
\end{array} \right ). 
\eqn
Since $\Xi$ is invertible, the kernel of the linear transformation corresponding to \eqref{eq:Hesswk}  is isomorphic to the kernel of the linear transformation defined by 
\bqn 
\left ( \begin{array}{cc} 
0 &\mathcal{E} \\ ^T\mathcal{E} & \mathcal{F} 
\end{array} \right )
\eqn
which we shall now compute.  Cleary,  the column vector $^T(\tilde \xi, \tilde \alpha^{(i_1)}, \dots, \tilde \alpha^{(i_N)}, \tilde \beta^{(i_N)})$ lies in the kernel if and only if

\medskip
\begin{tabular}{ll}
{(a)} & $\sum_s  \tilde \alpha_s^{(i_j)} (\widetilde A_s^{(i_j)})_{ p^{(i_j)}}=0$ for all $1 \leq j \leq N$, $ \sum_m  \tilde \beta_m^{(i_N)} (\widetilde B_m^{(i_N)})_{ \tilde v^{(i_N)}}=0$; \\[2pt]
{($b_1$)} & $\sum _{r=1} ^{n - c ^{(i_1)}} \tilde \xi_r dp_r^{(i_1)}\big ( T_{p^{(i_1)}} ( G \cdot p^{(i_1)}
 ) \big )=0$; \\[2pt]
{($b_2$)} & $\sum_{r=1}^{c ^{(i_{j-1})}- c ^{(i_j)} -1} \tilde \xi_{n-c^{(i_{j-1})}+r+1} dp^{(i_j)}_r\big (T_{p^{(i_j)}}(G_{p^{(i_{j-1} )}} \cdot  {p^{(i_j)}})\big )=0$ for all  $2 \leq j \leq N$ ;\\[2pt]
{(c)} &$\sum_{r=1}^{c^{(i_N)}} \tilde \xi_{n - c^{(i_N)} + r} d\theta^{(i_N)}_r \big (T_{\tilde v^{(i_N)}} ( G_{p^{(i_N)}} \cdot \tilde v^{(i_N)})\big )=0$.
\end{tabular}
\medskip

\noindent
In this  case, the vector $^T ( \tilde \xi', \tilde \alpha ^{(i_1)}, \dots, \tilde \alpha^{(i_N)}, \tilde \beta^{(i_N)})$ lies in the kernel of \eqref{eq:Hesswk}, where $\tilde \xi ' = \tilde \xi  \cdot \Xi^{-1}$. 
Let now $E^{(i_j)}$,  $F^{(i_N)}$, and  $V^{(i_1\dots i_N)}$ be defined as in \eqref{eq:EF} and \eqref{eq:V}, and remember that we have the isomorphisms \eqref{eq:isom}. 
 For condition (a) to hold, it is  necessary  and sufficient that 
\bqn 
\tilde \alpha^{(i_j)} =0, \quad 1 \leq j \leq N, \qquad \sum_m  \tilde \beta_m^{(i_N)} (\widetilde B_m^{(i_N)})_{ \tilde v^{(i_N)}}=0.
\eqn 
On the other hand, condition $(b_1)$ means that on $T_{p^{(i_1)}} ( G \cdot p^{(i_1)})$ we have 
\bqn 
0=\sum _{r=1} ^{n - c ^{(i_1)}} \tilde \xi_r dp_r^{(i_1)}=\sum _{r=1} ^{n - c ^{(i_1)}} (\tilde \xi' \cdot \Xi)_r dp_r^{(i_1)}=\sum_{i=1}^n \tilde \xi ' d \tilde x_i,
\eqn
where $(\tilde x_1, \dots \tilde x_n)=( p_1^{(i_1)}, \dots,p_{n-c^{(i_1)}}^{(i_1)}, \theta_1^{(i_1)}, \dots, \theta_{c^{(i_1)}}^{(i_1)}) $ are the coordinates introduced in Section \ref{sec:DP}. Indeed, 
\bqn 
d \theta_s^{(i_1)} ( \tilde A^{(i_1)}_{p^{(i_1)}})= \tilde A ^{(i_1)} _{p^{(i_1)}} ( \theta_s ^{(i_1)}) =0 \qquad \forall s=1, \dots, c^{(i_1)}, \quad A^{(i_1)} \in \g^\perp_{p^{(i_1)}}, 
\eqn
$M_{i_1}(H_{i_1})$ being $G$-invariant. Therefore $\sum_{i=1}^n \tilde \xi ' (d \tilde x_i) _{p^{(i_1)}} \in \mathrm{Ann} \big (E^{(i_1)}_{p^{(i_1)}}\big )$. In the same way, condition $(b_2)$ implies that on $T_{p^{(i_j)}} ( G_{p^{(i_{j-1})}} \cdot p^{(i_j)})$ 
\begin{align*}
0=&\sum_{r=1}^{c ^{(i_{j-1})}- c ^{(i_j)} -1} ( \tilde \xi' \cdot \Xi)_{n-c^{(i_{j-1})}+r+1} dp^{(i_j)}_r=\sum_{r=1}^{c ^{(i_{j-1})}- c ^{(i_j)} -1} ( \tilde \xi'' \cdot \Xi^{(i_1\dots i_{j-1})})_{n-c^{(i_{j-1})}+r+1} dp^{(i_j)}_r \\
+& \sum_{s=1}^{n- c ^{(i_j)}} ( \tilde \xi'' \cdot \Xi^{(i_1\dots i_{j-1})})_{n-c^{(i_{j})}+s} d\theta^{(i_j)}_s+( \tilde \xi'' \cdot \Xi^{(i_1\dots i_{j-1})})_{n-c^{(i_{j-1})}+1} d\tau_{i_{j-1}} \\ 
=&\sum_{r=1}^{n- c ^{(i_{j-1})}} \tilde \xi'' _{n-c^{(i_{j-1})}+s} d\theta^{(i_{j-1})}_s,
\end{align*}
where we put $\tilde \xi '' = \tilde \xi' \cdot \Xi ^{(i_1)} \cdot \cdots \cdot \Xi^{(i_1 \dots i_{j-2})}$. Hereby we expressed the differential forms $d\theta^{(i_{j-1})}_s$ in terms of the differential forms $d\tau_{i_{j-1}}$, $dp^{(i_j)}_r$, and  $d\theta^{(i_j)}_s$, and took into account that the forms $d\tau_{i_{j-1}}$ and $d\theta^{(i_j)}_s$ vanish on $T_{p^{(i_j)}} ( G_{p^{(i_{j-1})}} \cdot p^{(i_j)}) \subset (\nu_{i_1 \dots i_{j-1}})_{p^{(i_{j-1})}}$. Now, if $v \in (\nu_{i_1 \dots i_{j-1}})_{p^{(i_{j-1})}}$, and $\sigma_v(t) = \exp_{p^{(i_{j-1})}} t v$,
\bqn 
v(p^{(i_{j-1})}_r)= \frac d{dt} p^{(i_{j-1})}_r(\sigma_v(t))_{t=0}=0,
\eqn
so that the forms $dp^{(i_{j-1})}_r$ must vanish on $(\nu_{i_1 \dots i_{j-1}})_{p^{(i_{j-1})}}$. In case that $3 \leq j \leq N$, the same argument shows that the forms $d\tau_{i_{j-2}}$ also vanish on $(\nu_{i_1 \dots i_{j-1}})_{p^{(i_{j-1})}}$, and proceeding inductively this way,  we see that conditions $(b_1)$ and $(b_2)$ are equivalent to 
\bqn 
\sum_{i=1}^n \tilde \xi ' (d \tilde x_i) _{p^{(i_1)}} \in \mathrm{Ann} \big (E^{(i_j)}_{p^{(i_1)}}\big ) \qquad \forall \quad  j=1, \dots, N.
\eqn
Similarly, one deduces that condition $(c)$ is equivalent to 
\bqn 
\sum_{i=1}^n \tilde \xi ' (d \tilde x_i) _{p^{(i_1)}} \in \mathrm{Ann} \big (F^{(i_N)}_{p^{(i_1)}}\big ).
\eqn
 On the other hand, by \eqref{eq:C},
\begin{gather*}
T_{(0, \dots, 0,h^{(i_N)},\xi)}  \mathrm{Crit} \big (\,^{(i_1\dots i_N)}\tilde \Phi^ {wk}_{\sigma_{i_j}, p^{(i_j)},\tilde v^{(i_N)}} \big )= \Big \{( \tilde \alpha^{(i_1)}, \dots, \tilde \alpha^{(i_N)}, \tilde \beta ^{(i_N)}, \tilde \xi' ): \tilde \alpha^{(i_j)}=0, \\  \sum_m \tilde  \beta^{(i_N)}_m \lambda(B^{(i_N)}_m) \in \g_{\tilde v^{(i_N)}}, \, \sum \tilde \xi_i' (d\tilde x_i) _{p^{(i_1)}} \in \mathrm{Ann}\Big ( \bigoplus_{j=1}^N E^{(i_j)}_{p^{(i_1)}}\oplus F^{(i_N)}_{p^{(i_1)}}\Big ) \Big \},
\end{gather*}
and the proposition follows.
\end{proof}
The previous proposition now  implies that for $\sigma_{i_1}\cdots \sigma_{i_N}=0$, the Hessian of  $
\,^{(i_1\dots i_N)}\tilde \Phi^ {wk}_{\sigma_{i_j}, p^{(i_j)},\tilde v^{(i_N)}}$
is transversally non-degenerate at each point $(0, \dots, 0,h^{(i_N)},\xi)$ of its critical set, and Theorem \ref{thm:II} follows with  Lemma \ref{lemma:A}.
\end{proof}

\medskip

\section{Asymptotics for the integrals  $I_{i_1\dots i_N}^{\rho_{i_1}\dots \rho_{i_N}}(\mu)$ }
\label{sec:INT}

We are now in position to describe the  asymptotic behavior of   the integrals $I_{i_1\dots i_N}^{\rho_{i_1}\dots \rho_{i_N}}(\mu)$ defined in \eqref{eq:N} using the stationary phase theorem. 
Let  $\mklm{(H_{i_1}), \dots, (H_{i_{N}})}$ be a maximal, totally ordered subset of non-principal isotropy types, and 
  $ \zeta^{\rho_{i_1}}_{i_1} \circ \dots \circ  \zeta^{\rho_{i_1}\dots \rho_{i_N}}_{i_1\dots i_N}$  a corresponding  local realization of the sequence of monoidal transformations $ \zeta^{(1)} \circ \dots \circ \zeta^{(N)}$   in a set of $(\theta^{(i_1)}, \dots, \theta^{(i_N)})$-charts labeled by the indices $\rho_{i_1},\dots ,\rho_{i_N}$.   Since the considered integrals are absolutely convergent, we can interchange the order of integration by Fubini, and write
\bq
\label{eq:fub}
I_{i_1\dots i_N}^{\rho_{i_1}\dots \rho_{i_N}}(\mu)= \int_{ (-1,1)^N}  \hat  J_{i_1\dots i_N}^{\rho_{i_1}\dots \rho_{i_N}}\Big (  \mu \cdot {\tau_{i_1}\cdots \tau_{i_N}} \Big ) \prod_{j=1}^N |\tau_{i_j}|^{c^{(i_j)} + \sum _{r=1}^j d^{(i_r)} -1} \d \tau_{i_N} \dots \d \tau_{i_1},
\eq
where we set
\begin{align}
\label{eq:62}
\begin{split}
\hat  J_{i_1\dots i_N}^{\rho_{i_1}\dots \rho_{i_N}}(\nu)= \int_{M_{i_1}(H_{i_1})} \Big [ \int_{\gamma^{(i_1)}((S_{i_1})_{p^{(i_1)}})_{i_2}(H_{i_2})} \dots  \Big [ &  \int_{\gamma^{(i_{N-1})}((S_{i_1\dots i_{N-1}})_{p^{(i_{N-1})}})_{i_{N}}(H_{i_{N}}) } \\
 \Big [ \int_{\gamma^{(i_N)}((S_{i_1\dots i_{N}})_{p^{(i_N)}})\times G_{p^{(i_{N})}}\times  \stackrel \circ D_ \iota( \g_{p^{(i_{N})}}^\perp) \times \cdots \times \stackrel \circ D_ \iota (\g_{p^{(i_{1})}}^\perp)\times \rn } & e^{i\nu \, ^{(i_1\dots i_N)} \tilde \Phi ^{wk,pre}}  \, a_{i_1\dots i_N}^{\rho_{i_1}\dots \rho_{i_N}} \,   {\mathcal{J}}_{i_1\dots i_N}^{\rho_{i_1}\dots \rho_{i_N}} \\
  \d \xi  \d A^{(i_1)} \dots  \d A^{(i_N)}  \d h^{(i_N)} \d \tilde v^{(i_N)} \Big ]   \d p^{(i_{N})} \dots  &\Big ]  \d p^{(i_{2})} \Big ] \d p^{(i_{1})},
  \end{split}
\end{align}
according to  the notation introduced in Theorem \ref{thm:MT1}, 
and introduced the new parameter
\bqn
\nu = \mu\cdot  {\tau_{i_1}\cdots \tau_{i_N}}.
\eqn
Now,  for a sufficiently small  $\epsilon>0 $ to be chosen later we define
\begin{align*}
\, ^1I_{i_1\dots i_N}^{\rho_{i_1}\dots \rho_{i_N}}(\mu)&= \int_{\prod_{j=1}^N(-1,1)\setminus (-\epsilon,\epsilon)}  \hat J_{i_1\dots i_N}^{\rho_{i_1}\dots \rho_{i_N}}\Big (  \mu\cdot {\tau_{i_1}\cdots \tau_{i_N}} \Big ) \prod_{j=1}^N |\tau_{i_j}|^{c^{(i_j)} + \sum _{r=1}^j d^{(i_r)} -1} \d \tau_{i_N} \dots \d \tau_{i_1},\\
\, ^2 I_{i_1\dots i_N}^{\rho_{i_1}\dots \rho_{i_N}}(\mu)&= \int_{(-\epsilon,\epsilon)^N}  \hat J_{i_1\dots i_N}^{\rho_{i_1}\dots \rho_{i_N}}\Big (  \mu\cdot {\tau_{i_1}\cdots \tau_{i_N}} \Big ) \prod_{j=1}^N |\tau_{i_j}|^{c^{(i_j)} + \sum _{r=1}^j d^{(i_r)} -1} \d \tau_{i_N} \dots \d \tau_{i_1} . 
\end{align*}
\begin{lemma}
\label{lemma:kappa}
One has $c^{(i_j)} + \sum _{r=1}^j d^{(i_r)} -1\geq \kappa$ for arbitrary $j=1,\dots, N$, where $\kappa$ denotes the dimension of $G/H_L$.
\end{lemma}
\begin{proof}
We first note that $c^{(i_j)} = \dim (\nu_{i_1\dots i_j})_{p^{(i_j)}} \geq \dim G_{p^{(i_j)}} \cdot x^{(i_{j+1}\dots i_N)} +1$. 
Indeed, $(\nu_{i_1\dots i_j})_{p^{(i_j)}}$ is an orthogonal $G_{p^{(i_j)}}$-space, so that the dimension of the $G_{p^{(i_j)}}$-orbit of $x^{(i_{j+1}\dots i_N)}\in\gamma^{(i_j)}( (S^+_{i_1 \dots i_j})_{p^{(i_j)}})$ can be at most $c^{(i_j)}-1$. Now, under the assumption $\sigma_{i_1}\cdots \sigma_{i_N}\not=0$, \eqref{eq:gdecomp} and \eqref{eq:15} imply
\begin{align*}
T_{x^{(i_{j+1}\dots i_N)}} &(G_{p^{(i_j)}} \cdot x^{(i_{j+1}\dots i_N)}) \simeq T_{x^{(i_{1}\dots i_N)}} (G_{p^{(i_j)}} \cdot x^{(i_{1}\dots i_N)}) \\
=  \bigoplus _{l=j+1}^N &\tau_{i_{1}}\dots \tau_{i_{l-1}} E^{(i_l)} _{x^{(i_1 \dots i_N)}}  \oplus  \tau_{i_{1}}\dots \tau_{i_N} F^{(i_N)}_{x^{(i_1 \dots i_N)}}, 
\end{align*}
where the distributions $E^{(i_{l})}$, $F^{(i_{N})}$ where defined in \eqref{eq:EF}. On then computes
\begin{align*}
\dim G_{p^{(i_j)}}  \cdot x^{(i_{j+1}\dots i_N)}= \sum_{l=j+1} ^N \dim E^{(i_l)} _{x^{(i_1 \dots i_N)}}+ \dim F^{(i_N)}_{x^{(i_1 \dots i_N)}},
\end{align*}
which implies 
\bqn
c^{(i_j)} \geq \sum_{l=j+1} ^N \dim E^{(i_l)}_{x^{(i_1 \dots i_N)}} + \dim F^{(i_N)}_{x^{(i_1 \dots i_N)}} +1.
\eqn
But since $E^{(i_l)}_{x^{(i_1 \dots i_N)}}\simeq G_{p^{(i_{l-1})}} \cdot p^{(i_l)}$ for any $\sigma$, the last inequality holds for arbitrary $\sigma$, too.  
On the other hand, one has
\begin{align*}
d^{(i_j)}&=\dim \g_{p^{(i_j)}}^\perp =\dim [\lambda( \g_{p^{(i_j)}}^\perp) \cdot p^{(i_j)} ]=\dim [ \lambda( \g_{p^{(i_j)}}^\perp)  \cdot x^{(i_j\dots i_N)} ]= \dim E^{(i_j)}_{x^{(i_1 \dots i_N)}}.
\end{align*}
The assertion now follows with \eqref{eq:kappa}.
\end{proof}
As a consequence of the lemma, we obtain for $\, ^2 I_{i_1\dots i_N}^{\rho_{i_1}\dots \rho_{i_N}}(\mu)$ the estimate
\begin{align}
\label{eq:I2}
\begin{split}
\big | \, ^2 I_{i_1\dots i_N}^{\rho_{i_1}\dots \rho_{i_N}}(\mu)\big | &\leq C \int_{(-\epsilon,\epsilon)^N}  \prod_{j=1}^N |\tau_{i_j}|^{c^{(i_j)} + \sum _{r=1}^j d^{(i_r)} -1} \d \tau_{i_N} \dots \d \tau_{i_1} \\
&\leq C \int_{(-\epsilon,\epsilon)^N}  \prod_{j=1}^N |\tau_{i_j}|^{\kappa} \d \tau_{i_N} \dots \d \tau_{i_1} =\frac{2C}{\kappa+1} \epsilon^{N (\kappa+1)}
\end{split}
\end{align}
for some $C>0$. Let us now turn to the integral $\, ^1 I_{i_1\dots i_N}^{\rho_{i_1}\dots \rho_{i_N}}(\mu)$. After performing the change of variables $\delta_{i_1\dots i_N}$
one obtains for $\, ^1 I_{i_1\dots i_N}^{\rho_{i_1}\dots \rho_{i_N}}(\mu)$ the expression
\begin{gather*}
 \int\limits _{\epsilon < |\tau_{i_j} (\sigma)|<1} \hspace{-.5cm} J_{i_1\dots i_N}^{\rho_{i_1}\dots \rho_{i_N}}\Big ( \mu\cdot {\tau_{i_1}(\sigma)\cdots \tau_{i_N}(\sigma)} \Big )  \prod_{j=1}^N |\tau_{i_j}(\sigma)|^{c^{(i_j)} + \sum _{r=1}^j d^{(i_r)} -1} \, |\det D\delta_{i_1\dots i_N}(\sigma) | \d \sigma,
\end{gather*}
where $ J_{i_1\dots i_N}^{\rho_{i_1}\dots \rho_{i_N}}(\nu)$ is defined by the right hand side of  \eqref{eq:62}, with $ ^{(i_1\dots i_N)} \tilde \Phi ^{wk,pre } $ being replaced by 
 $ ^{(i_1\dots i_N)} \tilde \Phi ^{wk}_\sigma $, which denotes the weak transform $ ^{(i_1\dots i_N)} \tilde \Phi ^{wk}$ as a function of the variables $ p^{(i_j)},  \tilde v^{(i_N)}$, $\alpha^{(i_j)}, h^{(i_N)}, \xi$ alone, while the variables $\sigma=(\sigma_{i_1},\dots \sigma_{i_N})$ are regarded as parameters. 
 \begin{theorem}
\label{thm:J}
Let $\sigma=(\sigma_{i_1},\dots, \sigma_{i_N})$ be a fixed set of parameters. Then, for every $\tilde N \in \N$ there exists a constant $C_{\tilde N, ^{(i_1\dots i_N)} \tilde \Phi^{wk}_\sigma}>0$ such that 
\bqn
|J_{i_1\dots i_N}^{\rho_{i_1}\dots \rho_{i_N}} (\nu) -(2\pi |\nu|^{-1})^{\kappa}\sum_{j=0} ^{\tilde N-1} |\nu|^{-j} Q_j (^{(i_1\dots i_N)} \tilde \Phi ^{wk}_\sigma;a_{i_1\dots i_N} \mathcal{J}_{i_1\dots i_N})| \leq C_{\tilde N,^{(i_1\dots i_N)} \tilde \Phi ^{wk}_\sigma} |\nu|^{-\tilde N},
\eqn
with estimates for the coefficients $Q_j$, and an explicit expression for $Q_0$. Moreover,  the constants $C_{\tilde N, ^{(i_1\dots i_N)} \tilde \Phi ^{wk}_\sigma}$ and the coefficients $Q_j$ have uniform bounds in $\sigma$.
\end{theorem}
\begin{proof}
In what follows, we regard  $\tilde {\mathcal{M}}^{(N)}$ as a Riemannian manifold with the product metric induced by the Riemannian metrics on $\gamma^{(i_{j-1})} ((S_{i_1\dots i_{j-1}})_{p^{(i_{j-1})}})_{i_j}(H_{i_j})$, $\gamma^{(i_{N})} ((S_{i_1\dots i_{N}})_{p^{(i_{N})}})$, $\g^\perp_{p^{(i_j)}}$, $G_{p^{(i_N)}}$,  and $(-1,1)^N$, and  corresponding volume density. Similarly, $\tilde X$ carries a Riemannian structure, being a paracompact manifold. 
As a consequence of the main theorems, and Lemma \ref{lemma:A}, together with the observations preceding Proposition \ref{prop:1}, the phase function $^{(i_1\dots i_N)} \tilde \Phi ^{wk}_\sigma$ has a clean critical set for any value of the parameters $\sigma$. That is, in the relevant charts we have 
\begin{itemize}
\item the critical set of $ ^{(i_1\dots i_N)} \tilde \Phi ^{wk}_\sigma $ is a $\Cinft$-submanifold of codimension $2\kappa$ for arbitrary $\sigma$;
\item  the Hessian of $  \,^{(i_1\dots i_N)} \tilde \Phi ^{wk}_\sigma$ is transversally non-degenerate at each point of  its critical set.
\end{itemize}
Thus, the necessary conditions for applying the principle of the stationary phase to the integral $J_{i_1\dots i_N}^{\rho_{i_1}\dots \rho_{i_N}}(\nu)$ are fulfilled, and we obtain the desired asymptotic expansion by  Theorem \ref{thm:SP}. Note that the amplitude $a_{i_1 \dots i_N}^{\rho_{i_1} \dots \rho_{i_N}}$ might depend on $\mu$, compare the expression for $O(\mu^{n-2})$ in Theorem \ref{thm:Rt}, the dependence being of the form $a_\gamma(t, \kappa_\gamma (x ), \mu \eta)$, where $a_\gamma \in \mathrm{S}^0_{phg}$ is a classical symbol of order $0$. But since for large $|\eta|$
 \footnote{Hereby we are taking into account that due to the presence of the factor $\Delta_{\epsilon, R} (\zeta_\gamma(t, \kappa_\gamma(x), \eta) )$, $|\eta|$ is bounded away from zero.} 
\bqn 
\big | \gd ^\alpha_\eta a_\gamma ( t, \kappa _\gamma (x), \mu \eta) \big | =|\mu|^{|\alpha|} \big | (\gd^\alpha_\eta a _\gamma)( t, \kappa_\gamma(x), \mu \eta) \big | \leq C |\eta|^{-|\alpha|},
\eqn
this dependence does not interfer with the asymptotics. To see  the existence of the uniform bounds, note that by  Remark \ref{rmk:A} we have
\bqn 
C_{\tilde N, ^{(i_1\dots i_N)} \tilde \Phi ^{wk}_\sigma} \leq C'_{\tilde N} \sup_{ p^{(i_j)},  \tilde v^{(i_N)}, \alpha^{(i_j)}, h^{(i_N)},\xi} \norm{\Big ({\mathrm{Hess} \,  ^{(i_1\dots i_N)} \tilde \Phi ^{wk}_\sigma}_{|N \Crit ( ^{(i_1\dots i_N)} \tilde \Phi ^{wk}_\sigma)}\Big ) ^{-1}}.
\eqn
But since by Lemma \ref{lemma:A} the transversal Hessian 
\bqn 
{\mathrm{Hess} \,  ^{(i_1\dots i_N)} \tilde \Phi ^{wk}_\sigma}_{|N_{ (p^{(i_j)},  \tilde v^{(i_N)}, \alpha^{(i_j)}, h^{(i_N)},\xi ) } \Crit ( ^{(i_1\dots i_N)} \tilde \Phi ^{wk}_\sigma)}
\eqn 
is given by 
\bqn 
{\mathrm{Hess} \,  ^{(i_1\dots i_N)} \tilde \Phi ^{wk}}_{|N_{ (\sigma_{i_j}, p^{(i_j)},  \tilde v^{(i_N)}, \alpha^{(i_j)}, h^{(i_N)},\xi ) } \Crit ( ^{(i_1\dots i_N)} \tilde \Phi ^{wk})},
\eqn 
we finally obtain the estimate
\bqn 
C_{\tilde N, ^{(i_1\dots i_N)} \tilde \Phi ^{wk}_\sigma} \leq C'_{\tilde N} \sup_{\sigma_{i_j},  p^{(i_j)},  \tilde v^{(i_N)}, \alpha^{(i_j)}, h^{(i_N)},\xi} \norm{\Big ({\mathrm{Hess} \,  ^{(i_1\dots i_N)} \tilde \Phi ^{wk}}_{|N \Crit ( ^{(i_1\dots i_N)} \tilde \Phi ^{wk})}\Big ) ^{-1}} \leq C_{\tilde N, {i_1\dots i_N}}
\eqn
by a constant independent of $\sigma$. Similarly, one can show the existence of bounds of the form
\bqn 
|Q_j(^{(i_1\dots i_N)} \tilde \Phi ^{wk}_\sigma; a_{i_1\dots i_N} \Phi_{i_1\dots i_N}) |\leq \tilde C_{j, {i_1\dots i_N}},
\eqn
with constants $\tilde C_{j, {i_1\dots i_N}}$ independent of $\sigma$. 
\end{proof}

\begin{remark}
Before going on, let us remark that for the computation of the integrals $^1I_{i_1\dots i_N}^{\rho_{i_1}\dots \rho_{i_N}}(\mu)$ it is only necessary to have an asymptotic expansion for the integrals $J_{i_1\dots i_N}^{\rho_{i_1}\dots \rho_{i_N}} (\nu)$ in the case that $\sigma_{i_1} \cdots \sigma_{i_N}\not=0$, which can also be  obtained without the main theorems using only the factorization of the phase function $\Phi$ given by the resolution process, together with Lemma \ref{lemma:Reg}. Nevertheless, the main consequence to be drawn from the main theorems is that the constants $C_{\tilde N, ^{(i_1\dots i_N)} \tilde \Phi ^{wk}_\sigma}$ and the coefficients $Q_j$ in Theorem \ref{thm:J} have uniform bounds in $\sigma$. 
\end{remark}

As a consequence of Theorem \ref{thm:J}, we obtain for arbitrary $\tilde N\in \N$
\begin{gather*}
| J_{i_1\dots i_N}^{\rho_{i_1}\dots \rho_{i_N}} (\nu)  -(2\pi |\nu|^{-1})^{\kappa} Q_0( ^{(i_1\dots i_N)} \tilde \Phi ^{wk}_\sigma;a_{i_1\dots i_N} \Phi_{i_1\dots i_N})| \\ 
\leq 
\Big |J_{i_1\dots i_N}^{\rho_{i_1}\dots \rho_{i_N}} (\nu) -(2\pi |\nu|^{-1})^{\kappa}\sum_{l=0} ^{\tilde N-1} |\nu|^{-l} Q_l (^{(i_1\dots i_N)} \tilde \Phi ^{wk}_\sigma;a_{i_1\dots i_N} \Phi_{i_1\dots i_N})\Big | \\
  + (2\pi |\nu|^{-1})^{\kappa} \sum_{l=1} ^{\tilde N-1} |\nu|^{-l} |Q_l (^{(i_1\dots i_N)} \tilde \Phi ^{wk}_\sigma;a_{i_1\dots i_N} \Phi_{i_1\dots i_N})| \leq c_1 |\nu|^{-\tilde N}+c_2 |\nu|^{-\kappa} \sum_{l=1}^{\tilde N-1} |\nu|^{-l}
\end{gather*}
with constants $c_i>0$ independent  of both $\sigma$ and $\nu$.
From this we deduce
\begin{align*}
& \Big | \, ^1I_{i_1\dots i_N}^{\rho_{i_1}\dots \rho_{i_N}}(\mu)  - (2\pi /\mu)^{\kappa}  \int _{\epsilon < |\tau_{i_j}(\sigma) |< 1}   Q_0 \prod_{j=1}^N |\tau_{i_j}(\sigma)|^{c^{(i_j)} + \sum _{r=1}^j d^{(i_r)} -1-\kappa}  |\det D\delta_{i_1\dots i_N}(\sigma) | \d \sigma \Big| \\&\leq c_3 \mu^{-\tilde N}   \int_{\epsilon < |\tau_{i_j}(\sigma) |< 1}   \prod_{j=1}^N |\tau_{i_j}(\sigma)|^{c^{(i_j)} + \sum _{r=1}^j d^{(i_r)} -1-\tilde N} \, |\det D\delta_{i_1\dots i_N}(\sigma) | \d \sigma\\
&+ c_4 \mu^{-\kappa} \sum_{l=1}^{\tilde N-1} \mu ^{-l}  \int_{\epsilon < |\tau_{i_j}(\sigma) |< 1}   \prod_{j=1}^N |\tau_{i_j}(\sigma)|^{c^{(i_j)} + \sum _{r=1}^j d^{(i_r)} -1-\kappa -l} \, |\det D\delta_{i_1\dots i_N}(\sigma) | \d \sigma\\
& \leq c_5 \mu^{-\tilde N} \prod _{j=1}^N  (-\log \epsilon)^{q_j} \max \Big \{1,  \epsilon^{c^{(i_j)} + \sum _{r=1}^j d^{(i_r)} -\tilde N} \Big \} \\
& +c_6 \sum_{l=1}^{\tilde N -1} \mu^{-\kappa -l} \prod _{j=1}^N  (-\log \epsilon)^{q_{lj}} \max \Big \{1,  \epsilon^{c^{(i_j)} + \sum _{r=1}^j d^{(i_r)} -\kappa -l} \Big \},
\end{align*}
where the exponents $q_j$, $q_{lj}$ can take the values $0$ or $1$. Having in mind that we are interested in the case where $\mu \to +\infty$, we now set $
\epsilon=\mu^{-1/N}$.
Taking into account Lemma \ref{lemma:kappa}, one infers  that the right hand side of the last inequality can be estimated by  a constant times 
\bqn 
 \mu^{-\kappa-1} (\log \mu)^N,
\eqn
so that we finally obtain an asymptotic expansion for $I_{i_1\dots i_N}^{\rho_{i_1}\dots \rho_{i_N}}(\mu)$ by taking into account \eqref{eq:I2}, and the fact that 
\bqn 
(2\pi /\mu)^\kappa \int_{0 < |\tau_{i_j} |< \mu^{-1/N}}  Q_0  \prod_{j=1}^N |\tau_{i_j}|^{c^{(i_j)} + \sum _{r=1}^j d^{(i_r)} -1-\kappa} \d \tau_{i_N} \dots \d \tau_{i_1} = O(\mu^{-\kappa-1}). 
\eqn
\begin{theorem}
\label{thm:9}
Let the assumptions of the first main theorem be fulfilled. Then 
\bqn 
I_{i_1\dots i_N}^{\rho_{i_1}\dots \rho_{i_N}}(\mu)=({2 \pi}/ \mu)^\kappa \mathcal{L}_{i_1\dots i_N}^{\rho_{i_1}\dots \rho_{i_N}}+ O\big (\mu^{-\kappa-1}(\log \mu)^N\big ),
\eqn
where the leading coefficient $\mathcal{L}_{i_1\dots i_N}^{\rho_{i_1}\dots \rho_{i_N}}$ is given by 
\bq
\label{eq:L}
\mathcal{L}_{i_1\dots i_N}^{\rho_{i_1}\dots \rho_{i_N}}=\int_{\Crit( ^{(i_1\dots i_N)} \tilde \Phi^{wk})} \frac { a_{i_1\dots i_N}^{\rho_{i_1}\dots \rho_{i_N}} {\mathcal{J}}_{i_1\dots i_N}^{\rho_{i_1}\dots \rho_{i_N}} \, d\Crit( ^{(i_1\dots i_N)} \tilde \Phi^{wk})} {|\det \mathrm{Hess} ( ^{(i_1\dots i_N)} \tilde \Phi^{wk})_{N\Crit( ^{(i_1\dots i_N)} \tilde \Phi^{wk})}|^{1/2}},
\eq
where $ d\Crit( ^{(i_1\dots i_N)} \tilde \Phi^{wk})$ denotes the induced Riemannian volume density.
\end{theorem}\qed

\section{Statement of the main result}
\label{sec:MR}

We can now state the main result of this paper. But before, we shall say a few words about the desingularization process. Consider  the resolution of $\mathcal{N}$ constructed in Theorem \ref{thm:desing}, and denote the global morphism induced by the local transformations \eqref{eq:40} by $\mathcal{Z}:\tilde X \rightarrow X=T^\ast M \times G$. Consider further  the local ideal $I_\Phi=(\Phi)$ generated by the phase function \eqref{eq:phase}, together with the ideal sheaf  $I_\mathcal{C}\subset \E_X$ of \eqref{eq:crit}. The derivative of $I_\Phi$ is given by $D(I_\Phi) = I_{\mathcal{C}|T^\ast Y\times G}$, 
while  by the implicit function theorem $\mathrm{Sing} \,V_\Phi \subset V_\Phi \cap \Crit(\Phi)=\Crit(\Phi)$, where $V_\Phi$ denotes the vanishing set $V_\Phi$ of $\Phi$. The desingularization process carried out in Section \ref{sec:DP} yields  a partial
monomialization of $I_\Phi$ according to  the diagram 
 \begin{displaymath}
\begin{CD} 
\mathcal{Z}^\ast (I_{\mathcal{C}})\cdot \E_{\tilde x, \tilde X}    & \quad  \supset \quad  & \mathcal{Z}^\ast (I_\Phi) \cdot \E_{\tilde x, \tilde X}  = \prod_j\sigma_j^{l_j}   \cdot\mathcal{Z}^{-1}_\ast(I_\Phi) \cdot \E_{\tilde x, \tilde X} &\quad\ni\quad & \prod_j\sigma_j^{l_j}  \cdot  \,^{(i_1\dots i_N)}\tilde \Phi^ {wk}              \\
@A {\mathcal{Z}^\ast}AA @AA {\mathcal{Z}^\ast}A\\
I_{\mathcal{C}}  &\supset & I_\Phi   & \ni &\Phi   
\end{CD}
\end{displaymath}
where $\tilde x \in \tilde X$. By Theorem \ref{thm:MT1}, $D(\mathcal{Z}^{-1}_\ast(I_\Phi))$ is a resolved ideal sheaf, and Theorem \ref{thm:II}  shows  that  the weak transforms $ \,^{(i_1\dots i_N)}  \tilde \Phi^ {wk}$ have clean critical sets. This allowed us to derive asymptotics for the integrals $I_{i_1\dots i_N}^{\rho_{i_1}\dots \rho_{i_N}}(\mu)$ in Theorem \ref{thm:9}. 
Nevertheless, it is easy to see that $\mathcal{Z}^{-1}_\ast(I_\Phi)$ is not resolved. Furthermore,  the inclusion \eqref{eq:derideal} implies  that 
$
\mathcal{Z}^{-1}_\ast (I_{\mathcal{C}|T^\ast Y\times G}) \subset D(\mathcal{Z}^{-1}_\ast (I_\Phi)).
$
 But since we do not have equality,  this results only in a partial resolution $\tilde{\mathcal{C}}$  of $\mathcal{C}$. In particular, the induced global transform $\mathcal{Z}: \tilde{\mathcal{C}} \rightarrow \mathcal{C}$ is in general not an isomorphism over the smooth locus of $\mathcal{C}$. This is because of the fact that the centers of our monoidal transformations were only chosen over $M \times G$, to keep the phase analysis of the weak transform of $\Phi$ as simple as possible. In turn, the singularities of $\mathcal{C}$ along the fibers of $T^\ast M$ were not completely resolved.
Note that in order to obtain a partial monomialization of $I_\Phi$,   we had to  construct a strong resolution of  $\mathcal{N}$  in $\M=M \times G$, and not just a resolution of the $G$-action in $M$.  
As explained in Section \ref{sec:SPRS}, such a  resolution always exists and is equivalent to a monomialization of the corresponding ideal sheaf.  But in general, it would not be explicit enough to describe the asymptotic behavior of the integrals $I(\mu)$ introduced in \eqref{int}.  {In particular, the so-called numerical data of $\zeta$ are not known a priori, which in our case are given in terms of the dimensions $c^{(i_j)}$ and $d^{(i_j)}$.}
This is the reason why we were forced to construct an explicit resolution  of $\mathcal{N}$, using as centers  isotropy bundles over unions of maximally singular orbits.

Let us now return to our departing point, that is, the asymptotic behavior of the integrals $I(\mu)$, and the proof of Weyl's law for the reduced spectral counting function $N_\chi(\lambda)$.  If $G$ acts on the chart $Y$ only with principal type $G/H_L$, we can directly apply the stationary phase theorem to obtain an expansion for $I(\mu)$. Let us therefore assume that this is not the case. We still have to examine  contributions to $I(\mu)$ coming from integrals of the form
\begin{gather*}
\begin{split}
\tilde I_{i_1\dots i_N}^{\rho_{i_1} \dots \rho_{i_N}}(\mu) = \qquad \qquad \qquad \qquad\qquad \qquad\qquad \qquad\qquad \qquad \\ \int_{M_{i_1}(H_{i_1})\times (-1,1)} \Big [ \int_{\gamma^{(i_{1})}((S_{i_1})_{p^{(i_1)}})_{i_2}(H_{i_2})\times (-1,1) } \dots \Big [  \int_{\gamma^{(i_{N-1})}((S_{i_1\dots i_{N-1}})_{p^{(i_{N-1})}})_{i_{N}}(H_{i_{N}})\times (-1,1)} \\
\Big [ \int_{\gamma^{(i_{N})}((\nu_{i_1\dots i_{N}})_{p^{(i_N)}})\times G_{p^{(i_{N})}}\times S_ \iota( \g_{p^{(i_{N})}}^\perp) \times \cdots \times \stackrel \circ D_ \iota (\g_{p^{(i_{1})}}^\perp)\times \rn }  e^{i\mu {\tau_1 \dots \tau_N} \, ^{(i_1\dots i_N)} \tilde \Phi ^{wk}}  \, a_{i_1\dots i_N}^{\rho_{i_1} \dots \rho_{i_N}} \,   \bar {\mathcal{J}}_{i_1\dots i_N}^{\rho_{i_1} \dots \rho_{i_N}} \\
  \d \xi  \d A^{(i_1)} \dots  \d A^{(i_N)}  \d h^{(i_N)} \d \tilde v^{(i_N)} \Big ]  \d \tau_{i_N} \d p^{(i_{N})} \dots  \Big ] \d \tau_{i_2} \d p^{(i_{2})} \Big ]\d \tau_{i_1} \d p^{(i_{1})}, 
\end{split}
\end{gather*}
where $\mklm{(H_{i_1}), \dots, (H_{i_N})}$ is an arbitrary totally ordered subset of non-principal isotropy types, $S_ \iota( \g_{p^{(i_{N})}}^\perp)$ is the sphere of radius $\iota>0$ in $ \g_{p^{(i_{N})}}^\perp$,
while $a_{i_1\dots i_N}^{\rho_{i_1} \dots \rho_{i_N}}$ is an amplitude which is supposed to have compact support in a system of $(\theta^{(i_1)}, \dots$, $\theta^{(i_{N-1})}, \alpha^{(i_N)})$-charts labeled by the indices $(\rho_{i_1}, \dots, \rho_{i_N})$, and
\begin{align*}
\bar {\mathcal{J}}_{i_1\dots i_N}^{\rho_{i_1} \dots \rho_{i_N}} &=\prod_{j=1}^N |\tau_{i_j}|^{c^{(i_j)}+\sum_r^j d^{(i_r)}-1}{\mathcal{J}}_{i_1\dots i_N}^{\rho_{i_1} \dots \rho_{i_N}},
\end{align*}
 ${\mathcal{J}}_{i_1\dots i_N}^{\rho_{i_1} \dots \rho_{i_N}}$ being a smooth function which does not depend on the variables $\tau_{i_j}$. Now, a computation of the $\xi$-derivatives of $\, ^{(i_1\dots i_N)} \tilde \Phi ^{wk}$ in any of the $\alpha^{(i_N)}$-charts shows that $\, ^{(i_1\dots i_N)} \tilde \Phi ^{wk}$ has no critical points  there. Consequently, repeating the arguments of the previous section, and making use of the  non-stationary phase theorem, see  \cite{hoermanderI}, Theorem 7.7.1, one  computes for large $\tilde N \in \N$ that 
 \begin{align*}
| \tilde I_{i_1 \dots i_N}^{\rho_{i_1} \dots \rho_{i_N}}(\mu)| \leq c_7 \mu^{-\tilde N} \int_{\epsilon<|\tau_{i_j}| < 1} \prod_{j=1} ^N |\tau_{i_j}|^{c^{(i_j)}+\sum_r^j d^{(i_r)}-1-\tilde N}d\tau + c_8 \epsilon^{N (\kappa+1)} \leq c_9 \max\mklm{\mu^{-\tilde N},\mu^{-\kappa-1}},
 \end{align*}
 where we took $\epsilon=\mu^{-1/N}$. Choosing $\tilde N$ large enough, we therefore conclude that
 \bqn
 | \tilde I_{i_1 \dots i_N}^{\rho_{i_1} \dots \rho_{i_N}}(\mu)| =O(\mu^{-\kappa-1}).
 \eqn
  As a consequence of this we see that, up to terms of order $O(\mu^{-\kappa-1})$,  $I(\mu)$ can be written as a sum
\begin{align}
\label{eq:65}
I(\mu)&=\sum _{N=1}^{\Lambda-1}\, \sum_{\stackrel{i_1<\dots< i_N}{ \rho_{i_1}, \dots ,\rho_{i_N}}} I_{i_1\dots i_N}^{\rho_{i_1} \dots \rho_{i_{N}}}( \mu)+\sum_{N=1}^{\Lambda-1} \, \sum_{\stackrel{i_1<\dots< i_{N-1}<L}{ \rho_{i_1}, \dots ,\rho_{i_{N-1}}}} I_{i_1\dots i_{N-1} L}^{\rho_{i_1} \dots \rho_{i_{N-1}}}( \mu),
\end{align}
where the first term is a sum over  maximal, totally ordered subsets of non-principal isotropy types, while the second term is a sum over  totally ordered subsets of non-principal isotropy types. The asymptotic behavior of the integrals $I_{i_1\dots i_N}^{\rho_{i_1}\dots \rho_{i_N}}(\mu)$ has been determined in the previous section, and using Lemma \ref{lemma:Reg} it is not difficult to see that the integrals $I_{i_1\dots i_{N-1} L}^{\rho_{i_1} \dots \rho_{i_{N-1}}}( \mu)$ have analogous asymptotic descriptions. This leads us to the following
\begin{theorem} 
\label{thm:I(mu)}
Let $M$ be a connected, closed Riemannian manifold, and $G$  a compact, connected Lie group $G$ acting isometrically and effectively on $M$. Consider the oscillatory integral
\begin{align}
I(\mu)
&= \int _{T^\ast Y}  \int_{G} e^{i\mu  \Phi(x , \xi,g) }   a( g  x,  x , \xi,g)  \d g \d(T^\ast Y)(  x,  \xi),  \qquad \mu \to +\infty,  
\end{align}
where $(\kappa,Y)$ are local coordinates on $M$,  $ \d(T^\ast Y)(  x,  \xi)$ is the canonical volume density on $T^\ast Y$, and $dg$ the volume density  on $G$ with respect to some left invariant metric on $G$, while $a \in \CT(Y \times T^\ast  Y\times G)$ is an amplitude, and
$
\Phi(x, \xi, g) =\eklm{\kappa(x) - \kappa (g x), \xi }
$.  
Then $I(\mu)$ has the asymptotic expansion 
\bqn 
I(\mu) = (2\pi/\mu)^{\kappa} \mathcal{L}_0 + O\big (\mu^{-\kappa-1}(\log \mu)^{\Lambda-1}\big ),  \qquad \mu \to +\infty.
\eqn
Here $\kappa$ is the dimension of an orbit of principal type in $M$, $\Lambda$ the maximal number of elements of a totally ordered subset of the set of isotropy types, and the leading coefficient is given by 
\bq
\label{eq:L0}
\mathcal{L}_0=\int_{\mathrm{Reg}\, \mathcal{C}} \frac { a( g  x,  x , \xi,g) }{|\det   \, \Phi''(x, \xi,g)_{N_{(x, \xi, g)}\mathrm{Reg}\, \mathcal{C}}|^{1/2}} \d(\mathrm{Reg}\, \mathcal{C})(x,\xi,g),
\eq
where $\mathrm{Reg}\, \mathcal{C}$ denotes the regular part of  $\mathcal{C}=\mklm{(x,\xi,g) \in \Omega \times G: g \cdot (x,\xi) =(x,\xi)}$, and $\d(\mathrm{Reg}\, \mathcal{C})$ the induced volume density. 
In particular, the integral over $\mathrm{Reg}\, \mathcal{C}$ exists.
\end{theorem}
\begin{remark}
Since $M$ is compact, $T^\ast M$ is a paracompact manifold, admitting a Riemannian metric. The restriction of the Riemannian metric on $T^\ast M \times G$ to $\mathrm{Reg}\, \mathcal{C}$ then induces a volume density  $d( \mathrm{Reg}\,  \mathcal{C})$ on  $\mathrm{Reg}\, \mathcal{C}$. 
Note that equation \eqref{eq:L0} in particular means that the obtained asymptotic expansion for $I(\mu)$ is independent of the explicit partial resolution we used. The amplitude $a(gx,x,\xi,g)$ might depend on $\mu$ as in the expression for $O(\mu^{n-2})$ in Theorem \ref{thm:Rt}. But as explained in the proof of Theorem \ref{thm:J}, this has no influence on the final asymptotics.
\end{remark}
\begin{proof} 
Assume that $G$ acts on $Y$ with several orbit types. 
By Theorem \ref{thm:9} and \eqref{eq:65}  one has 
\bqn 
I(\mu) = (2\pi/\mu)^{\kappa} \mathcal{L}_0 + O\big (\mu^{-\kappa-1}(\log \mu)^{\Lambda-1}\big ),  \qquad \mu \to +\infty,
\eqn
where $\mathcal{L}_0$ is given as a sum of integrals of the form \eqref{eq:L}, and similar expressions for the leading terms of the integrals $I_{i_1\dots i_{N-1} L}^{\rho_{i_1} \dots \rho_{i_{N-1}}}( \mu)$. It therefore remains to show the equality \eqref{eq:L0}.  For this, let us  introduce  certain cut-off functions for the singular part $\mathrm{Sing} \, \Omega$ of $\Omega$. Denote the Riemannian distance on   $T^\ast M$ by $|\cdot|$, and   let $K$ be a compact subset in $T^\ast M$, $\epsilon >0$. We then define
 \bqn
 (\mathrm{Sing}\,  \Omega \cap K)_{\epsilon}=\mklm{\eta \in T^\ast M : |\eta-\eta'| < \epsilon \text{ for some } \eta' \in \mathrm{Sing} \, \Omega \cap K}.
 \eqn
By using a partition of unity, one can  show the existence of a  test function $u_\epsilon \in \CT((\mathrm{Sing} \,\Omega \cap K)_{3\epsilon})$ satisfying $u_\epsilon =1$ on $(\mathrm{Sing}\,  \Omega \cap K)_\epsilon$,  see  \cite{hoermanderI},  Theorem 1.4.1. We then have the following
\begin{lemma}\label{lemlim}
Let $a \in \CT(Y \times T^\ast Y \times G)$, $K$ be a compact subset in $T^\ast M$ such that $\supp_{(x,\xi)} a \subset K$, and $u_\epsilon$ as above. Then the  limit 
\bq
\label{eq:MC}
\lim_{\eps \to 0}  \int_{\mathrm{Reg} \, \mathcal{C}}\frac{ a(gx,x,\xi,g) (1-u_\eps)(x, \xi)}{|\det  \, \Phi'' (x,\xi,g)_{|N_{(x, \xi,g )}\mathrm{Reg} \, \mathcal{C}} |^{1/2}} d(\mathrm{Reg}\,  \mathcal{C})(x,\xi,g )
\eq
exists and is equal to $\mathcal{L}_0$.
\end{lemma}
\begin{proof}[Proof of Lemma \ref{lemlim}]
We define
\bqn
I_\eps(\mu)= \int _{T^\ast Y}  \int_{G} e^{i\mu  \Phi(x , \xi,g) }   a( g  x,  x , \xi,g)  (1-u_\epsilon)(x,\xi) \d g \d(T^\ast Y)(  x, \xi).   
\eqn
Since $(x,\xi,g) \in \mathrm{Sing}\,  {\mathcal{C}}$ implies $ (x,\xi) \in \mathrm{Sing}\,  \Omega$, a direct application of the generalized stationary phase theorem for fixed $\eps >0$ gives
\bq
\label{eq:asympt}
| I_\eps(\mu)- (2\pi /\mu)^\kappa \mathcal{L}_0(\eps) | \leq C_\eps \mu^{-\kappa-1},
\eq
where $C_\eps>0$ is a constant depending only on $\eps$, and
\bqn
\mathcal{L}_0(\eps)=  \int_{\mathrm{Reg} \,  {\mathcal{C}}}\frac{ a( g  x,  x , \xi,g)  (1-u_\epsilon)(x,\xi)}{|\det  \, \Phi'' (x, \xi, g)_{|N_{(x, \xi,g )}\mathrm{Reg}\,  {\mathcal{C}}} |^{1/2}} d(\mathrm{Reg} {\mathcal{C}})(x, \xi, g).
\eqn
On the other hand, applying our previous considerations to  $I_\eps(\mu)$ instead of $I(\mu)$, we obtain again an asymptotic expansion of the form \eqref{eq:asympt} for $I_\eps(\mu)$, with $\mu^{-\kappa-1}(\log \mu)^{\Lambda-1}$ instead of $\mu^{-\kappa-1}$, where now the first coefficient is given by a sum of integrals of the form \eqref{eq:L} with $a$ replaced by $a(1-u_\eps)$. Since the first term in the asymptotic expansion \eqref{eq:asympt} is uniquely determined, the two expressions for $\mathcal{L}_0(\eps)$ must be identical. The statement of the lemma now follows by  the Lebesgue theorem on bounded convergence.
\end{proof}
\begin{remark}
Note that the existence of the limit in  \eqref{eq:MC} has been established by  partially resolving the  singularities of the  set $\mathcal{C}$, the corresponding limit being given by $\mathcal{L}_0$.
\end{remark}
\noindent
\emph{End of proof of Theorem \ref{thm:I(mu)}}. 
 Let now $a ^+ \in \CT(Y \times T^\ast Y\times G, \R^+)$. Since one can assume that  $|u_\epsilon| \leq 1$, the lemma of Fatou implies that 
\bqn 
  \int_{\mathrm{Reg}\,  \mathcal{C}} \lim_{\eps \to 0}  \frac{ a^+( g  x,  x , \xi,g)  (1-u_\epsilon)(x,\xi)}{|\det  \, \Phi'' (x, \xi,g )_{|N_{(x, \xi,g)}\mathrm{Reg}\,  \mathcal{C}} |^{1/2}} d(\mathrm{Reg} \, \mathcal{C})(x, \xi,g)
\eqn
is mayorized by the limit  \eqref{eq:MC}, with $a$ replaced by $a^+$. Lemma \ref{lemlim} then implies that
\bqn 
  \int_{\mathrm{Reg}\,  \mathcal{C}} \frac{ a^+ (gx, x, \xi,g)}{|\det  \, \Phi'' (x, \xi,g )_{|N_{(x, \xi,g)}\mathrm{Reg} \, \mathcal{C}} |^{1/2}} d(\mathrm{Reg} \, \mathcal{C})(x, \xi,g ) < \infty.
\eqn
Choosing now $a^+$ to be equal  $1$ on a neighborhood of the support of $a$, and applying the theorem of Lebesgue on bounded convergence to the limit \eqref{eq:MC}, we obtain equation \eqref{eq:L0}. 
\end{proof}
 We shall now state the main result of this paper.
\begin{theorem}\label{thm:main}
 Let $M$ be a compact, connected, $n$-dimensional Riemannian manifold without boundary, and $G$  a compact, connected Lie group, acting effectively and isometrically on $M$. Let further  
 \bqn
 P_0: \Cinft(M) \longrightarrow \L^2(M)
 \eqn
  be an invariant, elliptic, classical pseudodifferential operator of order $m$ on $M$ with principal symbol $p(x,\xi)$, and assume that $P_0$ is positive and symmetric. Denote by $P$ its unique self-adjoint extension, and set 
\bqn
N_\chi(\lambda)=d_\chi \sum _{t \leq \lambda} \mult_\chi(t),
\eqn
 where $\mult_\chi(t)$ stands for the multiplicity of the unitary irreducible representation $\pi_\chi$ corresponding to the character $\chi \in \hat G$ in the eigenspace $E_t$ of $P$ belonging to the eigenvalue $t$. Let $\kappa$ be the dimension of a $G$-orbit of principal type, $\Lambda$ the maximal number of elements of a totally ordered subset of the set of isotropy types, and assume that $n-\kappa \geq 1$. Then \footnote{If $n-\kappa\geq 2$, the error term is slightly better, namely  $O\big (\lambda^{\frac{n-\kappa-1}m} (\log \lambda )^{\Lambda-1} \big )$.}
\bqn 
N_\chi(\lambda)= \frac{d_\chi [{\pi_\chi}_{|H}:1]}{(n-\kappa)(2\pi)^{n-\kappa}}   \mathrm{vol} \, [(\Omega \cap S^\ast M)/G] 
 \,  {\lambda} ^{\frac{n-\kappa}m }   + O\big (\lambda^{\frac{n-\kappa-1}m} (\log \lambda )^\Lambda \big ), \qquad \lambda \to +\infty,
\eqn
where  $d_\chi$ is the dimension of the irreducible representation $\pi_\chi$,  $ [{\pi_\chi}_{|H}:1]$  the multiplicity of the trivial representation in the restriction of $\pi_\chi$ to a principal isotropy group $H$, and $S^\ast M= \mklm{(x,\xi) \in T^\ast M: p(x,\xi) = 1}$, while $\Omega=\mathbb{J}^{-1}(0)$ is the zero level of the momentum map $\mathbb{J}:T^\ast M \rightarrow \g^\ast$ of the underlying Hamiltonian action.
 \end{theorem}
\begin{proof}
Let $\rho \in \CT(-\delta, \delta)$ and $\delta>0$ be sufficiently small. Theorems \ref{thm:Rt} and \ref{thm:I(mu)} together yield  
\begin{gather*}
\hat \sigma_\chi (\rho e^{i(\cdot) \mu})
=  d_\chi \rho(0) \mathcal{L} \,  ({\mu}{/2\pi} )^{n-\kappa -1}  + O\big (\mu^{n-\kappa-2}(\log \mu )^{\Lambda -1} \big ),
\end{gather*}
where 
\bqn
\mathcal{L}=  \lim_{\epsilon \to 0} \sum _\gamma \int_{\mathrm{Reg}\, \mathcal{C} } \frac { u_{\gamma,\epsilon}( x , \xi,g) }{|\det   \, \Phi''_\gamma (x, \xi,g)_{N_{(x, \xi, g)}\mathrm{Reg}\, \mathcal{C} }|^{1/2}} \d(\mathrm{Reg}\, \mathcal{C})(x,\xi,g),
\eqn
and $u_{\gamma,\epsilon}( x , \xi,g) = \overline{\chi(g)} f_\gamma(x)   \Delta_{\epsilon,1}(q( x, \xi))$.
In order to compute $\mathcal{L}$, let us note that for any  smooth, compactly supported function $\alpha$ on $\Omega  \cap T^\ast Y_\gamma$ one has the formula
\bqn
\int_{{\mathrm{Reg}} \, {\mathcal{C}}}\frac{\overline {\chi(g)}  \alpha(x,\xi)}{|\det  \, \Phi_\gamma'' (x,\xi,g)_{|N_{(x, \xi,g)}{\mathrm{Reg}} \, {\mathcal{C}}_\gamma} |^{1/2}} d({\mathrm{Reg}}  \, {\mathcal{C}})(x, \xi,g)
=[{\pi_\chi}_{|H}:1]\int_{{\mathrm{Reg}} \, \Omega} \alpha (x, \xi) \frac{d({\mathrm{Reg}}\, \Omega)(x, \xi)}{\mbox{vol }\mathcal{O}_{(x,\xi)}},
\eqn
where $H$ is a principal isotropy group, compare \cite{cassanas-ramacher09}, Lemma 7. As a consequence of this, we obtain the expression
\begin{align*}
\mathcal{L}&=   [{\pi_\chi}_{|H}:1] \lim_{\epsilon \to 0} \sum _\gamma \int_{{\mathrm{Reg}} \, \Omega}  f_\gamma(x)        \Delta_{\epsilon,1}(q( x, \xi))  \frac{d({\mathrm{Reg}}\, \Omega)(x,\xi)}{\mbox{vol }\mathcal{O}_{(x,\xi)}}\\
&=   [{\pi_\chi}_{|H}:1]   \sum _\gamma \int_{{\mathrm{Reg}} \, \Omega\, \cap \,  S^\ast M}  f_\gamma(x)    \lim_{\epsilon \to 0} \int  \Delta_{\epsilon, 1}(s)  \frac{ s ^{n-\kappa-1} \, ds }{ \mbox{vol}\, \mathcal{O}_{(x,  s \omega)}}  \,    {d({\mathrm{Reg}}\, \Omega \, \cap S^\ast M)(x,\omega)} \\
&=   [{\pi_\chi}_{|H}:1]  \sum _\gamma \int_{{\mathrm{Reg}} \, \Omega\, \cap \,  S^\ast M}  f_\gamma(x)   \,    \frac{d({\mathrm{Reg}}\, \Omega \, \cap S^\ast M)(x,\omega)}{ \mbox{vol}\, \mathcal{O}_{(x, \omega)}} \\
&= [{\pi_\chi}_{|H}:1] \mbox{vol} \, [({\mathrm{Reg}}\, \Omega \cap S^\ast M)/G].
\end{align*}
 Here we took into account that by Proposition \ref{prop:dim}, the set $\mklm{(x,\xi) \in {\mathrm{Reg}}\, \Omega: x \in {\mathrm{Sing}}\, M }$ has measure zero with respect to the induced volume form on $\mathrm{Reg}\, \Omega$, compare \cite{cassanas-ramacher09}, Lemma 3.
 Next, let 
\bqn 
N^Q_\chi(\mu)=d_\chi \sum_{t\leq \mu} \mult ^Q _\chi(t)=\sum _{\mu_j \leq \mu} m^Q_\chi (\mu_j), \qquad m^Q_\chi (\mu_j)=d _\chi  \mult ^Q _\chi (\mu_j)/ \dim E^Q_{\mu_j},
\eqn
 denote the equivariant spectral counting function of $Q= P^{1/m}$. An asymptotic description  for  $N^Q_\chi(\mu)$ can then be deduced from the one of $\hat \sigma_\chi (\rho e^{i(\cdot) \mu})$ by a  classical Tauberian argument \cite{bruening83}. Thus, let $\rho \in \CT(-\delta ,\delta)$ be such that $1=\int \hat \rho(s) \d s =2 \pi \rho(0)$. Then
\bqn 
N^Q_\chi(\mu)=\int_{-\infty} ^{+\infty} N^Q_\chi(\mu -s ) \hat \rho (s) ds + \int _{-\infty} ^{+\infty} [N^Q_\chi(\mu)-N^Q_\chi(\mu -s)] \hat \rho (s) ds= : H_\chi(\mu)+R_\chi(\mu). 
\eqn
$H_\chi(\mu)= \int_{-\infty} ^{+\infty} N^Q_\chi(s) \hat \rho (\mu -s) \, ds$ is a $\Cinft$-function, and by expressing its derivative by a Stieltjes integral one obtains
\begin{align*}
\frac{ dH_\chi}{\d \mu}(\mu) &=\int_{-\infty} ^{+\infty} \frac{\gd} { \gd \mu} \hat \rho(\mu -s) N^Q_\chi(s) \d s = - \int_{-\infty} ^{+\infty} \frac{\gd} { \gd s } \hat \rho(\mu -s)N^Q_\chi(s) \ d s  \\ & = \int_{-\infty} ^{+\infty}  \hat \rho(\mu -s) \d N^Q_\chi(s) =\sum_{j=1}^\infty m^Q_\chi(\mu_j) \hat \rho (\mu-\mu_j) = \hat \sigma_\chi (\check\rho e^{i(\cdot) \mu}),
\end{align*}
where we took into account that,  since $\sigma_\chi \in \S'(\R)$, $N^Q_\chi(\mu)$ is polynomially bounded, and $\check \rho(s) =\rho(-s)$. Now, in addition, let $\rho$ be such that $\hat \rho \geq 0$, and  $\hat \rho(s) \geq c_{10} >0$ for $|s| \leq 1$. Then, for $\mu \in \R$,
\bqn 
N^Q_\chi(\mu +1) -N^Q_\chi(\mu) \leq \sum _{|\mu - \mu_j| \leq 1} m^Q_\chi(\mu_j) \leq \frac 1 {c_{10}} \sum_{j=1}^\infty m^Q_\chi(\mu_j) \hat \rho( \mu-\mu_j) .
\eqn
  From $\hat \sigma_\chi (\check \rho e^{i(\cdot) \mu})=O(\mu^{n-\kappa -1})$  one  then infers that $R_\chi(\mu)=O(\mu^{n-\kappa -1})$ as  $1 \leq \mu \to +\infty$. On the other hand, since $\hat \sigma_\chi (\check \rho e^{i(\cdot) \mu})$ is rapidly decaying as $\mu \to -\infty$, integration gives 
  \bqn 
H_\chi(\mu)= \int_{1}^\mu \hat \sigma_\chi (\check\rho e^{i(\cdot) s}) \d s +c_{11}=\frac {d_\chi \mathcal{L}}{n-\kappa} \, ( \mu/2\pi)^{n-\kappa}   +
\begin{cases}
O\big (\mu^{n-\kappa-1}(\log \mu)^{\Lambda-1} \big ), & n-\kappa \geq 2,\\
O\big ((\log \mu)^{\Lambda} \big ), & n-\kappa=1, 
\end{cases}
 \eqn 
as $1 \leq \mu \to +\infty$, while  $R_\chi(\mu), \, H_\chi(\mu)\to 0$ as ${\mu \to -\infty} $. 
The proof of the theorem is now complete, since by the spectral theorem, $N_\chi(\lambda)=N^Q_\chi(\lambda^{1/m})$. 
\end{proof}

\section{On the spectrum of $\Gamma \setminus G$}
\label{sec:G}

As an application, we shall consider the spectrum of a discrete, uniform subgroup  $\Gamma$ of a connected, semisimple Lie group $G$ with finite center. Thus, let $\theta$ be a Cartan involution of $G$, and $\g=\k\oplus \p$ the  decomposition of the Lie algebra $\g$ of $G$ into the eigenspaces of $\theta$. Let $K$ be the analytic subgroup corresponding to $\k$, which is a maximal compact subgroup of $G$. Since $\Gamma$ is a uniform lattice, $M=\Gamma\setminus G$ is a closed manifold. 
By definition, $\theta$ is an involutive automorphism of $\g$ such that the bilinear form $\eklm{X,Y}_\theta=-\eklm{X, \theta \, Y}$ is strictly positive definite, where $\eklm{X,Y}=\tr (\ad X \circ   Y)$  denotes the Killing form on $\g$. The form $\eklm{\cdot ,\cdot }_\theta$ defines a left-invariant metric on $G$,  and by requiring that the the projection $G \rightarrow M$ is a Riemannian submersion, we obtain a Riemannian structure on $M$.  Since $\Ad(K)$ commutes with $\theta$, and leaves invariant the Killing form, $K$ acts on $G$ and on $M$ from the right  in an isometric  and effective way.
Note that the isotropy group of a point $\Gamma g\in M$ is conjugate 
to the finite group $gKg^{-1}\cap \Gamma$. Hence, all $K$-orbits in $M$ are either principal or exceptional. Since the maximal compact subgroups of $G$ are precisely the conjugates of $K$, exceptional $K$-orbits arise from elements in $\Gamma$ of finite order.   If $\Gamma$ is torsion-free, meaning that no non-trivial element $\gamma \in \Gamma$ is conjugate in $G$ to an element of $K$, there are no exceptional orbits. In this case  the action of $\Gamma $ on $G/K$ is free, and $\Gamma \setminus G/K$ becomes a smooth manifold of dimension $n-d$, where $n=\dim M$, and $d=\dim K$. As an immediate  consequence of  Theorem \ref{thm:main} we now  obtain
\begin{corollary} 
\label{cor:2}
 Let   
 $ P_0: \Cinft(\Gamma\setminus G) \rightarrow \L^2(\Gamma\setminus G)$
  be a $K$-invariant, elliptic, classical pseudodifferential operator of order $m$ on $\Gamma\setminus G$, and assume that $P_0$ is positive and symmetric. Denote by $P$ its unique self-adjoint extension, and let $
N_\chi(\lambda)$ be the reduced spectral counting function of $P$. Then, for each  $ \chi \in \hat K$, 
\bqn 
N_\chi(\lambda)= \frac{d_\chi [{\pi_\chi}_{|H}:1]}{(n-d)(2\pi)^{n-d}}   \mathrm{vol} \, [(\Omega \cap S^\ast (\Gamma \setminus G))/K] 
 \,  {\lambda} ^{\frac{n-d}m }   + O\big (\lambda^{\frac{n-d-1}m} (\log \lambda )^{\Lambda-1} \big ), \qquad \lambda \to +\infty,
\eqn
where $H\subset K$ is a principal isotropy group of the $K$-action on $\Gamma \setminus G$, and $\Lambda$ is bounded by the number of $\Gamma$-conjugacy classes of elements of finite order in $\Gamma$. 
\end{corollary}
\qed

Under the assumption that $\Gamma$ has no torsion, this result was derived previously by Duistermaat--Kolk--Varadarajan \cite{DKV1} for the Laplace--Beltrami operator $\Delta$ on $\L^2(\Gamma\setminus G /\,  K)\simeq \L^2(\Gamma\setminus G)^K$, i.e. in case that $\chi$ corresponds to the trivial representation. They proved this by  studying the spectrum on $\L^2(\Gamma \setminus G/ K)$ of the whole algebra $\D(G/K)$ of $G$-invariant differential operators on $G/K$, which is defined as follows.  Let $G=KAN$ be an Iwasawa decomposition of $G$, $\mathfrak{a}$ the Lie algebra of $A$, and $W$ the Weyl group of $(\g,\mathfrak{a})$. 
Since $\D(G/K)$ is  commutative, there is an orthogonal decomposition of $\L^2(\Gamma \setminus G/K)$  into finite dimensional subspaces of smooth simultaneous eigenfunctions of $\D(G/K)$. Now, each homomorphism from $\D(G/K)$ to $\C$ is precisely of the form $\chi_\mu:\D(G/K)\rightarrow \C$, where $\mu \in \mathfrak{a}^\ast_{\C}/W$. The spectrum $\Lambda(\Gamma)$ of $\D(G/K)$ on $\Gamma \setminus G/K$ is then defined as the set of all $\mu \in \mathfrak{a}^\ast_{\C}/W$ for which there exists a non-zero $\phi \in \Cinft(\Gamma \setminus G /K)$ with $D\phi=\chi_\mu(D) \phi$ for all $D \in \D(G/K)$.  The main result of \cite{DKV1} is  a description of the asymptotic growth of the tempered spectrum $\Lambda(\Gamma)_{temp}=\Lambda(\Gamma) \cap i\mathfrak{a}^\ast$, together with an estimate for the complementary spectrum $\Lambda(\Gamma)\setminus \Lambda(\Gamma)_{temp}$, using the Selberg trace formula, and the Paley--Wiener theorems of Gangolli and Harish-Chandra. From this, Weyl's law for $\Delta$ on $\L^2(\Gamma \setminus G/K)$ follows readily, since the eigenvalue of $\Delta$ corresponding to $\mu \in \Lambda_{temp}(\Gamma)$ is essentially given by $\norm{\mu}^2$. 

Let now  $ P_0: \Cinft(\Gamma\setminus G) \rightarrow \L^2(\Gamma\setminus G)$  satisfy the conditions of Corollary \ref{cor:2}, and in addition assume that it commutes with the right regular representation $R$ of $G$ on $\L^2(\Gamma\setminus G)$. Then each eigenspace of $P$ becomes a unitary $G$-module. Since $\Gamma \setminus G$ is compact, $R$ decomposes into a direct sum of irreducible representations of $G$ according to 
\bqn 
\L^2(\Gamma \setminus G) \simeq \bigoplus_{\omega \in \hat G}  m_\omega \H_\omega,
\eqn
where $\hat G$ denotes the unitary dual of $\hat G$, and $m_\omega$  the multiplicity of $\omega$ in $R$. In the same way, each eigenspace of $P$ decomposes into a direct sum of irreducible $G$-representations. Let  $\mult_\omega(t)$ be the multiplicity of $\omega\in \hat G$ in the eigenspace $E_t$ of $P$ belonging to the eigenvalue $t$, and $[\omega_{|K}:\chi]$ the multiplicity of $\chi \in \hat K$ in the $K$-representation obtained by restricting  $\omega$ to $K$. Then
\bqn 
\mult_\chi(t)= \sum_{\omega \in \hat G} \mult_\omega(t) \, [ \omega_{|K}:\chi]. 
\eqn
Thus, the study of the reduced spectral counting function $N_\chi(\lambda)$ amounts to a  description of the asymptotic multiplicities of those irreducible $G$-representations $\omega \in \hat G$ containing  a certain   $K$-type $\chi\in \hat K$. 
As a consequence of Corollary \ref{cor:2} one now deduces
\begin{theorem}
 Let $ P_0: \Cinft(\Gamma\setminus G) \rightarrow \L^2(\Gamma\setminus G)$ be a $G$-invariant, elliptic, classical pseudodifferential operator of order $m$ on $\Gamma\setminus G$, and assume that $P_0$ is positive and symmetric. Denote by  $\mult_\omega(t)$  the multiplicity of $\omega\in \hat G$ in the eigenspace $E_t$ belonging to the eigenvalue $t$ of the self-adjoint extension $P$ of $P_0$.  Then, for each $\chi \in \hat K$,
\begin{align*}
 \sum_{t\leq \lambda, \, \omega \in \hat G} \mult_\omega(t) \, [ \omega_{|K}:\chi] &= \frac{ [{\pi_\chi}_{|H}:1]}{(n-d)(2\pi)^{n-d}}   \mathrm{vol} \, [(\Omega \cap S^\ast (\Gamma \setminus G))/K] 
 \,  {\lambda} ^{\frac{n-d}m }   \\&+ O\big (\lambda^{\frac{n-d-1}m} (\log \lambda )^{\Lambda-1} \big ), \qquad \lambda \to +\infty,
\end{align*}
where $n=\dim \Gamma \setminus G$, $d=\dim K$, and  $\Lambda$ is bounded by the number of $\Gamma$-conjugacy classes of elements of finite order in $\Gamma$. 
\end{theorem}
\qed


\begin{thebibliography}{10}

\bibitem{albin-melrose}
P.~Albin and R.~Melrose, \emph{Equivariant cohomology and resolution}, arXiv
  Preprint 0907.3211, 2009.

\bibitem{AGV88}
V.~I. Arnold, S.~M. Gusein-Zade, and A.~N. Varchenko, \emph{Singularities of
  differentiable maps}, vol. I, II, Birkh\"{a}user, Boston, Basel, Berlin,
  1988.

\bibitem{atiyah70}
M.~F. Atiyah, \emph{Resolution of singularities and division of distributions},
  Comm. Pure Appl. Math. \textbf{23} (1970), 145--150.

\bibitem{avacumovic}
V.G. Avacumovi\v{c}, \emph{{\"{U}ber die Eigenfunktionen auf geschlossenen
  Riemannschen Mannigfaltikeiten}}, Math. Z. \textbf{65} (1956), 327--344.

\bibitem{bernshtein71}
I.~N. Bernshtein, \emph{{Modules over a ring of differential operators. Study
  of the fundamental solutions of equations with constant coeficients}}, Funct.
  Anal. and its Appl. \textbf{5} (1971), no.~2, 89--101.

\bibitem{bernshtein-gelfand69}
I.N. Bernshtein and S.~I. Gel'land, \emph{{The meromorphic behavior of the
  function $P^\lambda$}}, Funkt. Anal. and its Appl. \textbf{3} (1969), no.~1,
  84--85.

\bibitem{bott56}
R.~Bott, \emph{On the iteration of closed geodesics and the {Sturm}
  intersection theory}, Comm. Pure Appl. Math. \textbf{9} (1956), 171--206.

\bibitem{bredon}
G.E. Bredon, \emph{Introduction to compact transformation groups}, Academic
  Press, New York, 1972, Pure and Applied Mathematics, Vol. 46.

\bibitem{bruening83}
J.~Br\"{u}ning, \emph{{Zur Eigenwertverteilung invarianter elliptischer
  Operatoren}}, J. Reine Angew. Math. \textbf{339} (1983), 82--96.

\bibitem{bruening-heintze79}
J.~Br\"uning and E.~Heintze, \emph{Representations of compact {Lie groups} and
  elliptic operators}, Inv. math. \textbf{50} (1979), 169--203.

\bibitem{carleman}
T.~Carleman, \emph{Propri\'{e}t\'{e}s asymptotiques des fonctions fondamentales
  des membranes vibrantes}, C. R. S\'{e}me Cong. Math. Scand. Stockholm, 1934,
  Lund, 1935, pp.~34--44.

\bibitem{cassanas-ramacher09}
R.~Cassanas and P.~Ramacher, \emph{Reduced {Weyl} asymptotics for
  pseudodifferential operators on bounded domains {II}. {T}he compact group
  case}, J. Funct. Anal. \textbf{256} (2009), 91--128.

\bibitem{combescure-ralston-robert}
M.~Combescure, J.~Ralston, and D.~Robert, \emph{A proof of the {G}utzwiller
  semiclassical trace formula using coherent states decomposition}, Comm. Math.
  Phys. \textbf{202} (1999), 463--480.

\bibitem{donnelly78}
H.~Donnelly, \emph{{G-spaces, the asymptotic splitting of $L^2(M)$ into
  irreducibles}}, Math. Ann. \textbf{237} (1978), 23--40.

\bibitem{duistermaat74}
J.~J. Duistermaat, \emph{{Oscillatory integrals, Lagrange immersions and
  unfolding of singularities}}, Comm. Pure Appl. Math. \textbf{27} (1974),
  207--281.

\bibitem{duistermaat-guillemin75}
J.~J. Duistermaat and V.W. Guillemin, \emph{The spectrum of positive elliptic
  operators and periodic bicharacteristics}, Inv. Math. \textbf{29} (1975),
  no.~3, 39--79.

\bibitem{duistermaat-kolk}
J.~J. Duistermaat and J.~A. Kolk, \emph{Lie groups}, Springer-Verlag, Berlin,
  Heidelberg, New York, 1999.

\bibitem{DKV1}
J.~J. Duistermaat, J.~A.~C. Kolk, and V.~S. Varadarajan, \emph{Spectra of
  compact locally symmetric manifolds of negative curvature}, Inv. Math.
  \textbf{52} (1979), 27--93.

\bibitem{helffer-elhouakmi91}
Z.~El~Houakmi and B.~Helffer, \emph{Comportement semi-classique en pr\'esence
  de sym\'etries: action d'un groupe de {L}ie compact}, Asymptotic Anal.
  \textbf{5} (1991), no.~2, 91--113.

\bibitem{garding}
L.~G{\aa}rding, \emph{On the asymptotic distribution of eigenvalues and
  eigenfunctions of elliptic differential operators}, Math. Scand. \textbf{1}
  (1953), 237--255.

\bibitem{grigis-sjoestrand}
A.~Grigis and J.~Sj{\"o}strand, \emph{Microlocal analysis for differential
  operators}, London Mathematical Society Lecture Note Series, vol. 196,
  Cambridge University Press, 1994.

\bibitem{guillemin-uribe90}
V.~Guillemin and A.~Uribe, \emph{Reduction and the trace formula}, J. Diff.
  Geom. \textbf{32} (1990), no.~2, 315--347.

\bibitem{helffer-robert84}
B.~Helffer and D.~Robert, \emph{Etude du spectre pour un op\'eratour
  globalement elliptique dont le symbole de {Weyl} pr\'esente des sym\'etries
  {I}. {A}ction des groupes finis}, Amer. J. Math. \textbf{106} (1984),
  1199--1236.

\bibitem{helffer-robert86}
\bysame, \emph{Etude du spectre pour un op\'eratour globalement elliptique dont
  le symbole de {Weyl} pr\'esente des sym\'etries {II}. {A}ction des groupes de
  {L}ie compacts}, Amer. J. Math. \textbf{108} (1986), 973--1000.

\bibitem{helgason78}
S.~Helgason, \emph{Differential geometry, {Lie} groups, and symmetric spaces},
  American Mathematical Society, Providence Rhode Island, 2001.

\bibitem{hironaka}
H.~Hironaka, \emph{Resolution of singularities of an algebraic variety over a
  field of characteristic zero}, Ann. of Math. \textbf{79} (1964), 109--326.

\bibitem{hoermander68}
L.~H\"ormander, \emph{The spectral function of an elliptic operator}, Acta
  Math. \textbf{121} (1968), 193--218.

\bibitem{hoermanderI}
L.~H{\"{o}}rmander, \emph{The analysis of linear partial differential
  operators}, vol.~I, Springer--Verlag, Berlin, Heidelberg, New York, 1983.

\bibitem{hoermanderIV}
\bysame, \emph{The analysis of linear partial differential operators}, vol.~IV,
  Springer--Verlag, Berlin, Heidelberg, New York, 1985.

\bibitem{kawakubo}
K.~Kawakubo, \emph{The theory of transformation groups}, The Clarendon Press
  Oxford University Press, New York, 1991.

\bibitem{kollar}
J.~K\'{o}llar, \emph{Lectures on resolution of singularities}, Annals of
  Mathematical Studies, vol. 166, Princeton University Press, Princeton, New
  Jersey, 2007.

\bibitem{malgrange74}
B.~Malgrange, \emph{Int\'{e}grales asymptotiques et monodromie}, Ann. Sci. ENS
  \textbf{7} (1974), 405--430.

\bibitem{meinrenken-sjamaar}
E.~Meinrenken and R.~Sjamaar, \emph{Singular reduction and quantization},
  Topology \textbf{38} (1999), no.~4, 699--762.

\bibitem{milnor63}
J.~W. Milnor, \emph{Morse theory}, Annals of Mathematical Studies, vol.~51,
  Princeton University Press, Princeton, New Jersey, 1963.

\bibitem{minakshisundaram-pleijel}
S.~Minakshisundaram and {\AA}.~Pleijel, \emph{Some properties of the
  eigenfunctions of the {Laplace} operator on {Riemannian} manifolds}, Canad.
  J. Math. \textbf{1} (1949), 242--256.

\bibitem{ortega-ratiu}
J.P. Ortega and T.S. Ratiu, \emph{Momentum maps and {H}amiltonian reduction},
  Progress in Mathematics, vol. 222, Birkh\"auser Boston Inc., Boston, MA,
  2004.

\bibitem{ramacher08}
P.~Ramacher, \emph{Reduced {Weyl} asymptotics for pseudodifferential operators
  on bounded domains {I}. {T}he finite group case}, J. Funct. Anal.
  \textbf{255} (2008), 777--818.

\bibitem{shubin}
M.~A. Shubin, \emph{Pseudodifferential operators and spectral theory}, 2nd
  edition, Springer--Verlag, Berlin, Heidelberg, New York, 2001.

\bibitem{lerman-sjamaar}
R.~Sjamaar and E.~Lerman, \emph{Stratified symplectic spaces and reduction},
  Ann. of Math. \textbf{134} (1991), 375--422.

\bibitem{sternberg}
S.~Sternberg, \emph{Lectures on differential geometry}, Prentice-Hall, Inc.,
  Englewood Cliffs, New Jersey, 1965.

\bibitem{varadarajan97}
V.~S. Varadarajan, \emph{The method of stationary phase and applications to
  geometry and analysis on {Lie} groups}, Algebraic and analytic methods in
  representation theory, Persp. Math., vol.~17, Academic Press, 1997,
  pp.~167--242.

\bibitem{weyl}
H.~Weyl, \emph{{Das asymptotische Verteilungsgesetz der Eigenwerte linearer
  partieller Differentialgleichungen (mit einer Anwendung auf die Theorie der
  Hohlraumstrahlung)}}, Math. Ann. \textbf{71} (1912), 441--479.

\end{thebibliography}

\providecommand{\bysame}{\leavevmode\hbox to3em{\hrulefill}\thinspace}
\providecommand{\MR}{\relax\ifhmode\unskip\space\fi MR }
\providecommand{\MRhref}[2]{%
  \href{http://www.ams.org/mathscinet-getitem?mr=#1}{#2}
}
\providecommand{\href}[2]{#2}

\end{document}